\documentclass[11pt, reqno]{amsart}
\usepackage{amssymb}
\usepackage{amsmath}
\usepackage{amsthm}
\usepackage{braket}
\usepackage[dvipsnames]{xcolor}
\usepackage{pdfpages}
\usepackage{graphicx}
\usepackage{caption}
\usepackage{subcaption}
\usepackage{mathrsfs} 
\usepackage{tikz-cd} 
\usepackage{bm}
\usepackage{enumitem}
\usepackage{mathtools}
\usepackage{pgfplots}
\usepackage{import}
\usepackage{xifthen}
\usepackage{transparent}
\usepackage{cancel}
\usepackage{tikz}
\usepackage[utf8]{inputenc}
\usetikzlibrary{arrows,automata}
\usetikzlibrary{positioning}
\usepackage{mathptmx}
\usetikzlibrary{calc}
\usepackage{scrextend}

\usepackage[utf8]{inputenc}
\usepackage[english]{babel}

\usepackage{hyperref}

\tikzset{
    state/.style={
           rectangle,
           rounded corners,
           draw=black, very thick,
           minimum height=2em,
           inner sep=2pt,
           text centered,
           },
}

\usepackage{lmodern,babel,adjustbox,booktabs,multirow}

\numberwithin{equation}{section}
\theoremstyle{plain}
\newtheorem{theorem}{Theorem}[section]

\theoremstyle{theorem}
\newtheorem{prop}[theorem]{Proposition}
\newtheorem{lem}[theorem]{Lemma}
\newtheorem{cor}[theorem]{Corollary}

\newtheorem*{question*}{Question}

\theoremstyle{plain}
\newtheorem{ih}{IH}

\theoremstyle{definition}
\newtheorem{defn}[theorem]{Definition}
\newtheorem{rmk}[theorem]{Remark}

\newcommand{\R}{\mathbb{R}}
\newcommand{\C}{\mathbb{C}}

\newcommand{\Proj}{\mathbb{P}}

\newcommand{\Z}{\mathbb{Z}}
\newcommand{\Hyp}{\mathbb{H}}

\newcommand{\D}{\mathbb{D}}

\DeclareMathOperator{\Rat}{Rat}

\DeclareMathOperator{\QC}{QC}

\DeclareMathOperator{\val}{val}

\DeclareMathOperator{\BP}{\mathcal{B}}

\DeclareMathOperator{\Mod}{Mod}
\DeclareMathOperator{\PSL}{PSL}

\DeclareMathOperator{\MCG}{Mod}
\DeclareMathOperator{\Isom}{Isom}

\DeclareMathOperator{\Res}{Res}

\DeclareMathOperator{\Aut}{Aut}
\DeclareMathOperator{\Int}{Int}
\DeclareMathOperator{\chull}{Cvx\, Hull}
\DeclareMathOperator{\hull}{hull}
\DeclareMathOperator{\diam}{diam}

\numberwithin{figure}{section}

\title[Kleinian Reflection Groups and Anti-rational Maps]{On Deformation Space Analogies between Kleinian Reflection Groups and Antiholomorphic Rational Maps}
\begin{document}

\begin{author}[R.~Lodge]{Russell Lodge}
\address{Department of Mathematics and Computer Science, Indiana State University, Terre Haute, IN 47809, USA}
\email{russell.lodge@indstate.edu}
\end{author}

\begin{author}[Y.~Luo]{Yusheng Luo}
\address{Institute for Mathematical Sciences, Stony Brook University, 100 Nicolls Rd, Stony Brook, NY 11794-3660, USA}
\email{yusheng.s.luo@gmail.com}
\end{author}

\begin{author}[S.~Mukherjee]{Sabyasachi Mukherjee}
\address{School of Mathematics, Tata Institute of Fundamental Research, 1 Homi Bhabha Road, Mumbai 400005, India}
\email{sabya@math.tifr.res.in}
\thanks{S. M. was supported by the Department of Atomic Energy, Government of India, under project no.12-R\&D-TFR-5.01-0500, an endowment of the Infosys Foundation, and SERB research project grant SRG/2020/000018.}
\end{author}

\begin{abstract}
In a previous paper \cite{LLM22}, we constructed an explicit dynamical correspondence between certain Kleinian reflection groups and certain anti-holomorphic rational maps on the Riemann sphere.
In this paper, we show that their deformation spaces share many striking similarities.
We establish an analogue of Thurston's compactness theorem for critically fixed anti-rational maps.
We also characterize how deformation spaces interact with each other and study the monodromy representations of the union of all deformation spaces.
\end{abstract}

\date{\today}


\maketitle

\setcounter{tocdepth}{1}
\tableofcontents

\section{Introduction}
Hyperbolic maps play an essential role in complex dynamics.
They form an open and conjecturally dense subset in the moduli spaces.
In the closely related field of hyperbolic $3$-manifolds, quasiconformal deformation spaces of Kleinian groups play a role analogous to that of hyperbolic components. The celebrated Sullivan dictionary between Kleinian groups and rational dynamics, which grew out Sullivan's introduction of quasiconformal methods in complex dynamics \cite{Sullivan85}, raises several important questions about possible parallels and mismatch between these two classes of parameter spaces.

In the recent work \cite{LLM22}, the authors constructed an explicit dynamically meaningful correspondence between \emph{critically fixed anti-holomorphic rational maps} and \emph{kissing Kleinian reflection groups} acting on the Riemann sphere.  The aim of the current paper is to investigate some fundamental parameter space implications of the above correspondence; specifically, to compare boundedness properties and mutual interaction structure of critically fixed hyperbolic components and deformation spaces of kissing reflection groups. Our study reveals a striking similarity between the two parameter spaces. In particular, we will show
\begin{itemize}
\item (Boundedness of deformation spaces.) The {\em pared deformation space} of a critically fixed anti-rational map is bounded if and only if the pared deformation space of the corresponding kissing reflection group is bounded (see Theorem \ref{thm:ab}).

\item (Interaction between deformation spaces.) The pared deformation space of a critically fixed anti-rational map bifurcates to another if and only if the pared deformation space of the corresponding kissing reflection group degenerates to the other (see Theorem \ref{thm:main}).

\item (Global topology.) There exists a surjective monodromy representation on the closure of the union of pared deformation spaces of all critically fixed anti-rational maps of degree $d$ to the mapping class group of the $d+1$-punctured sphere, which is analogous to a result of Hatcher and Thurston for kissing reflection groups (see Theorem \ref{thm:mono}).
\end{itemize}

Our investigation also yields some new phenomenon on how different deformation spaces can interact (see Appendices \ref{sec:pb} and \ref{sec:ve}).

Before we proceed to give the precise statements of our principal results, let us summarize the dynamical correspondence established in \cite{LLM22}.

\subsection*{Dictionary between kissing reflection groups and critically fixed anti-rational maps}

Let $\Rat_d^-(\C)$ be the space of all degree $d$ anti-holomorphic rational maps (anti-rational maps for short) on $\hat \C$, and let 
$$
\mathcal{M}_d^- = \Rat_d^-(\C) /\PSL_2(\C)
$$ 
be the moduli space of degree $d$ anti-rational maps. We will, when there is no possibility of confusion, refer to the elements of $\mathcal{M}_d^-$ (which are, formally speaking, M{\"o}bius conjugacy classes) as maps.
An anti-rational map $R$ is {\em hyperbolic} if all of its critical points converge to attracting cycles under iteration.
The set of all conjugacy classes of hyperbolic maps form an open subset of $\mathcal{M}_d^-$, and a connected component is called a {\em hyperbolic component}.

An anti-rational map is called critically fixed if it fixes each of its critical points. Let $\mathcal{R} \in \Rat_d^-(\C)$ be a degree $d$ critically fixed anti-rational map.
The union $\mathscr{T}$ of all {\em fixed internal rays} in the invariant Fatou components (each of which necessarily contains a fixed critical point of $\mathcal{R}$) is called the {\em Tischler graph} (see also \S \ref{subsec:pds}).

A graph is called {\em $k$-connected} if it contains more than $k$ vertices and remains connected with any $k-1$ vertices removed.
In \cite{LLM22}, we showed that the planar dual $\Gamma$ of $\mathscr{T}$ is {\em $2$-connected} and {\em simple} with $d+1$ vertices.
Conversely, any $2$-connected, simple, plane graph $\Gamma$ with $d+1$ vertices is isomorphic to the planar dual of the Tischler graph of some degree $d$ critically fixed anti-rational map $\mathcal{R}_\Gamma$, which is unique up to M{\"o}bius conjugation (cf. \cite{Geyer}).
We use $\mathcal{H}_{\Gamma}$ to denote the hyperbolic component containing $[\mathcal{R}_\Gamma]$.

A \emph{circle packing} is a connected collection of (oriented) circles in $\widehat{\C}$ with disjoint interiors. Note that the collection of circles may be infinite in this definition. Let $\mathcal{P}$ be a finite circle packing in $\widehat{\C}$.
A kissing reflection group $G$ is a group generated by reflections along the circles of $\mathcal{P}$.
The {\em contact graph} $\Gamma$ of $\mathcal{P}$ associates a vertex to each circle in $\mathcal{P}$, and two vertices are connected by an edge if the corresponding two circles touch.

In consistence with the convention adopted in \cite{LLM22}, a plane graph will mean a planar graph together with an embedding in $\widehat{\C}$ throughout this paper. We say that two plane graphs are isomorphic if the underlying graph isomorphism is induced by an orientation preserving homeomorphism of $\widehat{\C}$.
Two kissing reflection groups with isomorphic contact graphs are quasiconformally conjugate.
Thus, by the circle packing theorem, the quasiconformal conjugacy classes of kissing reflection groups are in correspondence with isomorphism classes of simple plane graphs.
It can be verified that the limit set of $G = G_\Gamma$ is connected if and only if the corresponding contact graph $\Gamma$ is $2$-connected (see \cite[\S 3]{LLM22}).

According to \cite[Theorem~1.1]{LLM22}, there exists a bijective correspondence between kissing reflection groups with connected limit set and critically fixed anti-rational maps such that the limit set of $G_\Gamma$ and the Julia set of $\mathcal{R}_\Gamma$ are homeomorphic via a dynamically natural map (see Figure \ref{fig:DynamicalCorrespondence}). This is the same equivariant homeomorphism that was used in \cite{LLMM19} to classify Julia set quasisymmetries in the context of critically fixed anti-rational maps whose Tischler graph is a triangulation.

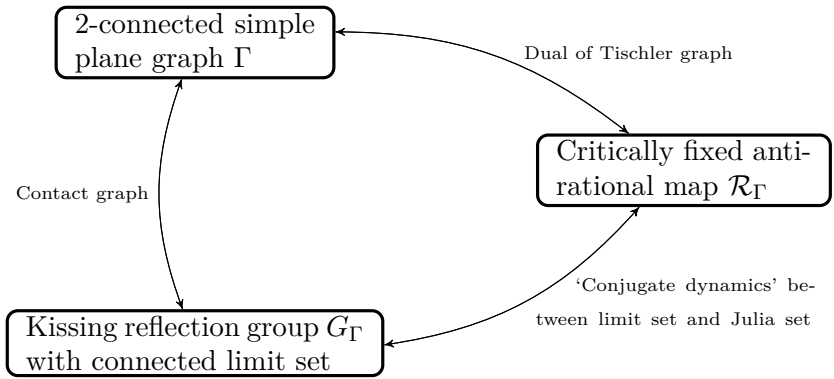
\begin{figure}
\centering
\begin{tikzpicture}[->,>=stealth']
 
 \node[state] (Graph) 
 {\begin{tabular}{l}
 \parbox{3.2cm}{
  {2-connected simple plane graph $\Gamma$}}\\
 \end{tabular}};

 \node[state,    
   yshift=-4cm, 		
  anchor=center] (RG) 	
 {
 \begin{tabular}{l} 	
  \parbox{4.5cm}{
  {Kissing reflection group $G_{\Gamma}$ with connected limit set}}\\
 \end{tabular}
 };

 \node[state,    	
  yshift=-1.7cm,
  right of=Graph, 	
  node distance=6.5cm, 	
  anchor=center] (AR) 	
 {
 \begin{tabular}{l} 
  \parbox{3.4cm}{
  {Critically fixed anti-rational map $\mathcal{R}_\Gamma$}
  }
 \end{tabular}
 };
 
  \path (Graph) edge[bend left=19] node[anchor=north,right]{\tiny \hspace{2mm} Dual of Tischler graph } (AR)
  (AR) edge[bend right=19] (Graph)
  (Graph) edge[bend right=20] node[anchor=south,left]{\tiny Contact graph } (RG)
  (RG) edge[bend left=20] (Graph)
  (RG) edge[bend right=20] node[anchor=north,right, text width=4cm]{\tiny \hspace{4mm} `Conjugate dynamics' between limit set and Julia set} (AR)
  (AR) edge[bend left=20] (RG);
  
  \end{tikzpicture}
\caption{Dynamical correspondence between kissing reflection groups and critically fixed anti-rational maps.}
\label{fig:DynamicalCorrespondence}
\end{figure}

As consequences of the above dictionary, it was shown in the same paper that various additional combinatorial structures of the simple plane graph $\Gamma$ translate into special dynamical/geometric properties of the corresponding kissing reflection group $G_\Gamma$ and the associated $3$-manifold (respectively, of the corresponding critically fixed anti-rational map $\mathcal{R}_\Gamma$). 
Of particular importance from the perspective of the current paper is the combinatorial property of being \emph{polyhedral}: a graph is $\Gamma$ is said to be polyhedral if $\Gamma$ is the $1$-skeleton of a convex polyhedron. Note that this is equivalent to $\Gamma$ being $3$-connected. We say that a closed set $\Lambda$ is a {\em round gasket} if
\begin{itemize}
\item $\Lambda$ is the closure of some infinite circle packing; and
\item the complement of $\Lambda$ is a union of round disks which is dense in $\widehat{\C}$.
\end{itemize}
We will call a homeomorphic copy of a round gasket a {\em gasket} (following the terminology which is used in \cite{LLM22} but different from that in \cite{LLMM19}).
By \cite[Theorems~1.4,1.5]{LLM22} and \cite{Thurston86}, 
\begin{align}\label{eqn:aclindrical}
\text{$\Gamma$ is polyhedral} &\iff \text{the limit set of $G_\Gamma$ is a gasket} \nonumber\\
&\iff \text{the Julia set of $\mathcal{R}_\Gamma$ is a gasket}\\
&\iff \text{the $3$-manifold for $G_\Gamma$ is \emph{acylindrical}}. \nonumber \\
&\iff \text{the \emph{pared deformation space} of $G_\Gamma$ is bounded} \nonumber.
\end{align}

\subsection*{Thurston's compactness theorem and boundedness of deformation spaces}\label{subsec:tct}
Acylindrical manifolds are of central importance in three dimension geometry and topology.
Relevant to our discussion, Thurston proved that the deformation space of an acylindrical $3$-manifold is bounded \cite{Thurston86}, which is a key step in Thurston's hyperbolization theorem for Haken manifolds.

In the convex cocompact case, it is well known that a hyperbolic $3$-manifold with non-empty boundary is acylindrical if and only if its limit is a Sierpinski carpet (see \cite{McM95} or \cite[Theorem ~1.8]{Pilgrim03}).
Based on this analogy, McMullen conjectured that the hyperbolic component of a rational map with Sierpinski carpet Julia set is bounded (see \cite[Question~5.3]{McM95}).
This conjecture, which stimulated a lot of work towards understanding the topology of hyperbolic components, still remains wide open. 

For Kleinian groups uniformizing acylindrical $3$-manifolds with cusps, 
Thurston's compactness theorem holds only for deformation spaces of \emph{pared manifolds} \cite{Thurston86}.
In such pared deformation spaces, the parabolic elements are preserved.

In light of \cite[Theorems~1.4,1.5]{LLM22} (see Characterization~\ref{eqn:aclindrical}) and the above discussion, it is natural to ask whether an analogue of Thurston's compactness theorem holds for the corresponding critically fixed anti-rational maps.

In our setting, the cusps of the $3$-manifold associated with a kissing reflection group $G_\Gamma$ are in bijective correspondence with repelling fixed points of $\mathcal{R}_\Gamma$ (cf. \cite[\S 4.3]{LLM22}).
Thus, preserving parabolics of $G_\Gamma$ in the group setting is analogous to controlling the multipliers of all repelling fixed points of $\mathcal{R}_\Gamma$.

This motivates the definition of the {\em pared deformation space} for $\mathcal{R}_\Gamma$.
We first note that the dynamics on any critically fixed Fatou component is conjugated to the dynamics of an anti-Blaschke product on the unit disk $\D$, which we call the {\em uniformizing model}.

Let $K > 0$. We define the {\em pared deformation space} $\mathcal{H}_\Gamma(K) \subseteq \mathcal{H}_\Gamma$ of $\mathcal{R}_\Gamma$ to be the connected component containing $\mathcal{R}_\Gamma$ of the set of anti-rational maps $R\in \mathcal{H}_\Gamma$ where the multiplier of any repelling fixed point in any uniformizing model is bounded by $K$ (see \S \ref{subsec:pds} and the discussion in \S \ref{subsec:dsthc} for further motivations and necessities of this definition).

We remark that by definition, $\mathcal{H}_\Gamma(K) \subseteq \mathcal{H}_\Gamma(K')$ if $K \leq K'$ and 
$$
\mathcal{H}_\Gamma = \bigcup_K \mathcal{H}_\Gamma(K).
$$
We also remark that in most situations, $\mathcal{H}_\Gamma(K)$ is not compact in $\mathcal{H}_\Gamma$ for all sufficiently large $K$ (see Theorem \ref{thm:cparm}).

We prove the following analogue of Thurston's compactness theorem and McMullen's conjecture in our setting.
\begin{theorem}\label{thm:ab}
Let $\Gamma$ be a simple, $2$-connected, plane graph, and $\mathcal{R}_\Gamma$ be a critically fixed anti-rational map with Tischler graph isomorphic (as plane graphs) to the planar dual of $\Gamma$.
Then the pared deformation space $\mathcal{H}_\Gamma(K)$ is bounded if and only if $\Gamma$ is polyhedral.

More precisely, if $\Gamma$ is polyhedral, then for any $K$, the space $\mathcal{H}_\Gamma(K)$ is bounded in $\mathcal{M}^-_d$.
Conversely, if $\Gamma$ is not polyhedral, then there exists some $K$ such that $\mathcal{H}_\Gamma(K)$ is unbounded in $\mathcal{M}^-_d$.
\end{theorem}
We remark that the hyperbolic component of a critically fixed anti-rational map is never bounded (see Proposition \ref{prop:ub}).

We also remark that in light of Characterization~\ref{eqn:aclindrical}, Theorem \ref{thm:ab} can be rephrased as follows: the pared deformation space of $\mathcal{R}_\Gamma$ is bounded if and only if the pared deformation space of the corresponding kissing reflection group $G_\Gamma$ is bounded.
This is one of the many instances of the similarities shared by the two deformation spaces.

\subsection*{Interactions between deformation spaces}
While individual hyperbolic components as well as individual deformation spaces of Kleinian groups are well understood in various settings, the closure/boundary of such open sets are far more complicated.

Let $\Gamma, \Gamma'$ be a pair of $2$-connected, simple, plane graphs with the same number of vertices.
We say that $\Gamma'$ {\em dominates} $\Gamma$, denoted by $\Gamma' \geq \Gamma$, if there exists an embedding $i: \Gamma \xhookrightarrow{} \Gamma'$ as plane graphs sending vertices to vertices.
We use $\Gamma' > \Gamma$ to mean $\Gamma' \geq \Gamma$ and $\Gamma' \neq \Gamma$.

Thanks to Thurston's hyperbolization theorem, the deformation space $\mathcal{QC}_\Gamma$ for $G_\Gamma$ satisfies
\begin{align}\label{eqn:ic}
\mathcal{QC}_{\Gamma'} \subsetneqq \overline{\mathcal{QC}_\Gamma} \text{  if and only if } \Gamma' > \Gamma. 
\end{align}
We refer the readers to \cite[\S 3]{LLM22} for a more precise statement.

An anti-rational map $R$ is said to be {\em parabolic} if every critical point of $R$ lies in the Fatou set and $R$ contains at least one parabolic cycle. 
A parabolic map $[R] \in \partial \mathcal{H}_{\Gamma}$ is called a {\em root} of $\mathcal{H}_{\Gamma}$ if its Julia dynamics is topologically conjugate to that of maps in $\mathcal{H}_{\Gamma}$.
We say $\mathcal{H}_\Gamma(K)$ {\em parabolic bifurcates} to $\mathcal{H}_{\Gamma'}(K)$ if $\Gamma \neq \Gamma'$ and the intersection $\partial \mathcal{H}_{\Gamma}(K) \cap \partial \mathcal{H}_{\Gamma'}(K)$ contains a root of $\mathcal{H}_{\Gamma'}$.
We show that the pared deformation spaces for critically fixed anti-rational maps have a similar intersecting structure.

\begin{theorem}\label{thm:main}
Let $\Gamma, \Gamma'$ be a pair of $2$-connected, simple, plane graphs with the same number of vertices. For all large $K$, the pared deformation space $\mathcal{H}_\Gamma(K)$ parabolic bifurcates to $\mathcal{H}_{\Gamma'}(K)$ if and only if $\Gamma' > \Gamma$.
\end{theorem}

It should be mentioned that in the group setting, the deformation spaces $\mathcal{QC}_\Gamma$ may have different dimensions for different graphs $\Gamma$ with the same number of vertices.
In the case $\Gamma' > \Gamma$, the whole deformation space $\mathcal{QC}_{\Gamma'}$ is contained in the closure $\overline{\mathcal{QC}_\Gamma}$, giving a `strata' structure to the space of all kissing reflection groups of a given rank. 

This is in stark contrast with the pared deformation spaces $\mathcal{H}_\Gamma(K)$, where all hyperbolic components corresponding to graphs $\Gamma$ with $d+1$ vertices have the same complex dimension $2d-2$.

On the other hand, the boundary $\partial \mathcal{H}_{\Gamma}(K)$ consists of roots of $\mathcal{H}_{\Gamma'}$ for $\Gamma' \geq \Gamma$ (see Corollary \ref{cor:br}).
Conjecturally, these roots form different dimensional subsets of $\partial \mathcal{H}_{\Gamma}(K)$, and a strata structure emerges.
The precise intersection structure between the closure of roots of different hyperbolic components remains mysterious, and is worth further investigations.

Note that $\geq$ gives a partial order on the space of all $2$-connected simple plane graphs.
It is easy to check that the space is connected under this partial order $\geq$.
Thus, we have
\begin{cor}
For all large $K$, the union of the closures of pared deformation spaces of degree $d$ critically fixed anti-rational maps 
$$
\displaystyle\bigcup_{\Gamma} \overline{\mathcal{H}_\Gamma(K)}  \subset \mathcal{M}_d^- \text{ is connected.}
$$

Similarly, the union of the closures of hyperbolic components of degree $d$ critically fixed anti-rational maps $\displaystyle\bigcup_{\Gamma} \overline{\mathcal{H}_\Gamma}$ is connected.
\end{cor}

We remark that using our techniques, one can show that $\displaystyle\bigcup_{\Gamma} \overline{\mathcal{H}_\Gamma(K)}$ and $\displaystyle\bigcup_{\Gamma} \overline{\mathcal{H}_\Gamma}$ are both path connected.

\subsection*{Monodromy representations}

Let $\mathfrak{M}_n$ be the moduli space of marked circle packings with $n$ circles in $\C$.
This space is closely related to the union of all deformation spaces of kissing reflection groups of rank $n$.
Let $\mathfrak{S}_n$ be the configuration space of $n$ marked points in $\C$.
There exists a natural map $\Psi: \mathfrak{M}_n \longrightarrow \mathfrak{S}_n$ which associates the centers of the circles to a marked circle packing.
In \cite{HT}, Hatcher and Thurston proved that the map $\Psi$ induces a homotopy equivalence between $\mathfrak{M}_n$ and $\mathfrak{S}_n$.

In the setting of anti-rational maps, the circles in the circle packings correspond to the pieces of suitable {\em Markov partitions} of Julia sets determined by Tischler graphs (see Section~\ref{sec:mr}). The braiding in the parameter space is manifested in the monodromy representations of these Markov partition pieces.

Let $\mathcal{H}_{\Gamma, Rat}(K) \subset \mathcal{H}_{\Gamma, Rat} \subset \Rat_d^-(\C)$ be the lifts (i.e., preimages under the projection map $\Rat_d^-(\C) \longrightarrow \mathcal{M}_d^-$) of $\mathcal{H}_\Gamma(K) \subseteq \mathcal{H}_\Gamma \subset \mathcal{M}_d^-$. 
Define
$$
\mathcal{X}_d(K) := \displaystyle\bigcup_{\Gamma} \overline{\mathcal{H}_{\Gamma, Rat}(K)} \subset \Rat_d^-(\C).
$$

In the case of {\em stable} families of rational maps, one can construct a monodromy representation by tracing periodic points (see \cite{McM88b}).
Note that $\mathcal{X}_d(K)$ is not stable. 
This forces us to adopt a different route to construct a monodromy representation. 

The Julia set of any map $R\in \mathcal{X}_d(K)$ admits a Markov partition consisting of $d+1$ pieces.
We prove that the Markov pieces move `continuously' away from the grand orbits of fixed points or $2$-cycles.
This allows us to construct a monodromy representation of $\mathcal{X}_d(K)$ into $\Mod(S_{0,d+1})$, the mapping class group of the $d+1$-punctured sphere. We show that

\begin{theorem}\label{thm:mono}
Let $d\geq 3$. For all large $K$, the monodromy representation
$$
\rho: \pi_1(\mathcal{X}_d(K))  \twoheadrightarrow \Mod(S_{0,d+1}) \text{ is surjective}.
$$
\end{theorem}

The above theorem indicates that the union of the closures of hyperbolic components of critically fixed anti-rational maps has a non-trivial global topology.
We also remark that we do not know the kernel of this monodromy representation (cf. \cite{BDK91}).

\subsection*{Outline of the paper}
In \S \ref{sec:bab} and \ref{sec:rert}, we study pared deformation spaces of degree $d$ anti-Blaschke products.
In \S \ref{sec:cad}, we use anti-Blaschke products to study pared deformation spaces for critically fixed anti-rational maps, and give a characterization for convergence of degenerations.
This allows us to prove Theorem \ref{thm:ab} and Theorem \ref{thm:main}.
The monodromy representation is studied in Section \ref{sec:mr}, where Theorem \ref{thm:mono} is proved.
Finally, in Appendices \ref{sec:pb} and \ref{sec:ve}, we study how the topology of the Julia sets of critically fixed anti-rational maps are related, and illustrate some consequences of our theory to the shared mating and self-bump phenomena.

\subsection*{Notes and references}

Early results on unboundedness of hyperbolic components were obtained in \cite{Makienko00,Tan02}. On the other hand, various hyperbolic components of quadratic rational maps were shown to be bounded in \cite{Epstein00}.

Our paper uses the technique of rescaling limits developed in \cite{Kiwi15}. Similar ideas were already used implicitly in an earlier work of Shishikura \cite{Shi}. Related studies can also be found in \cite{Arf1,Arf2}. Using the notion of rescaling limits, existence of unbounded hyperbolic components in a family of cubic rational maps was proved in \cite{BBM}, and some boundedness results in higher degrees were obtained by Nie and Pilgrim in \cite{NP18, NP19}. 

The pared deformation space we consider can be generalized to the case of periodic cycles \cite{L21a, L21b}. 
Many techniques used in this paper can be adapted to prove boundedness/unboundedness of such more general pared deformation spaces.

\noindent\textbf{Acknowledgements.} The authors are grateful to the anonymous referee for their valuable input.

\section{Degeneration of anti-Blaschke products}\label{sec:bab}
Let $\mathcal{R}_\Gamma$ be a critically fixed anti-rational map and let $\mathcal{H}_\Gamma \subset \mathcal{M}^-_d$ be the associated hyperbolic component.
Let $U_i, i=1,\cdots, m$ be a list of invariant critical Fatou components of $R \in \mathcal{H}_\Gamma$.
Then $U_i$ is a simply connected domain such that $R|_{U_i}$ is conformally conjugate to a degree $d_i$ proper anti-holomorphic map on $\D$, i.e., a degree $d_i$ anti-Blaschke product.
In fact, with appropriate marking, $\mathcal{H}_\Gamma$ is parametrized by a product of marked anti-Blaschke products (see \S \ref{subsec:pds}).

In this and the following section, we study degenerations of anti-Blaschke products.
We shall use it to understand deformations in $\mathcal{H}_\Gamma$ in \S \ref{sec:cad} and \S \ref{sec:mr}.
The key objectives of this section are
\begin{itemize}
\item to define {\em pared deformation spaces} for anti-Blaschke products (Definition \ref{defn:pdfb}), and to relate them to {\em quasi critically fixed deformations} (Definition \ref{defn:qfd} and Proposition \ref{prop:mc});
\item to construct {\em quasi-fixed trees} as combinatorial models to describe degenerations in pared deformation spaces (Theorem \ref{thm:qit}).
\end{itemize}
\
In \S \ref{sec:rert}, we will prove that every such model is realized (Theorem \ref{thm:rdtr}).
This will give a complete combinatorial description of degenerations in pared deformation spaces.

\subsection{Pared deformation space of anti-Blaschke products}
Let 
$$
\BP^{-}_d:= \{f(z) = \bar{z}\prod_{i=1}^{d-1} \frac{\bar{z}-a_i}{1-\overline{a_iz}}: |a_i| < 1\}
$$
be the space of normalized marked hyperbolic anti-Blaschke products.
Each map $f\in \BP^{-}_d$ has an attracting fixed point at $0$.
Let $m_{-d} (z):= \overline{z}^d$.
We shall identify $\mathbb{S}^1$ with $\R/\Z$. Under this identification, we have that $m_{-d}(t) = -d\cdot t$.
Note that for each $f\in \BP^-_d$, there exists a unique quasisymmetric homeomorphism $\eta_f:\mathbb{S}^1 \longrightarrow \mathbb{S}^1$ with
\begin{enumerate}
\item $\eta_f \circ m_{-d} = f \circ \eta_f$, and
\item $f\mapsto \eta_f$ is continuous on $\BP^-_d$ with $\eta_{m_{-d}} = \mathrm{id}$.
\end{enumerate}
The conjugacy $\eta_f$ provides a marking of the repelling periodic points for $f \in \mathcal{B}^-_d$.
We will call $\eta_f(1)$ the {\em marked repelling fixed point} for $f \in \mathcal{B}^-_d$.

Let $\mathcal{F}$ be the set of repelling fixed points for $m_{- d}$.
Then (the logarithm of the absolute value of) the multiplier of the repelling fixed point of $f$ corresponding to $x\in \mathcal{F}$ is 
$$
L_f(x) := \log |f'(\eta_f(x))|.
$$
Note that $f'$ is understood as $\frac{\partial}{\partial \overline{z}} f$.

As discussed in \S \ref{subsec:tct}, to define the analog of Thurston's deformation space of pared manifolds in the rational map setting, we need to control the multipliers of the repelling fixed points.

Thus, we define the following \emph{pared deformation space}.
\begin{defn}\label{defn:pdfb}
Let $K >0$. We define
$$
\BP^-_d(K):=\{f\in \BP^-_d: L_f(x) \leq K \text{ for all } x\in \mathcal{F}\}.
$$
We denote by $\mathring{\BP}^-_d(K)$ the component of $\BP^-_d(K)$ that contains $\bar z^d$.
\end{defn}

Let $\D$ be the unit disk equipped with the hyperbolic metric $d_\D$. 
For $f\in \BP_d^-$, we denote the set of all critical points of $f$ in $\D$ by $C(f)$.
Constraining the multipliers of repelling fixed points of $f$ is closely related to controlling the displacements of its critical points. 
To make this connection precise, we define
\begin{defn}\label{defn:qfd}
A degree $d$ anti-Blaschke product $f \in \BP^-_d$ is said to be {\em $M$-quasi critically fixed}, for $M>0$, if for all $c \in C(f)$, 
$$
d_{\D}(c, f(c)) \leq M.
$$
We denote the space of all $M$-quasi critically fixed anti-Blaschke products by $\BP^-_{d,qf}(M)$, and the component containing $\bar{z}^d$ by $\mathring{\BP}^-_{d,qf}(M)$.
\end{defn}
We first show that the above two notions are quasi equivalent.
\begin{prop}\label{prop:mc}
For each $K$, there exists $M = M(d, K)$ such that 
$$
\BP^-_d(K) \subseteq \BP^-_{d, qf}(M).
$$

Conversely, for each $M$, there exists $K' = K'(d, M)$ such that 
$$
\BP^-_{d, qf}(M)\subseteq \BP^-_d(K').
$$
\end{prop}

The proof will be given after some qualitative and quantitative bounds for anti-Blaschke products are established.

\subsection*{Algebraic limit.}
Let $\Rat_d^-$ denote the space of anti-rational maps of degree $d$.
By fixing a coordinate system, an anti-rational map can be expressed as a ratio of two homogeneous anti-polynomials $f(z:w) = (P(z,w):Q(z,w))$, where $P$ and $Q$ have degree $d$ with no common divisors.
Thus, using the coefficients of $P$ and $Q$ as parameters, we have
$$
\Rat_d^- = \Proj^{2d+1} \setminus V(\Res),
$$ 
where $\Res$ is the resultant of the two polynomials $P$ and $Q$, and $V(\mathrm{Res})$ is the hypersurface for which the resultant vanishes.
This embedding gives a natural compactification $\overline{\Rat_d^-} = \Proj^{2d+1}$, which will be called the {\em algebraic compactification}.
Maps $f\in \overline{\Rat_d^-}$ can be written as
$$
f = (P:Q) = (Hp: Hq),
$$
where $H = \gcd(P, Q)$.
We set $$\varphi_f := (p:q),$$ which is an anti-rational map of degree at most $d$.
The zeroes of $H$ in $\hat\C$ are called the {\em holes} of $f$, and the set of holes of $f$ is denoted by $\mathcal{H}(f)$.

If $\{f_n\}\subset \Rat_d^-$ converges to $f \in \overline{\Rat_d^-}$, we say that $f$ is the {\em algebraic limit} of the sequence $\{f_n\}$. 
It is said to have degree $k$ if $\varphi_f$ has degree $k$.
Abusing notation, sometimes we shall refer to $\varphi_f$ as the algebraic limit of $\{f_n\}$. 
A straightforward modification of \cite[ Lemma~4.2]{DeM05} to the anti-holomorphic case yields the following result.

\begin{lem}\label{lem:ac}
If $f_n$ converges to $f$ algebraically, then $f_n$ converges to $\varphi_f$ uniformly on compact subsets of $\widehat{\C}\setminus\mathcal{H}(f)$.
\end{lem}

A similar proof as in \cite[Lemma 4.5]{DeM05} gives
\begin{lem}\label{lem:ch}
If $f_n$ converges algebraically to $f$ of degree $\geq 1$ and $a\in \mathcal{H}(f)$ is a hole, then there exists a sequence of critical points $c_n$ of $f_n$ converging to $a$.
\end{lem}

\subsection*{Qualitative bounds for anti-Blaschke products.}
Note that by Schwarz reflection, any proper anti-holomorphic map on $\D$ extends uniquely to an anti-rational map on $\hat\C$.
This perspective is useful in many settings.
In this paper, we shall freely view them as anti-rational maps by an abuse of notation.

The normalization $f(0) = 0$ imposed on our anti-Blaschke products is useful essentially because of the following lemma (cf. \cite[Proposition~10.2]{McM09}).

\begin{lem}\label{deg_at_least_1_lem}
Any sequence $\{f_n\} \subset \BP^-_d$ has a subsequence converging algebraically to $f\in\overline{\Rat_d^-}$ where $e^{i\theta}\varphi_f \in \BP^-_k$ for some $\theta\in\R$ and $1\leq k \leq d$.

More generally, let $K \geq 0$. Any sequence of degree $d$ proper anti-holomorphic maps $f_n: \D\longrightarrow \D$ with $d_{\D}(0, f_n(0)) \leq K$ has a subsequence converging algebraically to $f\in\overline{\Rat_d^-}$ with degree $\geq 1$.
\end{lem}

\begin{proof}
By definition, any sequence $\{f_n\} \subset \BP^-_d$ can be written as $f_n(z)=  \overline{z}\prod_{i=1}^{d-1} \frac{\overline{z}-a_{i,n}}{1-\overline{a_{i,n}}\overline{z}}$. Possibly after passing to a subsequence, we may assume that each $\{a_{i,n}\}_n$ has a limit $a_{i,\infty}\in\overline{\D}$. A direct computation shows that if $a_{i,\infty}\in\mathbb{S}^1$ (respectively, if $a_{i,\infty}\in\D$), then the M{\"o}bius maps $\frac{\overline{z}-a_{i,n}}{1-\overline{a_{i,n}}\overline{z}}$ converge algebraically to the constant map $-a_{i,\infty}$ (respectively, to the M{\"o}bius map $\frac{\overline{z}-a_{i,\infty}}{1-\overline{a_{i,\infty}}\overline{z}}$). The first part of the lemma now follows from the above assertion and the fact that any subsequential algebraic limit of $\{f_n\}$ has a zero at the origin.

The more general statement follows by post-composing $f_n$ with a bounded sequence of $M_n \in \Isom^-(\D)$ so that $M_n \circ f_n \in \BP^-_d$ (such a bounded sequence $M_n$ exists because $f_n$ quasi fixes $0$).
\end{proof}

\subsection{Quasi equivalence of the definitions}\label{subsec:qe}
Let $\{f_n\}\subset \BP^-_d$.
Let 
$$
c_{1,n},\cdots, c_{d-1,n} \in \D
$$ 
be an enumeration of critical points of $f_n$.
It is understood that a critical point with multiplicity $k$ appears in the list $k$ times.
After possibly passing to a subsequence, we may assume that for any pair $(i,j)$, 
$$
\lim_{n\to\infty}d_{\D}(c_{i,n}, c_{j,n})
$$
exists (which can possibly be $\infty$).
We can thus construct a partition
$$
\{1,\cdots, d-1\} = \mathcal{I}_1\sqcup \mathcal{I}_2 \sqcup \cdots \sqcup \mathcal{I}_m
$$
so that 
\begin{itemize}
\item $\lim_{n\to\infty}d_{\D}(c_{i,n}, c_{j,n}) < \infty$ if $i$ and $j$ are in the same partition set; and
\item $\lim_{n\to\infty}d_{\D}(c_{i,n}, c_{j,n}) = \infty$ if $i$ and $j$ are in different partition sets.
\end{itemize}

We call $\mathcal{C}_{k,n}:=\{(c_{i,n}): i \in \mathcal{I}_k\}$ a {\em critical cluster} for $f_n$.
We say that the critical cluster $\mathcal{C}_{i,n}$ has degree $d_{i,n}=\sum m_{j,n} +1$ where $m_{j,n}$ is the order of the critical point and the sum is taken over all critical points in $\mathcal{C}_{i,n}$.
After passing to a subsequence, we can assume that $d_{i,n}$ and $m_{j,n}$ are all constant.
Therefore, we will simply denote the degree as $d_i$.

After passing to a subsequence, we can assume 
$$
\lim_{n\to\infty} d_{\D}(c_{i,n}, f_n(c_{i,n}))
$$ 
exists (which can possibly be $\infty$).
We say that the critical point $c_{i,n}$ is {\em active} if $\lim_{n\to\infty} d_{\D}(c_{i,n}, f_n(c_{i,n}))= \infty$, and {\em inactive} otherwise.
By the Schwarz lemma, if $c_{i,n}$ is active (respectively, inactive), then all of the critical points in the critical cluster containing $c_{i,n}$ are active (respectively, inactive).
Thus, we say that a critical cluster $\mathcal{C}_{k,n}$ is {\em active} if any of the critical points contained in $\mathcal{C}_{k,n}$ is active, and {\em inactive} otherwise.

We are ready to prove Proposition \ref{prop:mc}.
\begin{proof}[Proof of Proposition \ref{prop:mc}]
Let us fix $d\geq 2$, and $K>0$. By way of contradiction, suppose that there does not exist any $M>0$ with the desired property; i.e., there exists a sequence $\{f_n\} \subset \BP^-_d(K)$ and $c_n\in C(f_n)$ with
$$
d_{\D}(c_n, f_n(c_n)) \to \infty.
$$

After passing to a subsequence, let $\mathcal{C}_{1,n},\cdots, \mathcal{C}_{m,n}$ be the critical clusters of $f_n$.
As we will be using different coordinates later, we will use ${\bf 0}$ to denote the common fixed point of each $f_n$ in $\D$.
Let $d_{\D}({\bf 0}, \mathcal{C}_{k,n})$ be the sequence of distances between ${\bf 0}$ and the critical points in $\mathcal{C}_{k,n}$.
Our assumption implies that the sequence $\{f_n\}$ has at least one active critical cluster. After passing to a subsequence and reindexing, we assume that $\mathcal{C}_{1,n}$ is a furthest active critical cluster: $d_{\D}({\bf 0}, \mathcal{C}_{1,n}) \geq d_{\D}({\bf 0}, \mathcal{C}_{k,n})$ for all active $k$ and all $n$.
After reindexing, assume $c_{1,n} \in \mathcal{C}_{1,n}$; note that by the Schwarz lemma, $d_{\D}({\bf 0}, c_{1,n}) \to \infty$ as ${\bf 0}$ is fixed.

We are going to use {\em rescaling limits} at $c_{1,n}$ to conclude either there is a sequence of repelling fixed point of $f_n$ with multipliers $\to \infty$ or an active critical cluster that is further away. This will produce a contradiction and thus prove the first inclusion of the proposition.

\noindent\textbf{The set up:}
Let $L_n, M_n \in \Isom(\D)$ with $L_n(0) = c_{1,n}$ and $M_n(0) = f_n(c_{1,n})$.
Note that $M_n^{-1}\circ f_n \circ L_n(0) = 0$. 
So by Lemma \ref{deg_at_least_1_lem}, after passing to a subsequence, we may assume that $M_n^{-1}\circ f_n \circ L_n$ converges algebraically to $f \in \overline{\Rat^-_d}$.
We remark that this algebraic limit $f$ is referred to as a {\em rescaling limit} in the literature (see \cite{Kiwi15}).
By counting the critical points, we know the degree of $\varphi_f$ is $d_1$, in particular $\deg(\varphi_f) \geq 2$.

\begin{figure}[ht!]
\centering
\hspace{1cm}\ {\includegraphics[width=.72\linewidth]{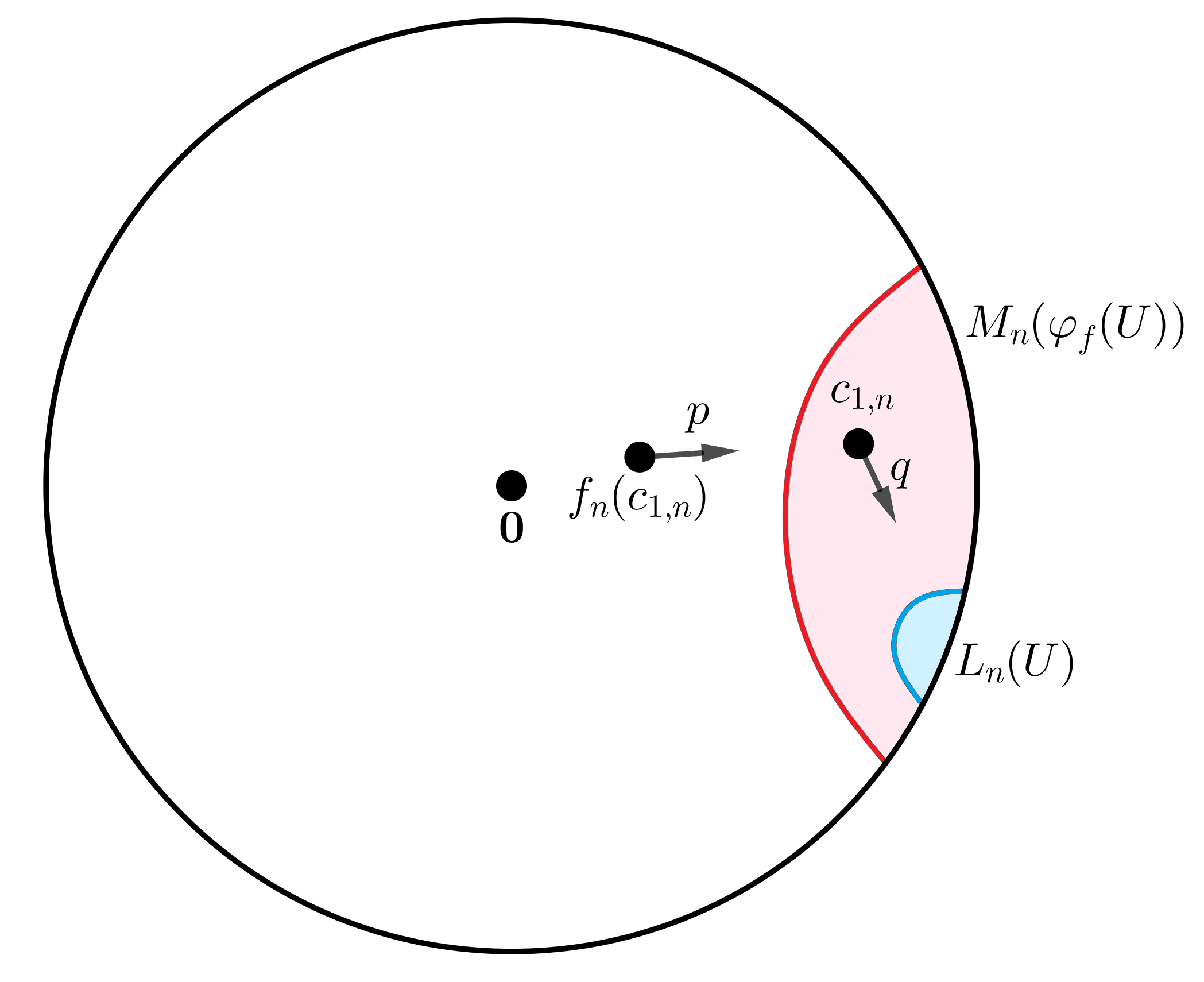}}
{\includegraphics[width=1\linewidth]{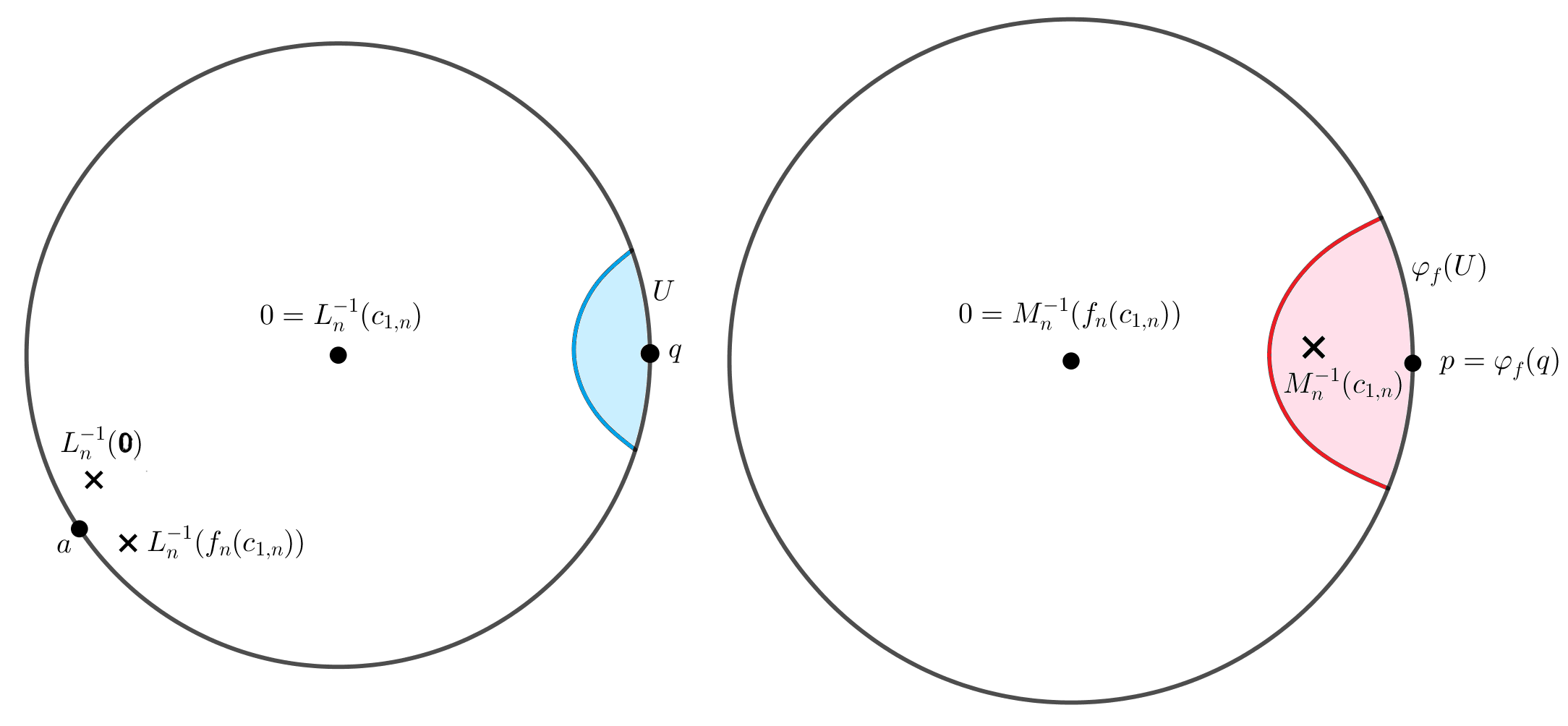}}
\caption{The top disk indicates the dynamics in the usual coordinate system.
The bottom two disks give a schematic picture in $M_n$ and $L_n$ coordinates.
Asymptotically, $c_{1,n}$ is in the direction $p\in \mathbb{S}^1$ when viewed from $f_n(c_{1,n})$. The point $q\in \mathbb{S}^1$, which is different from $\lim_n L_n^{-1}(\bf 0)$, is mapped to $p$ by the rescaling limit $\varphi_f$. 
Since $\varphi_f(U)$ 
is disjoint from $a$} in the $L_n$ coordinates, $L_n(U) \subseteq M_n(\varphi_f(U))$.
  
\label{blaschke_fig}
\end{figure}

As $d_{\D}(c_{1,n}, f_n(c_{1,n}))\to\infty$, we assume after passing to a subsequence that $\lim_{n} M_n^{-1}(c_{1,n}) = p\in \mathbb{S}^1$.
Similarly, we assume that $L_n^{-1}({\bf 0})$ and $L_n^{-1}(f_n(c_{1,n}))$ both converge in $\mathbb{S}^1$ (see the bottom left disk in Figure \ref{blaschke_fig} where both sequences are depicted converging to a single point as a consequence of the following claim).
\smallskip

\noindent\textbf{Claim 1:} $\lim_{n}  L_n^{-1}(f_n(c_{1,n})) = \lim_{n}  L_n^{-1}({\bf 0})$. 
\begin{proof}[Proof of Claim 1]
Suppose not. 
By standard hyperbolic geometry, this means that 
\begin{align*}
d_{\D}({\bf 0}, f_n(c_{1,n})) - d_{\D}({\bf 0}, c_{1,n}) &= d_{\D}(L_n^{-1}({\bf 0}), L_n^{-1}(f_n(c_{1,n}))) - d_{\D}(L_n^{-1}({\bf 0}), L_n^{-1}(c_{1,n})) \\
&= d_{\D}(L_n^{-1}({\bf 0}), L_n^{-1}(f_n(c_{1,n}))) - d_{\D}(L_n^{-1}({\bf 0}), 0) \\
&\to +\infty
\end{align*}
where the divergence follows from the estimate that
$$
d_{\D}(x_n,y_n) \approx d_{\D}(x_n,0) + d_{\D}(y_n,0),
$$
when $x_n, y_n$ converges to two different points on $\mathbb{S}^1$.
In particular $d_{\D}({\bf 0}, c_{1,n}) < d_{\D}({\bf 0}, f_n(c_{1,n}))$ for sufficiently large $n$, which is a contradiction to the Schwarz lemma.
\end{proof}

Let $a:= \lim_{n}  L_n^{-1}(f_n(c_{1,n})) = \lim_{n}  L_n^{-1}({\bf 0})$.
Since $\varphi_f$ had degree $\geq 2$, there are at least two preimages $q, q'\in \varphi_f^{-1}(p)$. Possibly after renaming $q$ and $q'$, we can assume that $q \neq a$ (as depicted in the lower left disk of Figure \ref{blaschke_fig}).
The first inclusion follows immediately from the following claim.
\smallskip

\noindent\textbf{Claim 2:} If $q$ is not a hole, then there is a sequence of repelling fixed point of $f_n$ with multipliers $\to \infty$.
Otherwise, there is an active critical cluster that is further than $\mathcal{C}_{1,n}$.

\begin{proof}[Proof of Claim 2]
If $q$ is not a hole, then $M_n^{-1}\circ f_n \circ L_n$ converges uniformly to $\varphi_f$ on a disk neighborhood $U$ of $q$.
Since $\varphi_f$ is a proper anti-holomorphic map on $\D$, $q \in \mathbb{S}^1$ is not a critical point.
After shrinking $U$ if necessary, we may also assume that $\varphi_f: U \longrightarrow \varphi_f(U)$ is univalent.
Note that as $L_n^{-1} \circ M_n$ converges compactly to the constant map $a$ away from $p$, one can verify that $L_n^{-1} (M_n(\varphi_f(U)))$ converges to the disk, and $U \subseteq L_n^{-1} (M_n(\varphi_f(U)))$ for all sufficiently large $n$. Thus
$$
L_n(U) \subseteq M_n(\varphi_f(U))
$$
for sufficiently large $n$ (see Figure \ref{blaschke_fig}). 
Since $\lim_{n}  d_{\D}(c_{1,n}, f_n(c_{1,n}))= \infty$, the modulus of the annulus $M_n(\varphi_f(U)) \setminus L_n(U) \to\infty$.
This means that there exists a repelling fixed point of $f_n$ in $L_n(U)$, and since $M_n(\varphi_f(U)) \setminus L_n(U)$ is an `approximate fundamental annulus' of this repelling fixed point, the multipliers of these fixed points tend to $\infty$. 

Otherwise, $q$ is a hole. Then by Lemma \ref{lem:ch}, there exists a sequence, say $\{L_n^{-1}(c_{2,n})\}$ converging to $q$ (see Figure \ref{blaschke_fig}).
Since $\lim_{n}  L_n^{-1}({\bf 0}) \in \mathbb{S}^1\setminus\{q\}$, we have $d_{\D}( {\bf 0}, c_{2,n})-d_{\D}({\bf 0}, c_{1,n}) \to\infty$.
Again, as $\lim_{n}  L_n^{-1}(f_n(c_{1,n})) \in \mathbb{S}^1\setminus\{q\}$, we have 
$$
d_{\D}(c_{2,n}, f_n(c_{1,n}))-d_{\D}(c_{2,n}, c_{1,n}) \to \infty.
$$
Then by the Schwarz lemma, we have $$d_{\D}(c_{2,n}, f_n(c_{1,n}))-d_{\D}(f_n(c_{2,n}), f_n(c_{1,n})) \to \infty,$$ so
$d_{\D}(c_{2,n}, f_n(c_{2,n})) \to \infty$.
Thus, $c_{2,n}$ is an active critical point of $f_n$ that is further away.
\end{proof}

We now proceed to prove the other inclusion. By way of contradiction, suppose that there exists a sequence $\{f_n\} \subset \BP^-_{d, qf}(M)$, and a sequence of repelling fixed points $p_n$ of $f_n$ with multipliers going to infinity.

Let $\mathcal{C}_{i,n}$ be a critical cluster and $c_{i,n}$ be a critical point in $\mathcal{C}_{i,n}$.
Choose $L_{i, n}\in \Isom(\D)$ such that $L_{i, n}(0) = c_{i, n}$.
After passing to a subsequence, we can assume that $L_{i, n}^{-1}\circ f_n \circ L_{i, n}$ converges algebraically to some $f_{i, \infty}\in \overline{\Rat^-_d}$ of degree at least $1$ by Lemma \ref{deg_at_least_1_lem}.
After passing to a subsequence, we assume $\lim_n L_{i, n}^{-1}(p_n)$ and $\lim_n L_{i, n}^{-1}(\mathcal{C}_{k,n})$ exist for $k \neq i$.
Let $p_{i, \infty} := \lim_n L_{i, n}^{-1}(p_n)$.
\smallskip

\noindent\textbf{Claim 3:}
There exists a critical cluster $\mathcal{C}_{i,n}$ so that $p_{i, \infty}$ is not a hole.
\begin{proof}[Proof of Claim 3]
Intuitively, the desired critical cluster is the `closest' one to the repelling fixed point $p_n$.
By Lemma \ref{lem:ch}, it suffices to find a critical cluster $\mathcal{C}_{i,n}$ so that $\lim_n L_{i, n}^{-1}(\mathcal{C}_{k,n}) \neq p_{i, \infty}$ for any $k \neq i$.

To find this cluster, we can start with $\mathcal{C}_{1,n}$.
Suppose that there are $l$ critical clusters with $L_{1,n}^{-1}(\mathcal{C}_{k,n}) \to p_{1, \infty}$.
After reindexing, we can assume $L_{1,n}^{-1}(\mathcal{C}_{2,n}) \to p_{1, \infty}$.
It can be verified that the number of critical clusters with $L_{2, n}^{-1}(\mathcal{C}_{k,n}) \to p_{2, \infty}$ is at most $l-1$.
Continue this process finitely many times, and we have the claim.
\end{proof}

By Claim 3, the log-multiplier of $f_n$ at $p_n$, which equals to the log-multiplier of $L_{i, n}^{-1}\circ f_n\circ L_{i, n}$ at $L_{i, n}^{-1}(p_n)$, converges to $\log |\varphi_{f_{i, \infty}}'(p)|$. This is a contradiction.
\end{proof}

Let $P \subseteq \D\cong \Hyp^2$ and $Q\subseteq \mathbb{S}^1$, we say that $\mathcal{K}\subseteq \D$ is the convex hull of $P \cup Q$ if $\mathcal{K}$ is the smallest convex set containing $P$ and having $Q$ as the ideal points.
Let $a,b\in \overline{\D}$, we use $[a,b]$ to denote the hyperbolic geodesic connecting $a$ and $b$ (including the end-points $a,b$).
Note that if $a$ or $b$ is in $\mathbb{S}^1$, then $[a,b]$ is the infinite geodesic with the points at infinity included.
We shall use $(a,b]$, $[a,b)$ and $(a,b)$ to denote the geodesics with partial or no end-points included.

The critical displacement bound from Proposition~\ref{prop:mc} yields a similar displacement bound throughout a certain convex hull.

\begin{prop}\label{cor:ch}
Let $f \in \BP_d^-(K)$ and $\mathcal{K}:= \chull (C(f) \cup F(f))$ where $C(f)$ is the set of critical points of $f$ in the unit disk $\D$, and $F(f)$ is the set of repelling fixed points of $f$ on $\mathbb{S}^1$.
There exists a constant $M=M(d,K)$ such that
$$
d_{\D}(x, f(x)) < M
$$
for all $x\in \mathcal{K}$.
\end{prop}
\noindent The result is a straightforward consequence of the following facts.\\
\noindent$\bullet$ The bound on the repelling fixed point multipliers implies that points in $C(f)$ and points in $\D$ close to $F(f)$ are moved by a uniformly bounded hyperbolic distance by $f$.\\
\noindent$\bullet$ If two points are moved by a bounded hyperbolic distance by $f$, then any point on the geodesic connecting them is also moved by a bounded hyperbolic distance.
\smallskip

\noindent We work out the details for completeness.
\begin{proof}
Let $x\in \mathcal{K}$, then $x$ lies in some (ideal) triangle with vertices in $C(f) \cup F(f)$.
Since hyperbolic triangles are thin, $x$ is within a uniformly bounded distance from some geodesic $\gamma:=[p_1,p_2]$ where $p_1, p_2\in C(f) \cup F(f)$.

We associate a truncated geodesic to $\gamma$. For each $i\in\{1,2\}$ so that $p_i \in C(f)$, let $q_i:=p_i$. For each $i$ so that $p_i \in F(f)$,  the fact that fixed point multipliers have bounded modulus permits us to choose $q_i\in \gamma$ sufficiently close to $p_i$ in the Euclidean metric so that $d_{\D}(q_i, f(q_i)) < M_1$ for some constant $M_1$ depending only on $K$.

Let $\gamma' = (q_1, q_2)$ be any truncated geodesic segment with $q_i$ as above. 
\smallskip

\noindent\textbf{Claim:} $d_{\D}(y, f(y)) \leq M_3$ for some $M_3=M_3(d,K)>0$, for all $y\in (q_1, q_2)$.
\begin{proof}[Proof of Claim]
By Proposition \ref{prop:mc}, $q_i$ is moved at most distance $M_2=M_2(d,K)$ by $f$.
Thus if $y\in (q_1, q_2)$, by the Schwarz lemma, we have $d_{\D}(f(y), q_i) \leq d_{\D}(y, q_i) +  \mathrm{max}\{M_1, M_2\}$.
We normalize so that $y= 0$, and $(q_1, q_2) \subseteq \R$.
Let $H_\pm$ be the horocycles based at $\pm 1$ that intersect the interval $(-1,1)$ at $\mp a$ where $a>0$ and $d_{\D}(0,a) = \mathrm{max}\{M_1, M_2\}$, and let $D_\pm$ be the associated horo-disks.
Then by standard hyperbolic geometry, we conclude that $f(y) \in D_+ \cap D_-$.
Therefore, $d_{\D}(y, f(y)) \leq M_3$ for some constant $M_3$ depending only on $d$ and $K$.
\end{proof}
Since $d_{\D}(x,\gamma)$ is uniformly bounded, once again by the Schwarz lemma, we have that $d_{\D}(x, f(x)) < M$ for all $x\in \mathcal{K}$, where $M=M(d,K)$.
\end{proof}

\subsection{Quasi-fixed trees}
Let $\mathcal{T}$ be a finite tree with vertex set $\mathcal{V}$.
A vertex $v\in \mathcal{V}$ is called 
\begin{itemize}
\item an {\em end} if $\mathcal{T}-\{v\}$ is connected; and 
\item a {\em branch point} if $\mathcal{T}- \{v\}$ has three or more components.
\end{itemize}
We denote the set of ends by $\epsilon(\mathcal{T})$ and the set of branch points by $\beta(\mathcal{T})$.

A {\em ribbon structure} on a finite tree $\mathcal{T}$ is the choice of a planar embedding $\mathcal{T} \xhookrightarrow{} \R^2$ up to isotopy.
The ribbon structure can be specified by a cyclic ordering on the edges incident to each vertex.
A finite tree with a ribbon structure will be called a finite ribbon tree.

\begin{defn}
We say that a finite ribbon tree $\mathcal{T}$ is a {\em (marked) $(d+1)$-ended ribbon tree} if
\begin{itemize}
\item $\mathcal{V} = \beta(\mathcal{T}) \cup \epsilon(\mathcal{T})$; and
\item $|\epsilon(\mathcal{T})| = d+1$, with a marked end-point $x\in \epsilon(\mathcal{T})$.
\end{itemize}
We denote the valence of a vertex $v\in\beta(\mathcal{T})$ by $\textrm{val}(v)$.
\end{defn}
\begin{rmk}
In this paper, we will only consider marked trees with a preferred isotopy class of planar embedding. For simplicity of notation, we will often drop the word `marked' and `ribbon'.
\end{rmk}

For our purposes, it is useful to define the {\em interior} of a finite tree $\mathcal{T}$ as
$$
\Int(\mathcal{T}) = \mathcal{T}\setminus \epsilon(\mathcal{T}).
$$
We shall prove that

\begin{theorem}\label{thm:qit}
Let $\{f_n\}\subset \BP^-_d(K)$. After passing to a subsequence, there exists a constant $K' = K'(d, K)> 0$, a $(d+1)$-ended ribbon tree $\mathcal{T}$, and a sequence of proper injective maps
$$
\phi_n: \Int(\mathcal{T}) \longrightarrow \D
$$
such that
\begin{itemize}
\item (Ends approximating.) 
The map $\phi_n$ extends to a continuous map $\phi_n: \mathcal{T} \longrightarrow \overline{\D}$.
Moreover, $\phi_n$ gives a bijection between $\epsilon(\mathcal{T})$ and $F(f_n)$ respecting the markings. 

\item (Degenerating vertices.) If $v \neq w \in \beta(\mathcal{T})$, then 
$$
d_{\D}(\phi_n(v), \phi_n(w)) \to \infty.
$$

\item (Geodesic edges.) If $E = [v, w] \subseteq \mathcal{T}$ is an edge, then $\phi_n(E)$ is the hyperbolic geodesic connecting $\phi_n(v)$ and $\phi_n(w)$\footnote{If $\phi_n(v)$ or $\phi_n(w)$ is in $\mathbb{S}^1$, then $\phi_n(E)$ is a hyperbolic geodesic ray.}.

\item (Critically approximating.) If $v\in \beta(\mathcal{T})$, then exactly $\val(v) - 2$ critical points of $f_n$ (counted with multiplicity) lie within a uniformly bounded distance from $\phi_n(v)$.

\item (Quasi-fixed.)
If $x\in \Int(\mathcal{T})$, then
$$
d_{\D}(f_n(\phi_n(x)), \phi_n(x)) \leq K' \text{ for all } n. 
$$

\end{itemize}
\end{theorem}

\begin{rmk}
We shall call the pair $(\mathcal{T}, \phi_n)$, or its image $\mathcal{T}_n:=\phi_n(\mathcal{T})$ the {\em quasi-fixed trees} for the sequence $f_n$. 
We remark that technically, this pair is associated to the particular subsequence chosen. We will be a little sloppy here and assume such a subsequence is already chosen for $f_n$.
We also remark that such ambiguity can be resolved by introducing an ultrafilter as in \cite{L19,L21c}.
\end{rmk}

A similar construction works in a more general orientation preserving setting (see \cite{L21a, L21b}).
The general construction works in the orientation reversing case as well.
For completeness, we briefly sketch the construction here.

The construction is by induction, and has two steps.
First, we construct the {\em core} $\mathcal{T}^c = \hull(\beta(\mathcal{T}))$ by taking the `spines' of a degenerating sequence of hyperbolic polygons $\chull(\{c_{1,n}, \cdots, c_{d-1,n}\})$ given by the convex hull of the critical points of $f_n$.
Next, we construct $\mathcal{T}$ by `attaching' ends to appropriate vertices of $\mathcal{T}^c$.

\subsection*{Construction of the core $\mathcal{T}^c$}
Let $c_{1,n},\cdots, c_{d-1,n} \in \D$ be an enumeration of critical points of $f_n$.
With the notation in \S \ref{subsec:qe}, 
we choose a representative $b_{k,n} \in \mathcal{C}_{k,n}$, $k = 1,\cdots, m$ for each critical cluster.
Let 
$$
\mathcal{P}:= \{(b_{k,n})_n: k=1,\cdots, m\}
$$
be the set of sequences (indexed by $n$) of representatives.
Note that each element in $\mathcal{P}$ is a sequence of points in $\D$, and there are $m$ elements in $\mathcal{P}$.
We also set
$$
\mathcal{P}_n:= \{b_{k,n}: k=1,\cdots, m\} \subset \D
$$
as the set of $n$-th terms in the representative sequences.
Note that by construction, for $i\neq j$,
$$
\lim_{n\to\infty} d_{\D}(b_{i,n}, b_{j,n})  = \infty.
$$

We construct a sequence of finite trees $\mathcal{T}^c_n$ with vertex set $\mathcal{P}_n$ inductively as follows (induction on the number of vertices).

As the base case, we let $\mathcal{T}^{c,1}_n = \{b_{1,n}\}$ be the degenerate tree with a single vertex $b_{1,n}$.

Assume that $\mathcal{T}^{c,k}_n$ is constructed with vertex set $\mathcal{V}^{c,k}_n= \{b_{1,n}, \cdots, b_{k,n}\}$.
Assume as an induction hypothesis that
\begin{align}\label{Eqn:H1}
\min_{j=k+1,\cdots, m} d_{\D}(b_{j,n}, \chull(\mathcal{V}^{c,k}_n)) \to \infty,
\end{align}
which is trivially satisfied for $k=1$.
We also assume that the edges of $\mathcal{T}^{c,k}_n$ are hyperbolic geodesic segments.

Roughly speaking, we will construct $\mathcal{T}^{c,k+1}_n$ by suitably adding the point $b_{j,n}$ (to $\mathcal{T}^{c,k}_n$) that is closest to the convex hull of the vertices of $\mathcal{T}^{c,k}_n$. To formalize this idea, note that after passing to a subsequence and changing indices, we may assume for all $n$,
$$
d_{\D}(b_{k+1,n}, \chull(\mathcal{V}^{c, k}_n)) = \min_{j=k+1,\cdots, m} d_{\D}(b_{j,n}, \chull(\mathcal{V}^{c, k}_n)).
$$
We now proceed to describe the vertex of $\mathcal{T}^{c,k}_n$ to which $b_{k+1,n}$ will be connected by an edge. Let $a_n \in \chull(\mathcal{V}^{c, k}_n))$ be the projection of $b_{k+1, n}$ onto $\chull(\mathcal{V}^{c, k}_n))$.
After passing to a subsequence, we may assume
$\lim_{n\to\infty} d_{\D}(a_n, b_{i,n})$
exists, which can possibly be $\infty$, for all $i = 1,\cdots, k$.

\begin{lem}\label{lem:projbranch}
With the above notation, there exists a unique $l \in \{1,\cdots, k\}$ with 
$$
\lim_{n\to\infty} d_{\D}(a_n, b_{l,n}) < \infty.
$$
\end{lem}
\begin{proof}
The uniqueness follows from the fact that 
$$
\lim_{n\to\infty} d_{\D}(b_{i,n}, b_{j,n}) = \infty
$$
for all $i,j$.

Since $a_n \in \chull(\mathcal{V}^{c, k}_n))$, there exists $K'$ with $d_{\D}(a_n, f_n(a_n)) \leq K'$ by Proposition \ref{cor:ch}.
Let $M_n \in \Isom(\D)$ with $M_n(0) = a_n$.
Then after passing to a subsequence, $M_n^{-1} \circ f_n \circ M_n$ converges compactly to a proper map $g: \D \longrightarrow \D$ of degree $\geq 1$ by Lemma \ref{deg_at_least_1_lem}.

Suppose that $\lim_{n\to\infty} d_{\D}(a_n, b_{i,n}) = \infty$ for all $i\in \{1,\cdots, k\}$.
Then $\deg g = 1$ as there are no critical points of $f_n$ within bounded distance from $a_n$.

Let $a_n \in [b_{j_1,n}, b_{j_2,n}]$ for some $j_1, j_2 \in \{1,\cdots, k\}$.
Then $M_n^{-1}(b_{j_1,n})$, $M_n^{-1}(b_{j_2,n})$, $M_n^{-1}(b_{k+1,n})$ converge to $3$ distinct points on $\mathbb{S}^1$.
By Proposition \ref{cor:ch}, the geodesic segments $[a_n, b_{j_1,n}], [a_n, b_{j_2,n}], [a_n, b_{k+1,n}]$ are $K'$-quasi-fixed. 
Thus, the $3$ limit points are fixed by $g$.
But this is not possible as any degree $1$ proper anti-holomorphic map of $\D$ has exactly $2$ fixed points on $\mathbb{S}^1$.
\end{proof}

Using Lemma \ref{lem:projbranch}, we define
$$
\mathcal{T}^{c,k+1}_n := \mathcal{T}^{c,k}_n \cup [b_{l,n}, b_{k+1,n}],
$$
where $ [b_{l,n}, b_{k+1,n}]$ is a hyperbolic geodesic connecting $b_{l, n}$ to $b_{k+1, n}$.
It is easy to verify that the induction hypothesis (see equation (\ref{Eqn:H1})) is satisfied, and each edge is a hyperbolic geodesic segment.

Applying the above inductive construction $m-1$ times, we obtain the tree 
$$
\mathcal{T}^c_n := \mathcal{T}^{c,m}_n
$$ 
containing all $m$ vertices of $\mathcal{P}_n$.

\subsection*{Attaching ends to the core}
Index the set $F(f_n) = \{p_{1,n}, \cdots, p_{d+1,n}\}$. 
For each $j = 1, \cdots, d+1$, let $a_{j,n}$ be the projection of $p_{j,n}$ onto $\chull(\mathcal{V}^c_n)$, where $\mathcal{V}^c_n$ is the vertex set of $\mathcal{T}^c_n$.

We fix $j\in\{1,\cdots, d+1\}$. After passing to a subsequence, we may assume
$\lim_{n\to\infty} d_{\D}(a_{j,n}, b_{i,n})$
exist, which can possibly be $\infty$, for all $i = 1, \cdots, m$.
The proof of Lemma \ref{lem:projbranch} applies verbatim to the current setting to show that there exists a unique $l_j \in \{1,\cdots, m\}$ with 
$$
\lim_{n\to\infty} d_{\D}(a_{j,n}, b_{l_j,n}) < \infty.
$$
Thus, we construct 
$$
\Int(\mathcal{T}_n) = \mathcal{T}^c_n \cup \bigcup_{j=1}^{d+1} [b_{l_j,n}, p_{j,n}) \subset \D,
$$
where $[b_{l_j,n}, p_{j,n})$ is a hyperbolic geodesic ray connecting $b_{l_j,n}$ to $p_{j,n}$.

We also define 
$$
\mathcal{T}_n = \Int(\mathcal{T}_n)  \cup \bigcup_{j=1}^{d+1} p_{j,n} \subset \D\cup \mathbb{S}^1,
$$
and declare the marked repelling fixed point of $f_n$ on $\mathbb{S}^1$ to be the marked endpoint of $\mathcal{T}_n$.

After passing to a subsequence, we may assume that the underlying finite ribbon trees $\mathcal{T}_n$ are all isomorphic (respecting the marking of endpoints). 
We denote the isomorphism type of this finite ribbon tree by $\mathcal{T}$, and the sequence of planar embeddings by
$$
\phi_n: \mathcal{T} \longrightarrow \mathcal{T}_n.
$$

\begin{lem}\label{lem:rdt}
$\mathcal{T}$ is a marked $(d+1)$-ended ribbon tree.
\end{lem}
\begin{proof}
It suffices to show every vertex $b_{i,n}\in \mathcal{P}_n$ is a branch point.
By Proposition \ref{cor:ch}, there exists $K'>0$ with $d_{\D}(b_{i,n}, f_n(b_{i,n})) \leq K'$.
Let $M_n \in \Isom(\D)$ with $M_n(0) = b_{i,n}$.
Then after passing to a subsequence, $M_n^{-1} \circ f_n \circ M_n$ converges compactly to a proper map $g: \D \longrightarrow \D$ of degree $\geq 1$ by Lemma \ref{deg_at_least_1_lem}.

The map $g$ has degree $\geq 2$ as some critical point of $f_n$ is within a uniformly bounded distance from $b_{i,n}$.
Thus $g$ has at least $3$ fixed points on $\mathbb{S}^1$.

Let $a\in \mathbb{S}^1$ be a fixed point of $g$, then either there exists a sequence of fixed points $p_n \in F(f_n)$ or a sequence of critical points $c_n \in C(f_n)$ converging to $a$ by Lemma \ref{lem:ch}.
The construction of $\mathcal{T}_n$ now implies that the fixed point $p_n$ or some critical point $c_n'$ of $f_n$ (converging to $a$) is adjacent to $b_{i,n}$ in the tree $\mathcal{T}_n$. In fact, this argument shows that the valence of the vertex $b_{i,n}$ of $\mathcal{T}_n$ is equal to the number of fixed points of $g$ on $\mathbb{S}^1$. The result now follows from our observation that $g$ has at least $3$ fixed points on $\mathbb{S}^1$.
\end{proof}

Let $b\in \mathcal{T}$ be a branch point. Let $\mathcal{C}_{i,n}$ be the corresponding critical cluster for $b$.
Recall that we have assumed (by passing to a subsequence) that the degree of each critical cluster $\deg(\mathcal{C}_{i,n})$ is independent of $n$.
We define $\deg(b) := \deg(\mathcal{C}_{i,n})$.
\begin{cor}\label{lem:vd}
The valence of a branch point $b$ of $\mathcal{T}$ is equal to $\deg(b)+1$. 
\end{cor}
\begin{proof}
We will continue to use the notation introduced in the proof of Lemma~\ref{lem:rdt}. By the proof of Lemma~\ref{lem:rdt}, the valence of the vertex $b_{i,n}$ of $\mathcal{T}_n$ is equal to the number of fixed points of $g$, which is two more than the number of critical points of $g$. But the number of critical points of $g$ is equal to the number of critical points of $f_n$ that lie within a uniformly bounded distance from $b_{i,n}$. We conclude (using the definition of $\deg(b)$) that the valence of the vertex $b_{i,n}$ of $\mathcal{T}_n$ is equal to $\deg(b)+1$. 
\end{proof}

\begin{proof}[Proof of Theorem \ref{thm:qit}]
By Lemma \ref{lem:rdt}, $\mathcal{T}$ is a $(d+1)$-ended ribbon tree.
It is easy to check that the first three conditions are satisfied.

The fourth (critically approximating) property follows from the proof of Corollary~\ref{lem:vd}. Since $\Int(\mathcal{T}_n) \subset \chull C(f_n) \cup F(f_n)$, the last quasi-fixed condition is satisfied by Proposition \ref{cor:ch}.
\end{proof}

\subsection{A special point on the quasi-fixed tree}\label{subsec:sqft}
Given $\{f_n\}\subset \BP_d^-(K)$ with quasi-fixed tree $\mathcal{T}$, we remark that there is a special point $p$ on $\mathcal{T}$ which corresponds to the attracting fixed point $0$ of $f_n$.
More precisely, 
\begin{itemize}
\item either there exists a (unique) vertex $v \in \mathcal{T}$ so that $d_{\D}(0, \phi_n(v))$ is bounded;
\item or there exists a (unique) edge $E \subseteq \mathcal{T}$ and $x_n \in \phi_n(E)$ so that
$d_{\D}(0, x_n)$ is bounded, and $d_{\D}(\phi_n(\partial E), x_n) \to \infty$.
\end{itemize}
In the first case, we take $p = v$ and in the second case, we can introduce a new vertex on the edge $E$ and let $p$ be that point.
This makes $(\mathcal{T}, p)$ a pointed tree (see \S \ref{subsec:pmrt} for more discussions).
We remark that in the second case, $\mathcal{T}$ is no longer a $d+1$-ended ribbon tree as $p$ has valence $2$.
Although this does not affect our argument, it makes some theorems slightly more cumbersome to state for pointed quasi-fixed trees.

We also remark that 
while the quasi-fixed trees are used to determine {\em whether a pared deformation space bifurcates to another} (see \S \ref{sec:cad}),
the special point gives additional information on {\em how a pared deformation space bifurcates to the other} (see Appendix \ref{sec:ve}).
To keep the statements simple, in many cases, we will only consider quasi-fixed trees without the special point.

\section{Realization of $(d+1)$-ended ribbon trees}\label{sec:rert}

Let $\mathcal{T}$ be a $(d+1)$-ended ribbon tree. We say it is {\em realizable} if there exist a sequence of quasi critically fixed maps $\{f_n\}\subset \BP^-_d(K)$ (for some $K>0$) and a sequence of planar embeddings $\phi_n$ such that $\phi_n(\mathcal{T})$ is the sequence of quasi-fixed trees for $f_n$.
In this section, we shall prove
\begin{theorem}\label{thm:rdtr}
For $d\geq 2$, every $(d+1)$-ended ribbon tree $\mathcal{T}$ is realizable.
\end{theorem}

\subsection{Pointed metric $(d+1)$-ended ribbon trees}\label{subsec:pmrt}
A $(d+1)$-ended ribbon tree $\mathcal{T}$ 
together with a special vertex $p\in \beta(\mathcal{T})$ is called a {\em pointed $(d+1)$-ended ribbon tree}. It denoted by $(\mathcal{T}, p)$.

For an anti-Blaschke product $f\in \BP^-_d$, $0 \in \D$ is the unique attracting fixed point of $f$.
It is thus natural to keep track of this attracting fixed point for quasi critically fixed degenerations.

We say that a pointed $(d+1)$-ended ribbon tree $(\mathcal{T}, p)$ is realizable if there exists a quasi critically fixed sequence $\{f_n\}\subset \BP^-_d(K)$ (for some $K>0$) realizing $\mathcal{T}$ such that $d_{\D}(0, \phi_n(p))$ is bounded.

\begin{rmk}
We remark that it is possible that the rescaling limit of $f_n$ based at the origin has degree one, in which case $d_{\D}(0, \phi_n(v)) \to \infty$ for all vertices $v\in \mathcal{T}$ (see \S \ref{subsec:sqft}). 
In this situation, we can introduce a new vertex $p \in \mathcal{T}$ so that $d_{\D}(0, \phi_n(p))$ is bounded.
The new vertex $p$ has valence $2$, so $\mathcal{T}$ is no longer a $d+1$-ended ribbon tree by our definition.
To simplify the notation, we shall first discuss the situation when the rescaling limit based at the origin has degree at least two.
In \S \ref{sec:d1}, we shall discuss some minor modifications required to handle the case when the rescaling limit at the origin has degree one.
\end{rmk}
 
It is also useful to introduce the following modified edge metric $d_\mathcal{T}$ on a $(d+1)$-ended ribbon tree $\mathcal{T}$:
\begin{itemize}
\item $d_\mathcal{T}(v,w)$ is the number of edges in $[v,w]$ if $v, w\in \beta(\mathcal{T})$;
\item $d_\mathcal{T}(v,w) = \infty$ if $v$ or $w$ is in $\epsilon(\mathcal{T})$.
\end{itemize}
In other words, $d_\mathcal{T}$ is the usual edge metric restricted to the vertices in the core $\mathcal{T}^c$, while the ends $\epsilon(\mathcal{T})$ are considered infinitely far away in this metric.

To define the distance between points on edges, we can extend the metric linearly and call
$$
(\mathcal{T}, p, d_\mathcal{T})
$$
a {\em pointed metric $(d+1)$-ended ribbon tree}. 

Let $s \in (0, \infty)$, we shall denote
$$
(\mathcal{T}, p, d_\mathcal{T}^s)
$$
as the dilation of the pointed metric $(d+1)$-ended ribbon tree by $s$.
More precisely, for any two points $v, w \in \mathcal{T}$, 
$$
d_\mathcal{T}^s(v,w) = sd_\mathcal{T}(v,w).
$$

We need to generalize the notion of realizable pointed $(d+1)$-ended ribbon trees to the context of metric trees. This is done in the following definition, where we further add a normalization to the effect that in the realization, the `critical cluster' corresponding to the branch point $p$ is a single critical point at the origin.
We state the definition for a continuous family $f_s\in \mathcal{B}^-_d(K)$.

\begin{defn}\label{defn:rpmt}
A pointed metric $(d+1)$-ended ribbon tree $(\mathcal{T}, p, d_\mathcal{T})$ is said to be {\em realizable} if 
there exist $K, M, R > 0$, a family $f_s \in \BP^-_d(K)$ with $s \in (0, \infty)$, and a family of $M$-quasi-isometric embeddings
$$
\phi_s: (\Int(\mathcal{T}), p, d_\mathcal{T}^s) \longrightarrow (\D, 0, d_{\D})
$$
so that
\begin{itemize}
\item (Superattracting.) $0$ is a critical point of $f_s$ of multiplicity $\val(p)-2$.

\item (Ends approximating.) The map $\phi_s$ extends to a continuous map $\phi_s: \mathcal{T} \longrightarrow \D \cup \mathbb{S}^1$ which gives a bijection between $\epsilon (\mathcal{T})$ and $F(f_s)$ respecting the markings.

\item (Geodesic edges.) If $E = [v, w] \subset \mathcal{T}$ is an edge, then $\phi_s(E)$ is the hyperbolic geodesic connecting $\phi_s(v)$ and $\phi_s(w)$.

\item (Critically approximating.)  If $v\in \beta(\mathcal{T})$, then there are exactly $\val(v) - 2$ critical points counted with multiplicities in the hyperbolic ball $B_{\D}(\phi_s(v), R)$.

\item (Quasi-fixed.) If $x\in \Int(\mathcal{T})$, then 
$$
d_{\D}(f_s(\phi_s(x)), \phi_s(x)) \leq M \text{ for all } s.
$$
\end{itemize}
\end{defn}

\begin{rmk}
The definition of the metric $d_{\mathcal{T}}^s$ and the requirement that $\phi_s$ is an $M$-quasi-isometry together imply the `degenerating vertices' condition for $\phi_s(\mathcal{T})$; i.e., if $v \neq w \in \beta(\mathcal{T})$, then $d_{\D}(\phi_s(v), \phi_s(w)) \to \infty,$ as $s\to+\infty$.
\end{rmk}

We shall prove the following more general theorem which immediately implies Theorem \ref{thm:rdtr} (see \cite[Theorem 4.1]{L21a} for the even more general realization theorem of quasi post-critically finite degenerations in the orientation preserving setting).
\begin{theorem}\label{thm:cvr}
For $d\geq 2$, every pointed metric $(d+1)$-ended ribbon tree $(\mathcal{T}, p, d_\mathcal{T})$ is realizable.
\end{theorem}

We first observe $\mathcal{T}$ has exactly one branch point (i.e., if $\mathcal{T}^c=\{p\}$) if and only if $\mathcal{T}$ is a star-tree with $(d+1)$ ends. In this case, the constant family $f_s(z)=\overline{z}^d$ is easily seen to realize the pointed metric $(d+1)$-ended ribbon tree $(\mathcal{T}, p, d_\mathcal{T})$. Therefore, in the remainder of this section, we will assume that $\beta(\mathcal{T})$ contains at least two distinct points.

The proof is by induction on the degree $d$ of the anti-Blaschke products (or equivalently, on the number of ends $(d+1)$ of the tree).

In the base case $d=2$, there is only one $3$-ended ribbon tree $\mathcal{T}$, which is a `tripod'.
Therefore, there is only one pointed metric $3$-ended ribbon tree $(\mathcal{T}, p, d_\mathcal{T})$, which is realized by the constant family $f_s(z) = \bar{z}^2$.

Assume as induction hypothesis that any pointed metric $(d+1)$-ended ribbon tree is realizable.
Let $(\mathcal{T}, p, d_\mathcal{T})$ be a pointed metric $(d+2)$-ended ribbon tree.
The induction consists of two steps.

First, we shall construct a subtree $\widetilde{\mathcal{T}}  \subset \mathcal{T}$ which is a pointed metric $(d+1)$-ended ribbon tree.
By induction hypothesis, it is realizable by $\widetilde f_s \in \BP^-_d$.

Second, we prove realizability of $\mathcal{T}$ by carefully adding a zero of the anti-Blaschke product.
More precisely, we construct a sequence 
$$
f_s = e^{i\theta_s} \frac{\bar{z}-a_s}{1-\bar{a_s}\bar{z}} \cdot \widetilde f_s
$$ 
where $\theta_s$ and $a_s$ are chosen appropriately, {and show that M{\"o}bius conjugates of $f_s$ in $\BP^-_{d+1}$ yield a family realizing $\mathcal{T}$}.
In this second step, we establish some estimates that allow us to control the critical orbits of $f_s$.

We also remark that instead of adding zeros of anti-Blaschke products inductively, one can inductively add critical points by applying Heins theorem  (see \cite{Heins62} and \cite{Zakeri98}). 
When the tree has only 2 branch points, the construction using Heins theorem is easier as no induction step is needed in this case (see Proposition \ref{prop:rabl}).
In general, similar estimates as in the second step are needed to control the orbits of the critical points. 

\subsection{Construction of the reduction}
Let $(\mathcal{T}, p, d_\mathcal{T})$ be a pointed metric $(d+2)$-ended ribbon tree.
It is convenient to introduce a radius function 
$$
r: \beta(\mathcal{T}) \longrightarrow \R_{\geq 0}
$$
by setting $r(v) = d_\mathcal{T}(p, v)$.
We call a vertex $w\in \beta(\mathcal{T})$ a {\em farthest} vertex if
$$
r(w) = \max_{v\in \beta(\mathcal{T})} r(v).
$$
By construction, $w$ is an end-point of the core $\mathcal{T}^c$ of $\mathcal{T}$. Since $w$ is a branch point of $\mathcal{T}$, it follows that at least two vertices in $\epsilon(\mathcal{T})$ are adjacent to $w$. Denoting the marked end-point of $\mathcal{T}$ by $x$ ($\in \epsilon(\mathcal{T})$), we see that at least one end of $\mathcal{T}$ adjacent to $w$ must be different from $x$. We pick such an end, and call it $y$.

We now define 
$$
\widetilde{\mathcal{T}} = \mathcal{T} \setminus (w,y] \subset \mathcal{T}.
$$
Note that $\widetilde{\mathcal{T}}$ inherits a ribbon structure from $\mathcal{T}$. We declare $x\in\epsilon(\widetilde{\mathcal{T}})$ to be the marked end-point of $\widetilde{\mathcal{T}}$. Clearly, $\widetilde{\mathcal{T}}$ has $(d+1)$ ends.

If $w$ has valence $\geq 4$ in $\mathcal{T}$, then $w$ remains a branch point in $\widetilde{\mathcal{T}}$. Thus in this case, $\widetilde{\mathcal{T}}$ is a $(d+1)$-ended ribbon tree. On the other hand, if $w$ has valence $3$ in $\mathcal{T}$, then $w$ has valence $2$ in $\widetilde{\mathcal{T}}$. In this case, we throw away $w$ from the vertex set of $\widetilde{\mathcal{T}}$, thus producing a $(d+1)$-ended ribbon tree $\widetilde{\mathcal{T}}$.

Since $(d+2)\geq 4$, the special branch point $p$ of $\mathcal{T}$ lies in $\widetilde{\mathcal{T}}$ and is a branch point thereof. We declare $p$ to be the special branch point of $\widetilde{\mathcal{T}}$. Moreover, the metric $d_\mathcal{T}$ induces a metric on $\widetilde{\mathcal{T}}$, which agrees with the modified edge metric $d_{\widetilde{\mathcal{T}}}$ on $\widetilde{\mathcal{T}}$. Hence, $(\widetilde{\mathcal{T}} , p, d_{\widetilde{\mathcal{T}}})$ is a pointed metric $(d+1)$-ended ribbon tree.

In order to reconstruct $\mathcal{T}$ from $\widetilde{\mathcal{T}}$, we need to keep track of the point $w$ at which the tree $\mathcal{T}$ was pruned. With this in mind,  we call $(\widetilde{\mathcal{T}} , p, d_{\widetilde{\mathcal{T}} })$ the {\em reduction} of $(\mathcal{T}, p, d_\mathcal{T})$ with the \emph{regluing point} $w$.

\subsection{Geometric bounds for proper anti-holomorphic maps}
Before giving the proof of the induction step, we need some geometric bounds to study anti-Blaschke products.

We start with some useful estimates that will be used frequently in the proof.
The following lemma follows from the Schwarz lemma and Koebe's theorem.
\begin{lem}\label{lem:hme}
Let $U$ be a simply connected domain in $\C$ with hyperbolic metric $\rho_U |dz|$, then for $z\in U$,
$$
\frac{1}{2d_{\R^2}(z, \partial U)} \leq \rho_U(z) \leq \frac{2}{d_{\R^2}(z, \partial U)}.
$$
\end{lem}

The next lemma is an equivalent formulation of the fact that $(\D, d_{\D})$ is a hyperbolic metric space, and hence the angle between two sides of a geodesic triangle in $\D$ controls the corresponding Gromov inner product.

\begin{lem}\label{lem:alb}
For each positive angle $\theta$, there exists a constant $C(\theta)\in (0,\infty)$ such that for a (non-degenerate) hyperbolic triangle $\Delta ABC$ with $\angle ABC\geq \theta$, we have 
$$
0\leq d_{\D}(A,B)+d_{\D}(B,C)-d_{\D}(A,C)\leq C(\theta).
$$ 
\end{lem}

The following fact in hyperbolic geometry, although quite standard, turns out to be very useful (see \cite[Lemma 4.5]{L21a}).
\begin{lem}\label{lem:he}
Let $z, z'\in \D$ with $|z| = 1-\delta$ and $|z'| = 1- \delta^\alpha$, where $\alpha\in (0,1)$.
Then
$$
\alpha d_{\D}(0,z) \leq d_{\D}(0, z') \leq \alpha d_{\D}(0,z) + \log 2.
$$
\end{lem}

Given $a\in \D\setminus\{0\}$, we denote $\hat{a} = \frac{a}{|a|} \in \mathbb{S}^1$, $\delta_a = 1-|a|$, and $\rho_a= d_{\D}(0, a)$.
The following lemma is an analogue of \cite[Lemma 4.6]{L21a} in the anti-holomorphic setting.
\begin{lem}\label{lem:em}
Let $M_a(z) = \frac{\bar{z}-a}{1-\bar{a}\bar{z}}$, then for $0<\alpha<1$,
$$
|M_a(z) + \hat{a}| \leq 2\delta_a^{1-\alpha}
$$
for all $z\in B_{\mathbb{R}^2}(0, 1-\delta_a^\alpha)$.\footnote{Note that $B_{\D}(0, \alpha\rho_a) \subset B_{\mathbb{R}^2}(0, 1-\delta_a^\alpha)$ by Lemma \ref{lem:he}.}
\end{lem}
\begin{proof}
\begin{align*}
|M_a(z) + \hat{a}|
&=|\frac{\bar{z}-a+\hat{a}-\hat{a}\bar{a}\bar{z}}{1-\bar{a}{\bar{z}}}|\\
&=|\frac{\delta_a\bar{z}+\hat{a}\delta_a}{1-\bar{a}{\bar{z}}}|
\end{align*}
Since $|z|<1$ and $|\hat{a}|=1$, $|\delta_az+\hat{a}\delta_a| \leq 2\delta_a$.
Since $|z| \leq 1-\delta_a^\alpha$, $|1-\bar{a}{z}| \geq \delta_a^\alpha$.
Thus, we have
$|M_a(z) + \hat{a}| \leq \frac{2\delta_a}{\delta_a^\alpha} = 2\delta_a^{1-\alpha}$.
\end{proof}

We will frequently use the following anti-holomorphic analogue of \cite[Theorem 10.11]{McM09}.
It follows immediately from the holomorphic setting by post-composing $f$ with $\bar z$.

\begin{theorem}\cite[Theorem 10.11]{McM09}\label{thm:almostisometry}
There is a constant $R>0$ such that for any proper anti-holomorphic map $f: \D \longrightarrow \D$ of degree $d$:
\begin{enumerate}
\item if $d_{\D}([a,b], C(f)) > R$, then $d_{\D}(f(a), f(b)) = d_{\D}(a,b) + O(1)$; and
\item if $d_{\D}([a,b], f(C(f))) > R$, then $d_{\D}(f^{-1}(a), f^{-1}(b)) = d_{\D}(a,b) + O(1)$.
\end{enumerate}
Here $C(f)$ denotes the critical set of $f$, and $f^{-1}$ is any inverse branch of $f$ that is continuous along $[a,b]$.
\end{theorem}

We shall also use the following estimate on locations of critical points, which is an analogue of \cite[Lemma 4.8]{L21a} in the anti-holomorphic setting.
\begin{lem}\label{lem:cl}
Let $C, \theta > 0$, and $f: \D\longrightarrow \D$ be a proper anti-holomorphic map of degree $d$.
Assume that there exist $k+1$ points $x_1,\cdots, x_k, z\in \D$ with
\begin{itemize}
\item $\angle x_i z x_j \geq \theta$ for any pairs $i \neq j$; and
\item $f(z) = a$, $f(x_i) = b$ for all $i$, with
$$
d_{\D}(z, x_i) \leq d_{\D}(a,b) + C.
$$
\end{itemize}
Then there exists a constant $R = R(C, \theta, d)$ such that there are $k-1$ critical points of $f$ counting multiplicities in $B_{\D}(z, R)$.
\end{lem}
\begin{proof}
Let $L = d_{\D}(a,b)$, and $z_{i,t}$ and $a_t$ be the points on the geodesic segments $[z, x_i]$ and $[a,b]$ with $d_{\D}(z, z_{i,t}) = d_{\D}(a, a_t) = t$.

\begin{figure}[ht!]
\begin{tikzpicture}[thick]
\node[anchor=south west,inner sep=0] at (0,0) {\includegraphics[width=.46\linewidth]{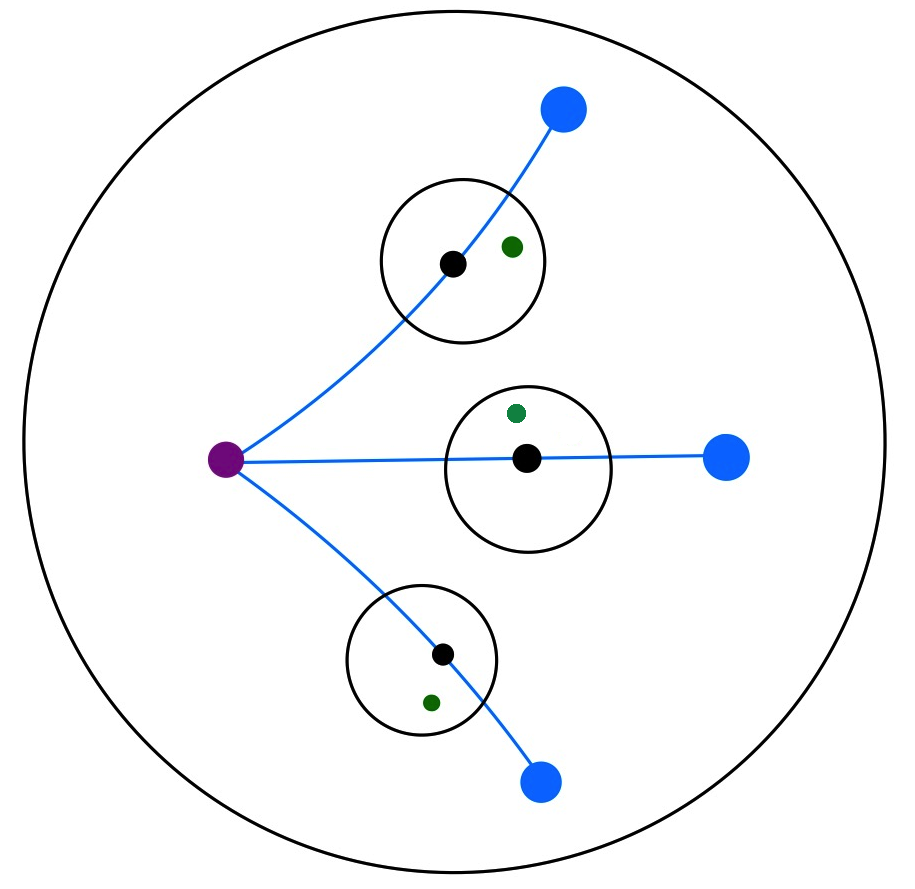}};
\node[anchor=south west,inner sep=0] at (6.6,0) {\includegraphics[width=.46\linewidth]{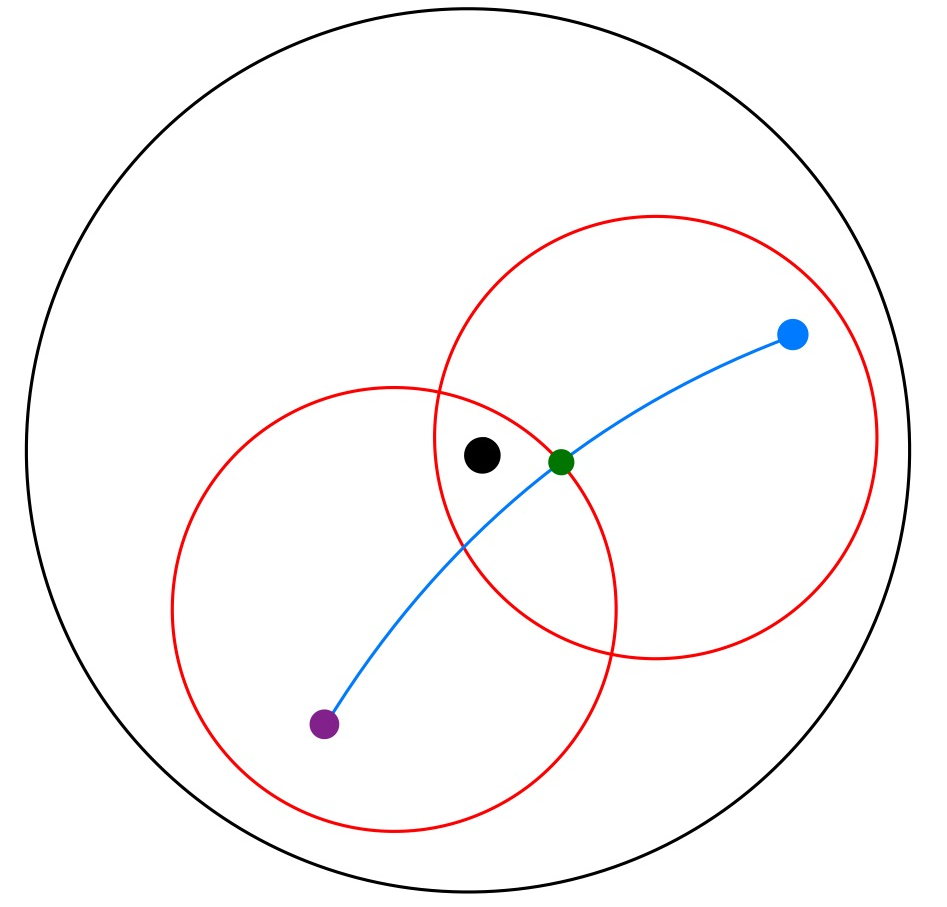}};
\node at (0.6,0.6) {$\D$};
\node at (12,0.6) {$\D$};
\node at (8.6,0.8)  {$a$};
\node at (11.5,3.2)  {$b$};
\node at (10.6,2.64)  {$a_{R_3}$};
\node at (8.7,2.25)  {\begin{tiny}$f(z_{i,R_3})$\end{tiny}};
\draw[->,line width=0.6pt] (8.9,2.36)->(9.44,2.74);
\node at (7.7,3.2)  {\begin{tiny}$B_{\D}(a,R_3)$\end{tiny}};
\node at (10.36,4.45)  {\begin{tiny}$B_{\D}(b,L+C-R_3)$\end{tiny}};
\node at (1.14,2.7) {$z$};
\node at (4.04,4.94) {$x_1$};
\node at (5.04,2.74) {$x_i$};
\node at (3.94,0.6) {$x_k$};
\node at (3.4,2.5) {\begin{tiny}$z_{i,R_3}$\end{tiny}};
\node at (4.3,2) {\begin{tiny}$B_{\D}(z_{i,R_3},R_2)$\end{tiny}};
\node at (4.3,3.5) {\begin{tiny}$f^{-1}(a_{R_3})$\end{tiny}};
\draw[->,line width=0.6pt] (4,3.3)->(3.42,3.05);
\end{tikzpicture}
\caption{The hyperbolic disks $B_{\D}(z_{i,R_3},R_2)$ are disjoint, and each of them contains an $f$-preimage of $a_{R_3}$.}
\label{crit_loc_fig}
\end{figure}
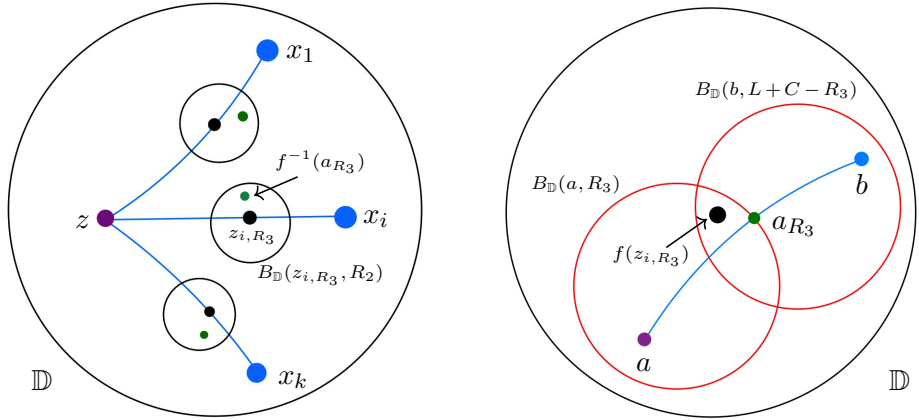

Since $d_{\D}(z_{i,t}, x_i) \leq L+C-t$, by Schwarz lemma, the image 
$$
f(z_{i,t})\in B_{\D}(a, t) \cap B_{\D}(b, L+C-t).
$$
Thus, there exists a constant $R_1 = R_1(C)$ so that $d_{\D}(f(z_{i,t}), a_t) \leq R_1$.

Therefore, by \cite[Corollary 10.3]{McM09}\footnote{The corollary in \cite{McM09} is stated for holomorphic maps, but the anti-holomorphic analogue follows immediately by post-composing with $\bar z$.}, there exists a constant $R_2 = R_2(R_1,d)$ so that 
$$
d_{\D}(z_{i,t}, f^{-1}(a_t)) \leq R_2.
$$

Since the angles $\angle x_i z x_j$ are bounded from below, there exists a constant $R_3 = R_3(\theta)$ so that the balls $B_{\D}(z_{i,t}, R_2)$ are disjoint for all $t\geq R_3$.
Thus there are $k$ different preimages of $a_{R_3}$ (i.e., the point on $[a,b]$ at hyperbolic distance $R_3$ from $a$) in the ball $B_{\D}(z, R_2+R_3)$ (see Figure~\ref{crit_loc_fig}).

Let $U$ be the component of $f^{-1}(B_{\D}(a, R_2+R_3))$ that contains $z$.
By the Schwarz Lemma, $B_{\D}(z, R_2+R_3) \subseteq U$.
Thus, the degree of the map $f: U \longrightarrow B_{\D}(a, R_2+R_3)$ is at least $k$.

We claim that there exists a constant $R = R(R_2, R_3, d)=R(C, \theta, d)$ so that $U \subset B_{\D}(z, R)$. 
Indeed, let $\gamma \subset U$ be a simple curve with $\diam (\gamma) = \zeta$. After removing $O(1)$ neighborhoods of the $d-1$ critical points, we can find  a simple curve $\gamma' \subset U$ with $\diam (\gamma') \geq \zeta/d - O(1)$ so that there are two points $a,b \in\gamma'$ with $d_\D(a,b) = \diam (\gamma')$ and the geodesic $[a,b]$ connecting $a,b$ is away from the critical points.
Then,
Theorem \ref{thm:almostisometry} guarantees that its image satisfies $\diam(f(\gamma')) \geq  \zeta/d - O(1)$.
Since $f(\gamma') \subset B_{\D}(a, R_2+R_3)$, the claim follows.

Since $f:\D \longrightarrow \D$ is proper, $U$ is simply connected. 
Therefore, there are at least $k-1$ critical points in $U \subset B_{\D}(z, R)$, proving the claim.
\end{proof}

\subsection{Induction step for realization}

Let $(\mathcal{T}, p, d_\mathcal{T})$ be a pointed metric $(d+2)$-ended ribbon tree.
Let $(\widetilde{\mathcal{T}}, p, d_{\widetilde{\mathcal{T}}})$ be a reduction of $(\mathcal{T}, p, d_\mathcal{T})$ with the  regluing point $w \in \widetilde{\mathcal{T}}$.

Since $(\widetilde{\mathcal{T}}, p, d_{\widetilde{\mathcal{T}}})$ is a pointed metric $(d+1)$-ended ribbon tree, by induction hypothesis, it is realizable by some family $\widetilde f_s\in \BP^-_d(\widetilde K)$.
Let 
$$
\widetilde{\phi}_s: (\Int(\widetilde{\mathcal{T}}), p, d_{\widetilde{\mathcal{T}}}^s) \longrightarrow (\D, 0, d_{\D})
$$
be the corresponding $\widetilde M$-quasi-isometries.

We also assume the following technical and auxiliary hypothesis.
We remark that the hypothesis gives a stronger version of Theorem \ref{thm:cvr}, but we drop it from the statement as it is too technical.

\begin{ih}\label{ih:m}
Let $v \in \beta(\widetilde{\mathcal{T}})\setminus\{p\}$ be a branch point.
Then there exist $\widetilde C, \widetilde \theta > 0$, and $\widetilde{\tau}(v):=\val(v)-1$ zeros $z_{1,s} = 0, z_{2,s}, \cdots, z_{\widetilde{\tau}(v), s}$ of $\widetilde f_s$ such that for all $s$,
\begin{itemize}
\item $d_{\D}(z_{i,s}, \widetilde\phi_s(v)) \leq d_{\D}(0, \widetilde f_s(\widetilde\phi_s(v)))+ \widetilde C$ for all $i$; and
\item $\angle z_{i,s} \widetilde\phi_s(v) z_{j, s} \geq \widetilde \theta$ for all $i \neq j$.
\end{itemize}
\end{ih}

\begin{rmk}
Note that by Schwarz lemma, $d_{\D}(0, \widetilde{f}_s(\widetilde{\phi}_s(v))) \leq d_{\D}(z_{i,s}, \widetilde{\phi}_s(v))$, and the base case $f_s(z) = \bar z^2$ trivially satisfies this auxiliary hypothesis. 
\end{rmk}

We consider two cases.

\noindent {\bf Case 1: $w$ is a vertex of $\widetilde{\mathcal{T}}$.}
Let $z_{1,s} = 0, z_{2,s}, \cdots, z_{\widetilde{\tau}(w), s}$ be the zeros of $\widetilde f_s$ associated with the branch point $\widetilde{\phi}_s(w)$ of the quasi-fixed tree $\widetilde{\phi}_s(\widetilde{\mathcal{T}})$ as in IH \ref{ih:m}.
We continuously choose a point $z_s$ so that
\begin{itemize}
\item $d_{\D}(z_{s}, \widetilde\phi_s(w)) = d_{\D}(0, \widetilde f_s(\widetilde\phi_s(w)))$; and
\item $\angle z_{s} \widetilde\phi_s(w) z_{i, s} \geq \theta$ for all $i$.
\end{itemize}
Here we can take $\theta = \widetilde \theta/3$ for example. 

Let $\hat{z}_s = \frac{z_s}{|z_s|} \in \mathbb{S}^1$, and $A_s(z) = \frac{\bar z-z_s}{1-\overline{z_s}\bar z}$.
We consider the family
$$
f_s(z) = \frac{-1}{\hat{z}_s} A_s(z) \widetilde f_s(z),
$$
which is conjugate to a map in $\BP^-_{d+1}$ by some rotation. 
Abusing notation, we shall denote by $f_s$ its representative in $\BP^-_{d+1}$ in the following.

\begin{lem}
The family $f_s$ realizes $(\mathcal{T}, p, d_\mathcal{T})$.
\end{lem}
\begin{proof}
Since $w$ is a vertex of $\widetilde{\mathcal{T}}$, we have $\mathcal{T}^c = \widetilde{\mathcal{T}}^c$.
We define 
$$
\phi_s := \widetilde \phi_s\ \textrm{on}\ \mathcal{T}^c.
$$
Let $v\in \beta(\mathcal{T})$.
\vspace{1mm}

\noindent \textbf{Claim.} There exists a constant $K_1$ with
$$
d_{\D}(f_s( \phi_s(v)), \widetilde f_s(\phi_s(v))) \leq K_1.
$$
\begin{proof}[Proof of claim]
Denote $\delta_s:= 1-|z_s|$ and $\rho_s:= d_{\D}(0, z_s)$.
Then, 

\begin{align*}
\rho_s &= d_{\D}(0,\widetilde{\phi}_s(w)) + d_{\D}(\widetilde{\phi}_s(w),z_s) + O(1)\hspace{2.54cm}  \textrm{(by Lemma~\ref{lem:alb})}\\
&= d_{\D}(0,\widetilde{\phi}_s(w)) +  d_{\D}(0, \widetilde f_s(\widetilde\phi_s(w))) + O(1) \hspace{2.04cm} \textrm{(by choice of $z_s$)}\\
&=2 d_{\D}(0, \widetilde\phi_s(w)) + O(1)\hspace{1.72cm} \textrm{(since $\widetilde{\phi}_s(w)$ is quasi-fixed under $\widetilde{f}_s$)} \\
&= 2 s r(w) + O(1)\hspace{2.64cm} \textrm{(since each $\widetilde{\phi}_s$ is an $M$-quasi-isometry)}.
\end{align*}
Since $w$ is a farthest vertex, $r(v)\leq r(w)$.
Since $\phi_s$ is a quasi-isometry, 
$$
d_{\D}(0, \phi_s(v))  =s r(v)+O(1) \leq \frac{1}{2}\rho_s + O(1).
$$
Therefore, $\phi_s(v) \in B(0, 1-K_2\delta_s^{\frac{1}{2}})$ for some constant $K_2$ by Lemma \ref{lem:he}.

By Lemma \ref{lem:em}, there exists some constant $K_3$ such that the error term satisfies
\begin{align}\label{eqn:del}
|f_s(\phi_s(v)) - \widetilde f_s(\phi_s(v))| = |\widetilde f_s(\phi_s(v))||A_s(\phi_s(v)) + \hat{z}_s| \leq K_3 \delta_s^{1/2}.
\end{align}
Since $\phi_s(v)$ is quasi-fixed by $\widetilde f_s$,
$$
d_{\D}(0, \widetilde f_s(\phi_s(v))) = s r(v) + O(1) \leq \frac{\rho_s}{2} + O(1).
$$
Since $|\frac{-1}{\hat{z}_s} A_s(z)| < 1$ for $z\in \D$,
$$
d_{\D}(0, f_s (\phi_s(v))) \leq d_{\D}(0, \widetilde f_s(\phi_s(v))) \leq \frac{\rho_s}{2} + O(1).
$$
Thus, by Lemma \ref{lem:hme}, there exists some constant $K_4$ so that the hyperbolic metric at $z\in[f_s (\phi_s(v)), \widetilde f_s (\phi_s(v))]$ satisfies
\begin{align}\label{eqn:hypme}
\rho_{\D}(z) |dz| \leq \frac{K_4}{\delta_s^{1/2}}|dz|.
\end{align}
Therefore, by equation \ref{eqn:del} and equation \ref{eqn:hypme}, we conclude that
$$
d_{\D}(f_s( \phi_s(v)), \widetilde f_s(\phi_s(v))) \leq \frac{K_4}{\delta_s^{1/2}} \cdot K_3 \delta_s^{1/2} \leq K_1,
$$
for some constant $K_1$.
\end{proof}

Since $\phi_s(v)$ is quasi-fixed by $\widetilde f_s$ for any vertex $v\in \beta(\mathcal{T})$, $\phi_s(v)$ is quasi-fixed by $f_s$ by the claim.

By IH \ref{ih:m} and our construction, there are $\tau(v):=\val_{\mathcal{T}}(v)-1$ zeros $z_{1,s}, \cdots, z_{\tau(v), s}$ of $f_s$ (here $\val_{\mathcal{T}}(v)$ stands for the valence of $v$ in $\mathcal{T}$) such that for all $s$,
\begin{itemize}
\item $d_{\D}(z_{i,s}, \phi_s(v)) \leq d_{\D}(0, f_s(\phi_s(v)))+ C$ for all $i$; and
\item $\angle z_{i,s} \phi_s(v) z_{j, s} \geq \theta$ for all $i \neq j$.
\end{itemize}
Indeed, for $v\neq w$, since $\val_{\widetilde{\mathcal{T}}}(v)=\val_{\mathcal{T}}(v)$, the required zeros of $f_s$ are simply the zeros of $\widetilde{f}_s$ associated with $\widetilde{\phi}_s(v)=\phi_s(v)$. For $v=w$, first note that $\val_{\widetilde{\mathcal{T}}}(v)+1=\val_{\mathcal{T}}(v)$. Hence, all but one required zeros of $f_s$ are given by the zeros of $\widetilde{f}_s$ associated with $\widetilde{\phi}_s(w)=\phi_s(w)$, while the remaining one is $z_s$.
Thus, by Lemma \ref{lem:cl}, there exists a constant $R$ so that $B_{\D}(\phi_s(v), R)$ contains at least $\tau(v) - 1=\val_{\mathcal{T}}(v)-2$ critical points of $f_s$.

Since $\mathcal{T}$ is a $(d+2)$-ended ribbon tree, we have 
$$
\sum_{v\in \beta(\mathcal{T})} (\val_{\mathcal{T}}(v) - 2) = d.
$$
Therefore, by counting, every critical point of $f_s$ is within hyperbolic distance $R$ from $\phi_s(\beta(\mathcal{T}))$.

Hence, if $c_s$ is a critical point of $f_s$ with $d_{\D}(c_s, \phi_s(v)) \leq R$, then by Schwarz lemma,
$$
d_{\D}(f_s(c_s), f_s(\phi_s(v))) \leq d_{\D}(c_s, \phi_s(v)) \leq R.
$$
Since $\phi_s(v)$ is quasi-fixed, so is $c_s$.
Therefore, there exists a constant $K$ so that $f_s\in \BP^-_{d+1}(K)$.

By the same argument as in Theorem \ref{thm:qit}, we can attach the $d+1$ ends to $\phi_s(\mathcal{T}^c)$.
We can extend the map $\phi_s$ to $\mathcal{T}$ by isometry on each end.
Conjugating $f_s$ with a rotation, we may assume that $\phi_s$ respects the markings.
Since $\phi_s(\mathcal{T}) \subset \chull C(f_s) \cup F(f_s)$, we conclude that $\phi_s(\mathcal{T})$ is quasi-fixed by Proposition \ref{cor:ch}.

Since the fixed points of the rescaling limit $g:= \lim M_{v,s}^{-1}\circ f_s\circ M_{v,s}$ (where $M_{v,s}\in\textrm{Aut}(\D)$ sends $0$ to $\phi_s(v)$) correspond to the edges incident at $v$,
the angles between different edges of $\phi_s(\mathcal{T})$ are bounded from below.
Thus $\phi_s$ is a quasi-isometry.
It can also be easily checked that $\phi_s$ satisfies all the other properties and IH \ref{ih:m}.
\end{proof}

\noindent \textbf{Case 2: $w$ is not a vertex of $\widetilde{\mathcal{T}}$.}
Recall that in this case, $w$ is the regluing point in $\widetilde{\mathcal{T}}$ with valence $2$ and $\widetilde{\phi}_s(w)$ is quasi-fixed by $\tilde f_s$.
We continuously choose a point $z_s$ so that
\begin{itemize}
\item $z_s$ is on the geodesic ray from $0$ to $\tilde \phi_s(w)$; and
\item $d_\D(z_s, \tilde \phi_s(w)) = d_\D(0, \tilde f_s(\tilde \phi_s(w)))$.
\end{itemize}
The proof now proceeds almost identically to the previous case. 
Therefore, by induction, we conclude the proof of Theorem \ref{thm:cvr}.

\subsection{Degree one rescaling limits}\label{sec:d1}
In this subsection, we shall briefly discuss the necessary minor modifications when the rescaling limit at the origin has degree $1$.

We say that $(\mathcal{T}, p, d_\mathcal{T})$ is an \emph{extended pointed metric $d+1$-ended ribbon tree} if
\begin{itemize}
\item $p$ has valence $2$;
\item the tree $\widetilde{\mathcal{T}}$ obtained by disregarding $p$ as a vertex is a $d+1$-ended ribbon tree equipped with the linear extension of the modified edge metric.
\end{itemize}
With the same notion of scaling metrics $d_\mathcal{T}^s$ as before, we define realization of extended pointed metric $(d+1)$-ended ribbon trees as follows
\begin{defn}\label{defn:rpmte}
An extended pointed metric $(d+1)$-ended ribbon tree $(\mathcal{T}, p, d_\mathcal{T})$ is said to be {\em realizable} if 
there exist $K, M, R > 0$, a family $f_s \in \BP^-_d(K)$ with $s \in (0, \infty)$, and a family of $M$-quasi-isometric embeddings
$$
\phi_s: (\Int(\mathcal{T}), p, d_\mathcal{T}^s) \longrightarrow (\D, 0, d_{\D})
$$
so that
\begin{itemize}
\item (Ends approximating.) The map $\phi_s$ extends to a continuous map $\phi_s: \mathcal{T} \longrightarrow \D \cup \mathbb{S}^1$ which gives a bijection between $\epsilon (\mathcal{T})$ and $F(f_s)$ respecting the markings.

\item (Geodesic edges.) If $E = [v, w] \subset \mathcal{T}$ is an edge, then $\phi_s(E)$ is the hyperbolic geodesic connecting $\phi_s(v)$ and $\phi_s(w)$.

\item (Critically approximating.)  If $v\in \beta(\mathcal{T})$, then there are exactly $\val(v) - 2$ critical points counted with multiplicities in $B_{\D}(\phi_s(v), R)$.

\item (Quasi-fixed.) If $x\in \Int(\mathcal{T})$, then 
$$
d_{\D}(f_s(\phi_s(x)), \phi_s(x)) \leq M \text{ for all } s.
$$
\end{itemize}
\end{defn}
We remark that the only deviation here from Definition \ref{defn:rpmt} is that the superattracting condition is omitted.
With this definition, we also have
\begin{theorem}\label{thm:cvre}
For $d\geq 2$, every extended pointed metric $(d+1)$-ended ribbon tree $(\mathcal{T}, p, d_\mathcal{T})$ is realizable.
\end{theorem}
We remark that the proof by induction is almost identical as in Theorem \ref{thm:cvr}.
The only modification is that the base case starts with $\bar z$ instead of $\bar z^2$.

In most situations, the rescaling limit at a vertex $v \neq p$ is an anti-Blaschke product with an attracting fixed point on $\mathbb{S}^1$.
However, if the extended pointed $(d+1)$-ended ribbon tree $(\mathcal{T}, p)$ has three {\em core vertices}, i.e, three vertices that are not endpoints, we can ensure that the rescaling limits are parabolic anti-Blaschke products.  Moreover, we can guarantee real-symmetry of these rescaling limits, a fact that will be essential for some of the applications in Appendix \ref{sec:ve}.

\begin{prop}\label{prop:rabl}
Let $(\mathcal{T}, p, d_\mathcal{T})$ be an extended pointed metric $(d+1)$-ended ribbon tree with three core vertices $v_1, v_2, p$ and $p \in (v_1, v_2)$.
Then there exists a family $f_t$ realizing $(\mathcal{T}, p, d_\mathcal{T})$ so that the rescaling limits at $v_1$ and at $v_2$ are conjugate to parabolic unicritical real anti-Blaschke products.
\end{prop}
\begin{proof}
By a classical result of Heins (see \cite{Heins62} and \cite{Zakeri98}), we can construct $f_t$ with two real critical points $\pm r_t$ of multiplicities $\val(v_i)-2, i=1,2$, where $d_\D(0, r_t) = t$.
Post composing with an element $M_s \in \Isom(\D)$, we may also assume $f_t$ is real and has fixed points at $0$ and $\pm 1$.
It can be verified that $f_t$ realizes $(\mathcal{T}, p, d_\mathcal{T})$, and the rescaling limits at $v_1$ and at $v_2$ are parabolic unicritical real anti-Blaschke products.
\end{proof}

\begin{rmk}\label{rmk:ext}
We remark that the same construction also works for an extended pointed metric $(d+1)$-ended ribbon tree $(\mathcal{T}, p, d_\mathcal{T})$ with two core vertices $v_1, p$ where $p$ has valence $2$.
Indeed, we can simply take the family of unicritical real anti-Blaschke products in $\BP^-_d$ converging to the parabolic one.
\end{rmk}

\section{Boundedness and mutual interaction of deformation spaces}\label{sec:cad}

In the previous two sections, we have shown that degenerations in $\BP_d^-(K)$ are parametrized by $(d+1)$-ended ribbon trees.
Since the hyperbolic component $\mathcal{H}_\Gamma$ is essentially parametrized by a product of marked anti-Blaschke products (see \S \ref{subsec:pds}), these trees give combinatorial parametrizations of degenerations in the pared deformation space $\mathcal{H}_\Gamma(K)$.

The key objective of this section is to prove a combinatorial characterization for convergence in the pared deformation space $\mathcal{H}_\Gamma (K)$ (see Theorem \ref{thm:cparm}).
Both Theorem \ref{thm:ab} and Theorem \ref{thm:main} follow from this characterization.

\subsection{Pared deformation space}\label{subsec:pds}

Let $\mathcal{R}$ be a critically fixed anti-rational map of degree $d$. Let $c_1,\cdots, c_m$ be the distinct critical points of $\mathcal{R}$, and the local degree of $\mathcal{R}$ at $c_i$ be $d_i$ ($i=1,\cdots, m$). Since $\mathcal{R}$ has $(2d-2)$ critical points counted with multiplicity, we have that 
$$
\displaystyle\sum_{i=1}^m (d_i-1)=2d-2\ \implies\ \displaystyle\sum_{i=1}^m d_i=2d+m-2.
$$

Suppose that $U_i$ is the invariant Fatou component containing $c_i$ ($i=1,\cdots, m$). 
Then, $U_i$ is a simply connected domain such that $\mathcal{R}\vert_{U_i}$ is conformally conjugate to $\overline{z}^{d_i}\vert_\D$ \cite[Theorem~9.3]{Milnor06}. 
This defines internal rays in $U_i$, and $\mathcal{R}$ maps the internal ray at angle $\theta\in\R/\Z$ to the one at angle $-d_i\theta\in\R/\Z$. It follows that there are $(d_i+1)$ fixed internal rays in $U_i$. 
A straightforward adaptation of \cite[Theorem~18.10]{Milnor06} now implies that all these fixed internal rays land at repelling fixed points on $\partial U_i$.

We define the \emph{Tischler graph} $\mathscr{T}$ of $\mathcal{R}$ as the union of the closures of the fixed internal rays of $\mathcal{R}$. Since the fixed internal rays of $\mathcal{R}$ land pairwise at the repelling fixed points, we disregard the repelling fixed points of $\mathcal{R}$ from the vertex set $\mathcal{V}(\mathscr{T})$ of the Tischler graph. Thus by our convention, $\mathcal{V}(\mathscr{T})$ is the set of critical points of $\mathcal{R}$ (cf. \cite[\S 4]{LLM22}).

In \cite{LLM22}, it was proved that a plane graph $\mathscr{T}$ is the Tischler graph of a critically fixed anti-rational map if and only if its planar dual $\Gamma$ is simple and $2$-connected.

\subsection*{Hyperbolic components of critically fixed anti-rational maps}

Let $\Rat_d^-(\C)$ be the space of degree $d$ anti-rational maps, and 
$$
\mathcal{M}_d^- = \Rat_d^-(\C) / \PSL_2(\C).
$$
Let $\mathcal{R}_\Gamma \in \Rat_d^-(\C)$ be a critically fixed anti-rational map whose Tischler graph $\mathscr{T}$ is dual to $\Gamma$.
We denote the component of hyperbolic maps containing $[\mathcal{R}_\Gamma]$ by $\mathcal{H}_\Gamma \subset \mathcal{M}^-_d$.

By the anti-holomorphic version of \cite[Proposition~6.9]{McM88}, the Julia set dynamics of any $[R]\in \mathcal{H}_\Gamma$ is topologically conjugate to the Julia set dynamics of $\mathcal{R}_\Gamma$ such that the topological conjugacy extends to an orientation-preserving homeomorphism of $\widehat{\C}$.

In order to study degenerations in $\mathcal{H}_\Gamma$, we need to relate it to suitable spaces of marked anti-Blaschke products. This necessitates the following definition of marked hyperbolic components. 
Recall from the introduction that $\mathcal{H}_{\Gamma, Rat} \subset \Rat_d^-(\C)$ denotes the preimage of $\mathcal{H}_\Gamma\subset \mathcal{M}_d^-$ under the projection map $\Rat_d^-(\C) \longrightarrow \mathcal{M}_d^-$. We note that connectedness of $\PSL_2(\C)$ implies connectedness of $\mathcal{H}_{\Gamma, Rat}$.

Let $R\in \mathcal{H}_{\Gamma, Rat}$. We label the critical Fatou components as $U_1,\cdots, U_m$, where $R\vert_{U_i}$ has degree $d_i$, $i=1, \cdots, m$. We say that $q$ is a {\em boundary marking} for $R$, if it assigns to each critical Fatou component $U_i$ a fixed point $q(U_i)$ on the ideal boundary $\mathbb{S}^1=I(U_i)$.
In other words, $q$ is an $m$-tuple $(x_1,\cdots, x_m)$ where $x_i$ is a fixed point on $I(U_i)$.

We consider the space of pairs
$$
\mathcal{S}_\Gamma=\{(R,q): R\in \mathcal{H}_{\Gamma, Rat}, \text{and $q$ is a boundary marking for}\ R\},
$$
and let $\PSL_2(\C)$ act on $\mathcal{S}_\Gamma$ by
$$
(R, (x_1,\cdots, x_m)) \mapsto (M\circ R\circ M^{-1}, (\widetilde{M}(x_1),\cdots, \widetilde{M}(x_m))),\ M\in \PSL_2(\C),
$$
where $\widetilde{M}$ is the action on the ideal boundaries of $U_i$ induced by $M$.
The quotient of $\mathcal{S}_\Gamma$ under this action will be denoted by $\widetilde{\mathcal{S}}_\Gamma$.

We topologize $\mathcal{S}_\Gamma$ by declaring that a sequence converges in $\mathcal{S}_\Gamma$ if the anti-rational maps converge algebraically and the corresponding boundary markings converge to that of the limiting map.
The space $\widetilde{\mathcal{S}}_\Gamma$ is endowed with the quotient topology under the projection map $\mathcal{S}_\Gamma\longrightarrow \widetilde{\mathcal{S}}_\Gamma$.

Let $\widetilde{\mathcal{H}}_\Gamma$ be a component of $\widetilde{\mathcal{S}}_\Gamma$. Then a similar argument as in the proof of \cite[ Theorem~5.7]{Milnor12} shows that
$$
\widetilde{\mathcal{H}}_\Gamma \cong \BP^-_{d_1} \times \BP^-_{d_2} \times \cdots \times \BP^-_{d_m}.
$$
The \emph{marked hyperbolic component} $\widetilde{\mathcal{H}}_\Gamma$ is a branched cover of $\mathcal{H}_\Gamma$.

\begin{rmk}
The marked hyperbolic component $\widetilde{\mathcal{H}}_\Gamma$ contains a unique equivalence class represented by a marked critically fixed map. This class corresponds to the element $(\overline{z}^{d_1},\cdots,\overline{z}^{d_m})\in \BP^-_{d_1} \times \BP^-_{d_2} \times \cdots \times \BP^-_{d_m}$.
\end{rmk}

Let $K>0$, we define the {\em marked pared deformation space} as
$$
\widetilde{\mathcal{H}}_\Gamma(K) \cong \mathring{\BP}^-_{d_1}(K) \times \mathring{\BP}^-_{d_2}(K) \times \cdots \times \mathring{\BP}^-_{d_m}(K),
$$
(see Definition~\ref{defn:pdfb}) and the {\em pared deformation space} $\mathcal{H}_\Gamma(K)\subset \mathcal{M}^-_d$ as the projection of $\widetilde{\mathcal{H}}_\Gamma(K)$.
A sequence $[R_n] \in  \mathcal{H}_\Gamma$ (respectively, a family $[R_s]$) is called a {\em degeneration} if $[R_n]$ (respectively, $[R_s]$) escapes every compact set of $\mathcal{H}_\Gamma$.

\subsection{Enriched Tischler graph}
Let $\mathscr{T}$ be a Tischler graph of a critically fixed anti-rational map.
\begin{defn}\label{enrichment_def}
We say a graph $\mathscr{T}^{En}$ is an {\em enrichment} of $\mathscr{T}$ if $\mathscr{T}^{En}$ is obtained from $\mathscr{T}$ by replacing each vertex $v$ of valence $\val_{\mathscr{T}}(v)$ with a $\val_{\mathscr{T}}(v)$-ended ribbon tree.
Similar to the convention we adopted for Tischler graphs, we define the vertex set for an enrichment as the set of branch points in $\mathscr{T}^{En}$.
\end{defn}

Enrichments of a Tischler graph arise naturally as we study degenerations in the pared deformation space $\mathcal{H}_\Gamma(K)$.
Let $[R_n]$ be a sequence in $\mathcal{H}_\Gamma(K)$, which we may lift to a sequence in $\widetilde{\mathcal{H}}_\Gamma(K)$.
Then we get $m$ sequences of anti-Blaschke products $f_{1,n}, \cdots, f_{m,n}$.
After passing to a subsequence, we assume that $\mathcal{T}_{1,n},\cdots, \mathcal{T}_{m,n}$ are the quasi-fixed trees for these sequences of anti-Blaschke products.

Using the uniformization maps, we may regard $\mathcal{T}_{i,n}$ as a subset of the corresponding critical Fatou component $U_{i,n}$ of $R_n$ where the marked (ideal) boundary fixed point of $U_{i,n}$ corresponds to the marked endpoint of the tree $\mathcal{T}_{i,n}$. We let 
$$
\mathscr{T}^{En}_n = \bigcup_{i=1}^m \mathcal{T}_{i,n}.
$$
Recall that for each fixed $i$, the trees $\mathcal{T}_{i,n}$ are isomorphic as endpoint-marked plane graphs for all $n$. We denote the plane isomorphism class of $\mathcal{T}_{i,n}$ by $\mathcal{T}_i$, and the plane isomorphism class of $\mathscr{T}^{En}_n$ by $\mathscr{T}^{En}$. 

Then by construction, it is clear that $\mathscr{T}^{En}$ is an enrichment of $\mathscr{T}$.
We call $\mathscr{T}^{En}$ an {\em enriched Tischler graph} corresponding to the sequence $[R_n]$. 
Note that the isomorphism class of $\mathscr{T}^{En}$ does not depend on the lift of $[R_n]$, but may depend on the choice of subsequence. In what follows, we will tacitly assume that such a subsequence has already been chosen, and refer to $\mathscr{T}^{En}$ as \emph{the} enriched Tischler graph corresponding to $[R_n]$.

Similarly, if $[R_s]$ is a family in $\mathcal{H}_\Gamma(K)$, we can lift it to a family in $\widetilde{\mathcal{H}}_\Gamma(K)$, and get $m$ families of anti-Blaschke products $f_{1,s}, \cdots, f_{m,s}$.
Assume that each $f_{i,s}$ realizes a $(d_i+1)$-ended ribbon tree $\mathcal{T}_i$. 
Then we get a family of graphs 
$$
\mathscr{T}^{En}_s = \bigcup_{i=1}^m \mathcal{T}_{i,s}
$$
that is isomorphic as plane graphs to some graph $\mathscr{T}^{En}$.
We call this graph the {\em enriched Tischler graph} corresponding to the family $[R_s]$.

\begin{rmk}
We note that if $\mathcal{T}_{i,n}$ and $\mathcal{T}_{i,n}'$ are isomorphic as plane trees but not as (endpoint-)marked plane trees, then the above enrichment procedure applied to  $\mathcal{T}_{i,n}$ and $\mathcal{T}_{i,n}'$ may give rise to non-isomorphic enrichments (see Figure~\ref{fig:NSG}.
\end{rmk}

By Theorem \ref{thm:rdtr}, we have the following.

\begin{prop}\label{prop:enriched_tischler_realized}
Let $\mathscr{T}$ be the Tischler graph of a critically fixed anti-rational map $\mathcal{R}_\Gamma$.
Then for all large $K$, any enrichment $\mathscr{T}^{En}$ of $\mathscr{T}$ arises as the enriched Tischler graph for some sequence $[R_n] \in \mathcal{H}_\Gamma(K)$ (or some family $[R_s] \in \mathcal{H}_\Gamma(K)$).
\end{prop}

We now study the effect of enrichment on dual graphs. We need the following notions of pseudo-simple graphs and domination.

\begin{defn}
\noindent\begin{enumerate}\upshape
\item We say that a bigon in a plane graph is \emph{topologically trivial} if the two edges of the bigon are path-homotopic rel. the vertices of the graph.
A plane graph $\Gamma$ is called {\em pseudo-simple} if 
\begin{itemize}
\item $\Gamma$ has no self-loop, and
\item $\Gamma$ has no topologically trivial bigon.
\end{itemize}
\item Given two pseudo-simple plane graphs $\Gamma, \Gamma'$ with the same number of vertices, we say that $\Gamma'$ {\em dominates} $\Gamma$, denoted by $\Gamma' \geq \Gamma$, if there exists an embedding $i:\Gamma \xhookrightarrow{} \Gamma'$ as plane graphs. 
\end{enumerate}
\end{defn}

Dual graphs of Tischler graphs and their enrichments are related by the following lemma.

\begin{lem}\label{lem:detdt}
Let $\mathscr{T}$ be the Tischler graph of a critically fixed anti-rational map and let $\mathscr{T}^{En}$ be an enrichment of $\mathscr{T}$.
Let $\Gamma^{En}$ and $\Gamma$ be their planar dual graphs respectively.
Then 
\begin{enumerate}	
\item $\Gamma^{En}$ is pseudo-simple, and
\item $\Gamma^{En}$ dominates $\Gamma$.
\end{enumerate}
\end{lem}

\begin{proof}
1) By our convention, each vertex of $\mathscr{T}^{En}$ has valence at least $3$. Moreover, since each face of $\mathscr{T}$ is a Jordan domain \cite[Lemma~4.5]{LLM22} and every vertex of $\mathscr{T}$ is blown up to a tree in $\mathscr{T}^{En}$, it follows that every face of $\mathscr{T}^{En}$ is also a Jordan domain. Hence, every edge of $\mathscr{T}^{En}$ is contained in the boundary of at least two faces and two edges lying on the common boundary of two faces of $\mathscr{T}^{En}$ cannot be adjacent. It follows that the dual graph $\Gamma^{En}$ contains no self-loop or topologically trivial bigon; i.e., $\Gamma^{En}$ is pseudo-simple.

2) Since each vertex of $\mathscr{T}$ gets split into finitely many vertices in $\mathscr{T}^{En}$, the vertices of $\mathscr{T}^{En}$ can be grouped into several clusters, where each cluster corresponds to a vertex of $\mathscr{T}$.

The edges of $\mathscr{T}^{En}$ can be divided into two categories:
we say that an edge of $\mathscr{T}^{En}$ is a {\em crossing} edge if it connects vertices in two different Fatou components; and a {\em non-crossing} edge otherwise.
Note that the crossing edges connect vertices of $\mathscr{T}^{En}$ lying in different clusters, and hence are in one-to-one correspondence with the edges of $\mathscr{T}$.
On the other hand, each non-crossing edge connects two (not necessarily distinct) vertices of $\mathscr{T}^{En}$ lying in the same cluster (see Figure~\ref{fig:NSG}). Thus, each non-crossing edge of $\mathscr{T}^{En}$ corresponds to a vertex of $\mathscr{T}$.

It is now easy to see that the faces of $\mathscr{T}^{En}$ are in one-to-one correspondence with the faces of $\mathscr{T}$, and hence the dual graphs $\Gamma, \Gamma^{En}$ have the same number of vertices.
Moreover, two faces $A, B$ of $\mathscr{T}^{En}$ meet along a common crossing edge on their boundaries if and only if the corresponding faces in $\mathscr{T}$ meet along the corresponding edge on their boundaries.
This gives rise to the desired embedding of $\Gamma$ into $\Gamma^{En}$ as plane graphs (see Figure~\ref{fig:NSG}).
\end{proof}
\begin{figure}[h!]
\captionsetup{width=1.2\linewidth}
\begin{tikzpicture}

\filldraw[red] (0,0) circle (3.2pt);
\filldraw[red] (4,0) circle (3.2pt);

\draw [-, PineGreen, line width=0.5pt, out=40,in=140]  (0,0) to (4,0); 
\draw [-, PineGreen, line width=0.5pt, out=320,in=230]  (0,0) to (4,0); 
\draw [-, PineGreen, line width=0.5pt, out=90,in=180]  (0,0) to (2,2); 
\draw [-, PineGreen, line width=0.5pt, out=0,in=90]  (2,2) to (4,0); 
\draw [-, PineGreen, line width=0.5pt, out=270,in=180]  (0,0) to (2,-2); 
\draw [-, PineGreen, line width=0.5pt, out=0,in=270]  (2,-2) to (4,0); 

\filldraw[black] (2,0) circle (2.5pt);
\filldraw[black] (2,1.28) circle (2.5pt);
\filldraw[black] (2,-1.28) circle (2.5pt);
\filldraw[black] (-1.5,0) circle (2.5pt);

\draw [-, black, line width=0.5pt]  (2,0) to (2,1.28); 
\draw [-, black, line width=0.5pt]  (2,0) to (2,-1.28); 
\draw [-, black, line width=0.5pt, out=120, in=25]  (2,1.28) to (-0.2,2.2); 
\draw [-, black, line width=0.5pt, out=205,in=90]  (-0.2,2.2) to (-1.5,0); 
\draw [-, black, line width=0.5pt, out=270,in=160]  (-1.5,0) to (0.2,-2); 
\draw [-, black, line width=0.5pt, out=-20,in=240]  (0.2,-2) to (2,-1.28); 

\draw [->, brown, line width=0.8pt]  (3,2.4) to (4,3.2); 

\filldraw[red] (6,3.5) circle (2.2pt);
\filldraw[red] (6,4.5) circle (2.2pt);
\filldraw[red] (10,3.5) circle (2.2pt);
\filldraw[red] (10,4.5) circle (2.2pt);

\draw [-, red, line width=0.8pt] (6,3.5) to (6,4.5);
\draw [-, red, line width=0.8pt] (6,3.5) to (6.7,2.5);
\draw [-, red, line width=0.8pt] (6,3.5) to (5.3,2.5);
\draw [-, red, line width=0.8pt] (6,4.5) to (6.7,5.4);
\draw [-, red, line width=0.8pt] (6,4.5) to (5.3,5.4);

\draw [-, red, line width=0.8pt] (10,3.5) to (10,4.5);
\draw [-, red, line width=0.8pt] (10,3.5) to (10.7,2.5);
\draw [-, red, line width=0.8pt] (10,3.5) to (9.3,2.5);
\draw [-, red, line width=0.8pt] (10,4.5) to (10.7,5.4);
\draw [-, red, line width=0.8pt] (10,4.5) to (9.3,5.4);

\draw [-, PineGreen, line width=0.5pt, out=40,in=140]  (6.7,5.4) to (9.3,5.4); 
\draw [-, PineGreen, line width=0.5pt, out=320,in=230]  (6.7,2.5) to (9.3,2.5); 
\draw [-, PineGreen, line width=0.5pt, out=230,in=180]  (5.3,2.5) to (8,1); 
\draw [-, PineGreen, line width=0.5pt, out=0,in=-60]  (8,1) to (10.7,2.5);
\draw [-, PineGreen, line width=0.5pt, out=135,in=180]  (5.3,5.4) to (8,7); 
\draw [-, PineGreen, line width=0.5pt, out=0,in=45]  (8,7) to (10.7,5.4);

\filldraw[black] (8,4) circle (2.5pt);
\filldraw[black] (8,6.3) circle (2.5pt);
\filldraw[black] (8,1.56) circle (2.5pt);
\filldraw[black] (4.5,4) circle (2.5pt);

\draw [-, black, line width=0.5pt]  (8,4) to (8,6.3); 
\draw [-, black, line width=0.5pt]  (8,4) to (8,1.56); 
\draw [-, black, line width=0.5pt, out=160, in=25]  (8,6.3) to (5.75,6.2); 
\draw [-, black, line width=0.5pt, out=205,in=90]  (5.75,6.2) to (4.5,4); 
\draw [-, black, line width=0.5pt, out=270,in=160]  (4.5,4) to (6.2,1.66); 
\draw [-, black, line width=0.5pt, out=-20,in=200]  (6.2,1.66) to (8,1.56); 

\draw [dashed, black, line width=0.5pt] (4.5,4) to (8,4);

\draw [dashed, black, line width=0.5pt, out=160, in=205]  (4.5,4) to (5,7.32);
\draw [dashed, black, line width=0.5pt, out=25, in=185]  (5,7.32) to (7.55,8);
\draw [dashed, black, line width=0.5pt, out=5, in=90] (7.55,8) to (11.7,5.6);
\draw [dashed, black, line width=0.5pt, out=-90, in=0] (11.7,5.6) to (8,4);

\draw [->, brown, line width=0.8pt]  (3,-2.4) to (4,-3.2); 

\filldraw[red] (5.4,-4) circle (2.2pt);
\filldraw[red] (6.6,-4) circle (2.2pt);
\filldraw[red] (10,-3.5) circle (2.2pt);
\filldraw[red] (10,-4.5) circle (2.2pt);

\draw [-, red, line width=0.8pt] (5.4,-4) to (6.6,-4);
\draw [-, red, line width=0.8pt] (5.4,-4) to (5,-3);
\draw [-, red, line width=0.8pt] (5.4,-4) to (5,-5);
\draw [-, red, line width=0.8pt] (6.6,-4) to (7,-3);
\draw [-, red, line width=0.8pt] (6.6,-4) to (7,-5);

\draw [-, red, line width=0.8pt] (10,-3.5) to (10,-4.5);
\draw [-, red, line width=0.8pt] (10,-3.5) to (10.7,-2.5);
\draw [-, red, line width=0.8pt] (10,-3.5) to (9.3,-2.5);
\draw [-, red, line width=0.8pt] (10,-4.5) to (10.7,-5.4);
\draw [-, red, line width=0.8pt] (10,-4.5) to (9.3,-5.4);

\draw [-, PineGreen, line width=0.5pt, out=-60,in=220]  (7,-5) to (9.3,-5.4); 
\draw [-, PineGreen, line width=0.5pt, out=60,in=135]  (7,-3) to (9.3,-2.5); 
\draw [-, PineGreen, line width=0.5pt, out=120,in=180]  (5,-3) to (8,-1); 
\draw [-, PineGreen, line width=0.5pt, out=0,in=60]  (8,-1) to (10.7,-2.5);
\draw [-, PineGreen, line width=0.5pt, out=240,in=180]  (5,-5) to (8,-7); 
\draw [-, PineGreen, line width=0.5pt, out=0,in=-45]  (8,-7) to (10.7,-5.4);

\filldraw[black] (8.4,-4) circle (2.5pt);
\filldraw[black] (8.4,-6.28) circle (2.5pt);
\filldraw[black] (8.4,-1.5) circle (2.5pt);
\filldraw[black] (4.5,-4) circle (2.5pt);

\draw [-, black, line width=0.5pt]  (8.4,-4) to (8.4,-6.28); 
\draw [-, black, line width=0.5pt]  (8.4,-4) to (8.4,-1.5); 
\draw [-, black, line width=0.5pt, out=190, in=-30]  (8.4,-6.28) to (5.8,-6); 
\draw [-, black, line width=0.5pt, out=150,in=270]  (5.8,-6) to (4.5,-4); 
\draw [-, black, line width=0.5pt, out=90,in=210]  (4.5,-4) to (5.8,-1.9); 
\draw [-, black, line width=0.5pt, out=30,in=175]  (5.8,-1.9) to (8.4,-1.5); 

\draw [dashed, black, line width=0.5pt, out= 200, in=90] (8.4,-1.5) to (6,-4);
\draw [dashed, black, line width=0.5pt, out= 270, in=160] (6,-4) to (8.4,-6.28);
\draw [dashed, black, line width=0.5pt, out=210, in=150]  (4.5,-4) to (5,-7.32);
\draw [dashed, black, line width=0.5pt, out=-30, in=180]  (5,-7.32) to (7.5,-8);
\draw [dashed, black, line width=0.5pt, out=0, in=270] (7.5,-8) to (11.7,-5.6);
\draw [dashed, black, line width=0.5pt, out=90, in=0] (11.7,-5.6) to (8.4,-4);
\end{tikzpicture}
\caption{A Tischler graph and two enrichments.}
 \label{fig:NSG}
\end{figure}
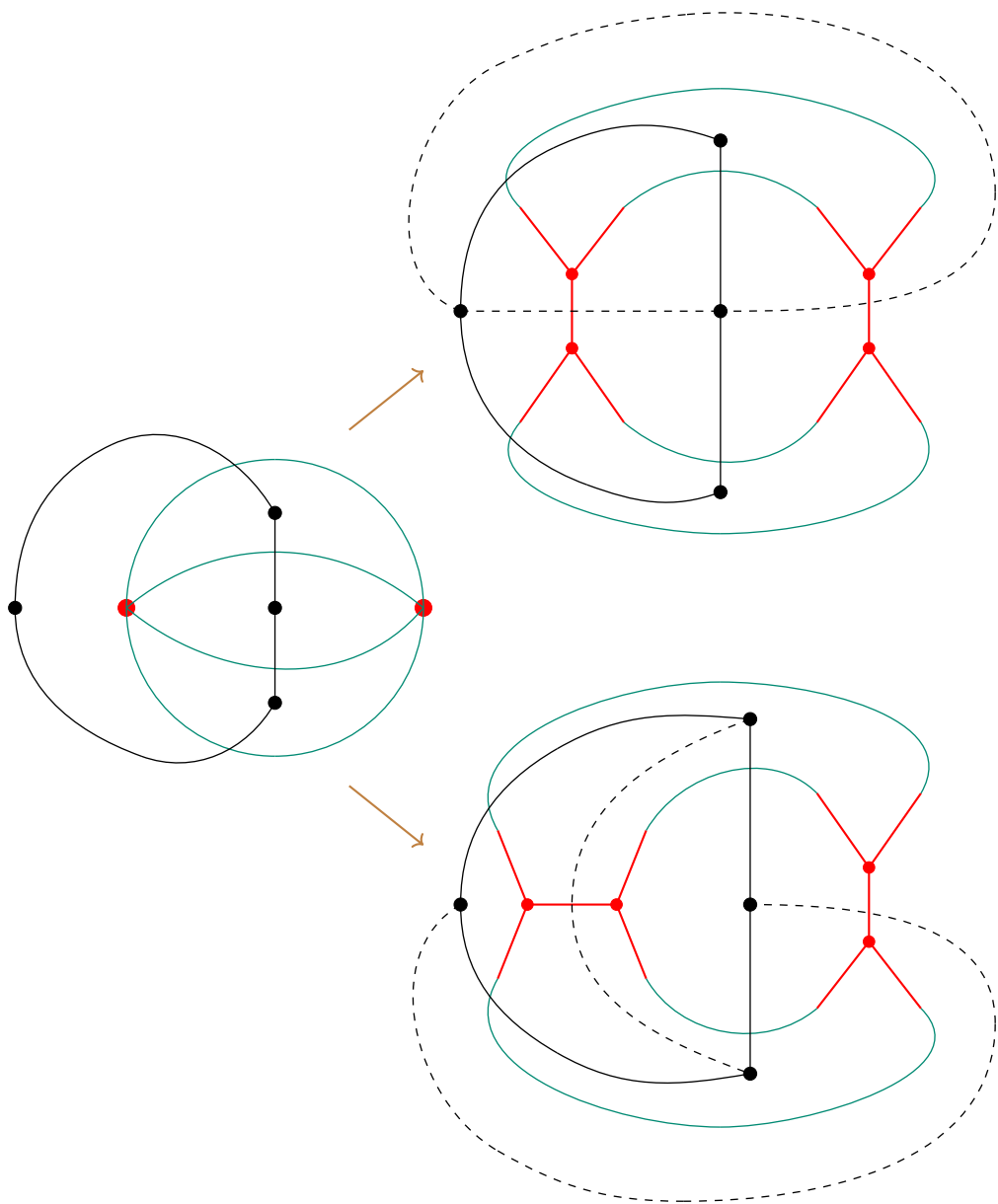
\noindent\textbf{Figure~\ref{fig:NSG}:} The Tischler graph $\mathscr{T}$ (in green) of $\overline{z}^3$ and its planar dual $\Gamma$ (in black) are shown on the left. The vertices of $\mathscr{T}$ are marked in red. Two different enrichments $\mathscr{T}^{En}$ of $\mathscr{T}$ (in green/red) and their planar duals $\Gamma^{En}$ (in black) are shown on the right. The red parts of $\mathscr{T}^{En}$ indicate the $4$-ended quasi-fixed trees that replace the vertices of $\mathscr{T}$. The vertex clusters of $\mathscr{T}^{En}$ correspond $1:1$ to these quasi-fixed trees. The green/red edges are the crossing edges and the red edges are non-crossing edges of $\mathscr{T}^{En}$. The dual graphs $\Gamma^{En}$ dominate $\Gamma$; they are obtained by adding the dashed edges to $\Gamma$. The top enrichment is not admissible since its planar dual contains a topologically non-trivial bigon given by the two dashed edges.

\begin{rmk}
If two faces $A, B$ of $\mathscr{T}^{En}$ share a common non-crossing edge on their boundaries, then the corresponding faces in $\mathscr{T}$ share the corresponding vertex on their boundaries. Thus the non-crossing edges of $\mathscr{T}^{En}$ yield the additional edges in $\Gamma^{En}$.
\end{rmk}

The converse of Lemma~\ref{lem:detdt} is also true and can be proved essentially by reversing the arguments. 

\begin{lem}\label{lem:rdetg}
Let $\mathscr{T}$ be the Tischler graph of a critically fixed anti-rational map, and $\Gamma'$ a pseudo-simple plane graph that dominates $\Gamma:=\mathscr{T}^{\vee}$. Then there exists an enrichment $\mathscr{T}^{En}$ of $\mathscr{T}$ whose dual graph is $\Gamma'$.

Moreover, different embeddings $i: \Gamma \xhookrightarrow{} \Gamma'$ (up to automorphisms of $\Gamma$ and $\Gamma'$) correspond to different enrichments (up to isomorphism of plane graphs).
\end{lem}
\begin{proof}
Let $i: \Gamma\xhookrightarrow{} \Gamma'$ be an embedding.
The embedding $i$ shows that each face of $\Gamma$ is split into finitely many faces in $\Gamma'$, and hence the faces of $\Gamma'$ can be grouped into several clusters, where each cluster corresponds to a face of $\Gamma$. Dualizing this structure, one concludes that the vertices of $\mathscr{T}':=(\Gamma')^{\vee}$ can be grouped into several clusters, where each cluster results from splitting of a vertex of $\mathscr{T}$ (see Figure~\ref{fig:NSG}, also compare the proof of Lemma~\ref{lem:detdt}).

Furthermore, if an edge of $\Gamma'$ corresponds to an edge of $\Gamma$ via the embedding, then the corresponding dual edge in $\mathscr{T}'$ is a `crossing edge'; i.e, it connects two vertices in different clusters. Clearly, such dual edges bijectively correspond to the edges of $\mathscr{T}$ (the planar dual of $\Gamma$). On the other hand, the additional edges in $\Gamma'$ are responsible for splittings of faces of $\Gamma$, and hence the corresponding dual edges in $\mathscr{T}'$ are `non-crossing edges'; i.e, they connect two (not necessarily distinct) vertices in the same cluster. It follows that $\mathscr{T}'$ is obtained from $\mathscr{T}$ by blowing up each vertex $v\in V(\mathscr{T})$ to a $\val_{\mathscr{T}}(v)$-ended tree $\mathcal{T}_v$. As $\Gamma'$ is pseudo-simple, each face of $\Gamma'$ borders on at least three other faces, and hence every vertex in $\mathcal{T}_v$ has valence $\geq 3$. Therefore, $\mathscr{T}'$ is obtained from $\mathscr{T}$ by blowing up each vertex $v\in V(\mathscr{T})$ to a  ribbon tree, which proves that $\mathscr{T}'$ is an enrichment of $\mathscr{T}$.
\end{proof}

We shall now fix an embedding of $\mathscr{T}^{En}$ in $\widehat\C$.
A simple closed curve $\gamma$ in $\widehat\C\setminus \mathcal{V}(\mathscr{T}^{En})$ is said to be {\em essential} if it separates the vertex set of $\mathscr{T}^{En}$.
An essential closed curve $\gamma$ in $\widehat\C\setminus V(\mathscr{T}^{En})$ is said to {\em cut} an edge if the two end-points of the edge lie in two different components of $\widehat\C\setminus\gamma$.

We call an enriched Tischler graph $\mathscr{T}^{En}$ {\em admissible} if its planar dual is simple and $2$-connected (see Figure \ref{fig:NSG} for a non-admissible example). Since the planar dual of a Tischler graph is $2$-connected, Lemma~\ref{lem:detdt} implies that admissibility of $\mathscr{T}^{En}$ is equivalent to requiring that $(\mathscr{T}^{En})^{\vee}$ has no topologically non-trivial bigon.

The following two graph theoretical lemmas will be used later.

\begin{lem}\label{lem:3e}
If $\mathscr{T}^{En}$ is admissible, then any essential simple closed curve $\gamma\subset\widehat\C\setminus \mathcal{V}(\mathscr{T}^{En})$ cuts at least $3$ edges.
\end{lem}
\begin{proof}
Let $\gamma$ be an essential simple closed curve.
Since it is essential, and every vertex has valence at least $3$, $\gamma$ cuts at least $2$ edges.
If $\gamma$ cuts exactly $2$ edges, then it corresponds to a double edge in the dual graph (see Figure~\ref{fig:NSG}).
Thus $\mathscr{T}^{En}$ is not admissible.
\end{proof}

\begin{lem}\label{lem:pg}
Let $\mathscr{T}$ be a Tischler graph.
If the dual $\Gamma$ of $\mathscr{T}$ is a polyhedral graph, then any enrichment of $\mathscr{T}$ is admissible.

On the other hand, if the dual $\Gamma$ of $\mathscr{T}$ is not a polyhedral graph, then there exists an enrichment of $\mathscr{T}$ that is not admissible.
\end{lem}
\begin{proof}
Recall that a graph is polyhedral if and only if it is $3$-connected.
Let $\Gamma$ be polyhedral, and let $\Gamma'$ be a pseudo-simple graph that dominates $\Gamma$.
Suppose for contradiction that $\Gamma'$ has a topologically non-trivial bigon with two vertices $v, v'$. 
This bigon divides the graph $\Gamma'$ into two non-trivial pieces, so $\Gamma' - \{v, v'\}$ is disconnected.
This implies $\Gamma - \{v, v'\}$ is disconnected, which is a contradiction.
Thus $\Gamma'$ is simple and 2-connected (in fact, $\Gamma'$ is always 3-connected).

Conversely, if $\Gamma$ is not polyhedral, then there exist two vertices $v, v'$ so that $\Gamma - \{v, v'\}$ is disconnected.
One can then construct a pseudo-simple graph $\Gamma'$ with a topologically non-trivial bigon (with vertices $v, v'$) that dominates $\Gamma$.

The lemma then follows from Lemma \ref{lem:detdt} and Lemma \ref{lem:rdetg}.
\end{proof}

\subsection{A characterization for convergence}
Recall that a parabolic map $[R] \in \partial \mathcal{H}_{\Gamma}$ is called a {\em root} of $\mathcal{H}_{\Gamma}$ if its Julia dynamics is topologically conjugate to that of maps in $\mathcal{H}_{\Gamma}$. In this subsection, we will prove the following theorem. 

\begin{theorem}\label{thm:cparm}
	Let $[R_n]\in \mathcal{H}_\Gamma(K)$ be a degeneration with associated enriched Tischler graph $\mathscr{T}^{En}$.
	Then $[R_n]$ has a convergent subsequence in $\mathcal{M}_d^-$ if and only if $\mathscr{T}^{En}$ is admissible.
	
	Moreover, if $\mathscr{T}^{En}$ is admissible with planar dual $\Gamma^{En}$, then any limit point of $[R_n]$ is a root of $\mathcal{H}_{\Gamma^{En}}$.
\end{theorem}

By Lemma \ref{lem:detdt}, the dual graph $\Gamma^{En}$ of an enrichment $\mathscr{T}^{En}$ of $\mathscr{T}=\Gamma^{\vee}$ dominates $\Gamma$. Thus we immediately have the following corollary.
\begin{cor}\label{cor:br}
Any map $[R]$ on the boundary $\partial \mathcal{H}_\Gamma(K)$ is a root of $\mathcal{H}_{\Gamma'}$ for some $\Gamma' \geq \Gamma$.
\end{cor}

\subsection*{Admissible $\implies$ convergence}
We will first prove one direction of Theorem \ref{thm:cparm}.
Let $[R_n]$ be a degeneration with admissible enriched Tischler graph $\mathscr{T}^{En}$.
Let $v$ be a vertex in $\mathscr{T}^{En}$, and let $v_n$ be the corresponding vertex for $\mathscr{T}^{En}_n$.
Let $U_n$ be the Fatou component of $R_n$ that contains $v_n$.
We may choose a representative $R_n \in [R_n]$ so that 
\begin{itemize}
	\item $v_n = 0$,
	\item $\D \subseteq U_n \subseteq \C$, and
	\item $1 \notin U_n$, i.e., $1\in\partial U_n$.
\end{itemize}
Under this normalization, after passing to a subsequence, we may assume that $(U_n, v_n)$ converges in Carath\'eodory topology to a pointed disk $(U, 0)$ and $R_n$ converges to $R$ algebraically.
By passing to a subsequence, we also assume that $\mathscr{T}^{En}_n$ converges in Hausdorff topology to $\mathscr{T}^{En}_\infty$.
Note that for now $\mathscr{T}^{En}_\infty$ is only a compact set of $\hat\C$. 
We will show that it is in fact a graph in Lemma \ref{lem:centt}.

Since $v_n$ is $M$-quasi-fixed under $R_n$ in the hyperbolic metric of $U_n$, the restrictions $R_n\vert_{U_n}$ converge compactly to a non-constant map on $U$.
Therefore, the degree of $R$ is at least $1$. Recall that $d$ is the degree of each $R_n$.

\begin{lem}\label{lem:degd}
	 The limiting map $R$ has degree $d$.
\end{lem}
\begin{proof}
	Suppose that this is not true.
	Let $a \in \mathcal{H}(R)$ be a hole.
	Let $B(a, \epsilon)$ be the $\epsilon$ neighborhood of $a$, and $C_\epsilon := \partial B(a, \epsilon)$.
	By shrinking $\epsilon$, we may assume that $\overline{B(a, \epsilon)}\setminus\{a\}$ contains neither holes nor fixed points of $\varphi_R$.
	Since $a$ is a hole and $\varphi_R$ has degree at least $1$, by Lemma \ref{lem:ch}, there is a sequence of critical points of $R_n$ converging to $a$. 
	Hence, the circle $C_\epsilon$ is an essential closed curve for $\mathscr{T}^{En}_n$ for large $n$.
	Thus, by Lemma \ref{lem:3e}, $C_\epsilon$ cuts at least $3$ edges of $\mathscr{T}^{En}_n$.

	Let $x_n\in C_\epsilon \cap \mathscr{T}^{En}_n$ be a point of intersection of the curve $C_\epsilon$ with a cutting edge.
	After passing to a subsequence, we may assume that $x_n \to x$ and $x_n \in V_{n}$ for some Fatou component $V_n$ of $R_n$.
	By Lemma \ref{lem:ac}, we have that $R_n(x_n)$ converges to $\varphi_R(x)$.
	Since $\overline{B(a, \epsilon)}\setminus\{a\}$ contains no fixed point of $\varphi_R$, we conclude that the spherical distances $d_{\mathbb{S}^2}(x_n,R_n(x_n))$ stay uniformly bounded away from $0$.
	On the other hand, we know that $d_{V_{n}}(x_n, R_n(x_n)) \leq M$ by Theorem \ref{thm:qit}.
	By Lemma \ref{lem:hme}, we conclude that the spherical distance $d_{\mathbb{S}^2}(x_n, \partial V_{n})$ is uniformly bounded away from $0$.
	Thus, after possibly passing to a subsequence, $(V_{n}, x_n)$ converges in Carath\'eodory topology to some pointed disk $(V, x)$.

	Let $\mathcal{T}_V: = \mathscr{T}^{En}_\infty \cap V$.
	Then if $y\in \mathcal{T}_V$, there exists a sequence $y_n \to y$ with $y_n \in \mathscr{T}^{En}_n$.
	Since $d_{V_n}(y_n, R_n(y_n)) \leq M$, we conclude that $d_V(y, \varphi_R(y)) \leq M$.
	This means that any limit point $y$ of $\mathcal{T}_V$ on $\partial V$ is fixed by $\varphi_R$.
	Since $C_\epsilon$ cuts the edge containing $x_n$, and there are no fixed points of $\varphi_R$ in $\overline{B(a, \epsilon)}\setminus \{a\}$, we conclude that $a$ is a limit point of $\mathcal{T}_V$ on $\partial V$, and $\varphi_R(a) = a$.

	If $a$ is a critical point of $\varphi_R$, then near $a$, the map $\varphi_R$ behaves like $\overline{z}^k$ for some $k\geq 2$.
	Since $a\in \partial V$, using Lemma \ref{lem:hme} and the fact
	$$
		\int_{t^k}^t \frac{1}{y} dy = \log \frac{1}{t^{k-1}} \to\infty
	$$
	as $t \to 0$,
	we conclude that there exists $x\in \mathcal{T}_V$ with $d_V(x, \varphi_R(x)) > M$, which is a contradiction.

	Thus, $a$ is not a critical point of $\varphi_R$.
	Since $C_\epsilon$ cuts at least $3$ edges of $\mathscr{T}^{En}_n$, we have three distinct edges $e_1, e_2,$ and $e_3$ of $\mathcal{T}_{V^1}$, $\mathcal{T}_{V^2}$, and $\mathcal{T}_{V^3}$ with $a \in e_1 \cap e_2 \cap e_3$.
	Note that the union of any two edges divides a neighborhood $N_a$ of $a$ into two components.
	Since $\varphi_R$ behaves like a reflection near $a$, there can be at most two invariant directions to $a$, and hence, there exists an edge, say $e_1$, which is mapped to the component of $N_a\setminus \overline{e_2 \cup e_3}$ not containing $e_1$.
	This means that there exists $x\in e_1$ with $d_{V_1}(x, \varphi_R(x))>M$, which is a contradiction.
\end{proof}

\subsection*{Convergence of Enriched Tischler graph}
The following lemma is used to understand the dynamics of the limiting map $[R]$. 
It states roughly that the plane isomorphism class of the sequence of plane graphs $\mathscr{T}^{En}_n$ is preserved under Hausdorff convergence.
\begin{lem}\label{lem:centt}
Let $[R_n]\in \mathcal{H}_\Gamma(K)$ be a degeneration converging to $[R]$.
Let $\mathscr{T}^{En}_\infty$ be an accumulation point of $\mathscr{T}^{En}_n$ in the Hausdorff topology.
	Then $\mathscr{T}^{En}_\infty$ is a graph, and it is isomorphic as a plane graph to the enriched Tischler graph $\mathscr{T}^{En}$ (of $\mathscr{T}=\Gamma^{\vee}$) corresponding to the sequence $[R_n]$.
\end{lem}
\begin{proof}
	Let $v$ be a vertex of $\mathscr{T}^{En}$, and $v_n$ be the corresponding vertex for $\mathscr{T}^{En}_n$.
	Let $U_n$ be the Fatou component of $R_n$ containing $v_n$.
	
	We claim that the spherical distance between $v_n$ and $\partial U_n$ is bounded below.
	Indeed, otherwise, since $v_n$ is $M$-quasi-fixed under $R_n$ in the hyperbolic metric of $U_n$, the sequence $v_n$ converges to a fixed point $v\in \mathcal{J}(R)$ of $R$.
	Since there are critical points of $R_n$ within uniformly bounded distance from $v_n$ in the hyperbolic metric of $U_n$, the fixed point $v$ must be a critical point of $R$ on its Julia set.
	This is a contradiction.
	
	Therefore, after passing to a subsequence, $(U_n, v_n)$ converges in Carath\'eodory topology to a pointed disk $(U_\infty^v, v_\infty)$, and $R$ maps $U_\infty^v$ to itself.
	Since the edges of $\mathscr{T}^{En}_n \cap U_n$ are hyperbolic geodesics in $U_n$, 
	$$
	\mathcal{T}_v :=\mathscr{T}^{En}_\infty\cap U_\infty^v
	$$ 
	is a graph. 
	Since the distance between two different vertices in $\mathscr{T}^{En}_n \cap U_n$ tends to infinity, $\mathcal{T}_v$ consists of a single branch point $v_\infty$ along with $\val_{\mathscr{T}^{En}}(v)$ half-open edges attached to it. It also follows that for $v\neq w$, we have $U_\infty^v\cap U_\infty^w=\emptyset.$
	
	Let $x$ be a limit point of $\mathcal{T}_v$ on $\partial U_\infty^v$.
	Then $x = \lim x_n$ with $x_n \in \mathscr{T}^{En}_n\cap U_n$.
	Since $d_{U_n}(x_n, R_n(x_n)) \leq M$, we conclude that $x$ is a fixed point of $R$.
Since $R$ has finitely many fixed points, $\overline{\mathcal{T}_v}$ is a graph. Moreover, for $w\neq v\in V(\mathscr{T}^{En})$, the closures $\overline{\mathcal{T}_{v}}$ and $\overline{\mathcal{T}_{w}}$ may intersect only at fixed points of $R$, proving that $\cup_{v\in V(\mathscr{T}^{En})} \overline{\mathcal{T}_v}$ is also a graph.

The quasi fixed property of $\mathscr{T}^{En}_n$ (in each invariant Fatou component) also implies that $\mathscr{T}^{En}_\infty$ intersects $\mathcal{J}(R)$ only at (non-critical) fixed points of $R$. The fact that at most two invariant accesses can land at each non-critical fixed point of an anti-rational map combined with Hausdorff convergence of $\mathscr{T}^{En}_n$ to $\mathscr{T}^{En}_\infty$ permits us to conclude that 
$$
\mathscr{T}^{En}_\infty = \bigcup_{v\in V(\mathscr{T}^{En})} \overline{\mathcal{T}_v}.
$$
It also follows that $\mathscr{T}^{En}_\infty$ has no vertex of valence $1$, and that the branch points of $\mathscr{T}^{En}_\infty$ correspond bijectively to the vertices of $\mathscr{T}^{En}$. Declaring the set of branch points of $\mathscr{T}^{En}_\infty$ to be its vertex set, one now easily sees that $\mathscr{T}^{En}_\infty$ is a graph isomorphic to $\mathscr{T}^{En}$ as plane graphs.
\end{proof}

\begin{proof}[Proof of Theorem \ref{thm:cparm}]
Suppose first that $\mathscr{T}^{En}$ is admissible. By Lemma \ref{lem:degd}, $[R_n] \to [R] \in \mathcal{M}_d^-$.

We will prove the moreover part that $[R]$ is a root associated to some enriched Tischler graph.
Since each critical point of $R$ is contained in $U_\infty^v$ for some $v\in \mathcal{V}(\mathscr{T}^{En})$, and each $U_\infty^v$ is contained in some invariant Fatou component $\Omega_v$ of $R$, we conclude that each critical point of $R$ lies in an invariant Fatou component.
	
	We claim that $\Omega_v \neq \Omega_w$ if $v\neq w$.
	This is clear if $\Omega_v$ is an attracting basin, as in this case $\Omega_v = U_\infty^v$.
	Now assume that $\Omega_v=\Omega_w$ is parabolic. We denote the parabolic fixed point by $a \in \partial \Omega_v$. Then $a\in \mathscr{T}^{En}_\infty$ lies in the interior of an edge with exactly one invariant access from $\Omega_v$ and one from outside. On the other hand, since $U_\infty^v\cap U_\infty^w=\emptyset$, there are edges in $\mathcal{T}_v$ and $\mathcal{T}_w$ landing at $a$ giving rise to two invariant accesses to $a$ from $\Omega_v=\Omega_w$. This contradiction proves the claim.
	
	A modification of \cite[Theorem 1.4]{CT18} in the anti-holomorphic setting provides us with a perturbation $[R']$ of $[R]$ so that $[R']$ is hyperbolic and the dynamics on the Julia sets are conjugate.

	Let $[\check{\mathcal{R}}]$ be the postcritically finite map in the hyperbolic component containing $[R']$. Since each critical point of $R$ lies in an invariant Fatou component, it follows that $[\check{\mathcal{R}}]$ is a critically fixed anti-rational map.

By Lemma \ref{lem:centt} and the fact that $\Omega_v \neq \Omega_w$ for $v\neq w$, we have that the Tischler graph of $[\check{\mathcal{R}}]$ is $\mathscr{T}^{En}$; i.e., $[\check{\mathcal{R}}]$ is the center of the hyperbolic component $\mathcal{H}_{\Gamma^{En}}$.
Thus $[R] \in \partial{\mathcal{H}_{\Gamma^{En}}}$, and the Julia dynamics of $[R]$ is conjugate to the Julia dynamics of $[\mathcal{R}_{\Gamma^{En}}]$, i.e., $[R]$ is a root of $\mathcal{H}_{\Gamma^{En}}$.

Conversely, suppose that $\mathscr{T}^{En}$ is not admissible, and suppose for contradiction that $[R_n]$ converges to $[R] \in \mathcal{M}_d^-$.
Then the conclusion of Lemma \ref{lem:centt} holds. 
The above argument then shows that $[R]$ is a root of a hyperbolic component $\mathcal{H}_{\Gamma^{En}}$ whose center is a critically fixed map with Tischler graph $\mathscr{T}^{En}$. But this would force the planar dual of $\mathscr{T}^{En}$ to be simple, which is a contradiction.
\end{proof}

\begin{rmk}
Although the dynamics of a root $[R]$ of $\mathcal{H}_{\Gamma^{En}}$ on its Julia set is topologically conjugate to that of $\mathcal{R}_{\Gamma^{En}}$, the dynamics of $[R]$ on its fixed Fatou components is not recorded by $\Gamma^{En}$. 
For example, the attracting basins of parabolic fixed points are not recorded.
This can be encoded by an additional \emph{arrow structure} on the graph $\Gamma^{En}$ as described in Appendix~\ref{sec:ve}.
\end{rmk}

\subsection{A characterization of boundedness of pared deformation spaces}
\begin{proof}[Proof of Theorem \ref{thm:ab}]
Let $\mathscr{T}$ be the dual of $\Gamma$.
Suppose first that $\Gamma$ is polyhedral.
By Lemma \ref{lem:pg}, any enrichment of $\mathscr{T}$ is admissible.
Thus, by Theorem \ref{thm:cparm}, $\mathcal{H}_\Gamma(K)$ is bounded for any $K$.

Now suppose that $\Gamma$ is not polyhedral.
By Lemma \ref{lem:pg}, there exists an non-admissible enrichment $\mathscr{T}^{En}$ of $\mathscr{T}$.
By Proposition~\ref{prop:enriched_tischler_realized}, we can construct a sequence in $\mathcal{H}_\Gamma(K)$ for all large $K$ with enriched Tischler graph $\mathscr{T}^{En}$.
Thus, by Theorem \ref{thm:cparm}, $\mathcal{H}_\Gamma(K)$ is unbounded for all large $K$.
\end{proof}

\subsection{Interactions between pared deformation spaces}
\begin{proof}[Proof of Theorem \ref{thm:main}]
Suppose $\Gamma' > \Gamma$. By Lemma \ref{lem:rdetg} and Proposition~\ref{prop:enriched_tischler_realized}, we can construct a sequence in $\mathcal{H}_\Gamma(K)$ whose enriched Tischler graph $\mathscr{T}^{En}$ is dual to $\Gamma'$.
Thus, by Theorem \ref{thm:cparm}, $\mathcal{H}_\Gamma(K)$ parabolic bifurcates to $\mathcal{H}_{\Gamma'}(K)$.

Conversely, suppose $\mathcal{H}_\Gamma(K)$ parabolic bifurcates to $\mathcal{H}_{\Gamma'}(K)$.
By Corollary \ref{cor:br}, any map on $\partial \mathcal{H}_\Gamma(K)$ is a root of $\mathcal{H}_{\Gamma''}$ for some $\Gamma'' \geq \Gamma$. Moreover, such a hyperbolic component is unique by Thurston rigidity.
Thus $\Gamma' > \Gamma$.
\end{proof}

\subsection{From pared deformation space to hyperbolic component}\label{subsec:dsthc}
Let us conclude this section with a brief discussion of analogues of the above results for the whole hyperbolic components of critically fixed anti-rational maps.
We first note that
\begin{prop}\label{prop:ub}
Let $\mathcal{H}_\Gamma$ be a hyperbolic component containing a critically fixed anti-rational map $\mathcal{R}_\Gamma$.
Then $\mathcal{H}_\Gamma$ is not bounded in $\mathcal{M}_d^-$.
\end{prop}

\begin{figure}[h!]
\begin{tikzpicture}[thick]

  \node at (4.5,4) [circle,fill=Black,inner sep=3pt] {};
  \node at (6.4,7) [circle,fill=Black,inner sep=3pt] {};
     \node at (6.4,1) [circle,fill=Black,inner sep=3pt] {};
  \node at (1,4) [circle,fill=black,inner sep=3pt] {};

  \node at (4,4.8) [circle,fill=red,inner sep=3pt] {};
  \node at (4,3) [circle,fill=red,inner sep=3pt] {};
    \node at (5.5,4) [circle,fill=red,inner sep=3pt] {};

   \draw[-,line width=1pt] (1,4)->(4.5,4);
   \draw[-,line width=1pt] (4.5,4)->(6.4,7);
   \draw[-,line width=1pt] (6.4,1)->(6.4,7);
   \draw[-,line width=1pt] (4.5,4)->(6.4,1);
   \draw[-,line width=1pt] (1,4)->(6.4,1);
   \draw[-,line width=1pt] (1,4)->(6.4,7);

   \draw[-,line width=3pt,blue] (4,4.653)->(4,3.148);
   \draw[-,line width=3pt,blue] (4.14,3.04)->(5.4,3.9);
   \draw[-,line width=3pt,blue] (5.4,4.1)->(4.12,4.72);
   
   \draw[-,line width=1pt,blue] (5.64,4)->(8,4);
  \draw[-,line width=1pt,blue] (3.92,2.88)->(2.5,0.6);
  \draw[-,line width=1pt,blue] (3.92,4.92)->(2.5,6.74);

   \end{tikzpicture}
 \caption{The contact graph $\Gamma$ is drawn in black. The Tischler graph contains a loop with bold blue edges. The restriction of this loop to each Fatou component is either empty or a union of two internal rays.}
 \label{tischler_loop_fig}
   \end{figure}
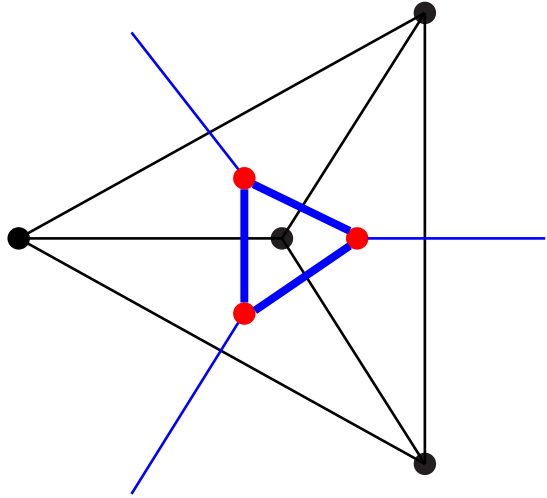
   
\begin{proof}
Let $\gamma$ be a simple loop in the Tischler graph.
If $\gamma$ intersects a Fatou component $U$, then $\gamma\cap U$ is the union of two internal rays connecting the super attracting fixed point $c \in U$ to two repelling fixed point $q_1, q_2 \in \partial U$ (see Figure~\ref{tischler_loop_fig}).
We can perturb the map $\mathcal{R}_\Gamma$ slightly so that all the superattracting fixed points $c$ become attracting, and we still have invariant rays connecting $c$ and $q_1, q_2$ (see \cite[Proposition~3.1]{Makienko00} for the holomorphic setting).
Then a standard pinching degeneration along the union of all these rays gives a diverging sequence in $\mathcal{H}_\Gamma$ (see \cite[Theorem~A(b)]{Tan02} for the holomorphic setting).
\end{proof}

This unboundedness result is one of the motivations to introduce the pared deformation space as the most naive translation of Thurston's compactness result is false: the hyperbolic component $\mathcal{H}_\Gamma$ of a critically fixed anti-rational map is never bounded.
This is not very surprising as the repelling fixed points correspond to cusps in the group setting, and we do not change the local geometry of the cusps when we deform the corresponding kissing reflection groups.

Since the pared deformation space $\mathcal{H}_\Gamma(K)$ is a subset of $\mathcal{H}_\Gamma$, Theorem \ref{thm:main} immediately implies that if $\Gamma' > \Gamma$, then $\mathcal{H}_\Gamma$ parabolic bifurcates to $\mathcal{H}_{\Gamma'}$. 
While the geometric control in $\mathcal{H}_\Gamma(K)$ allows us to show that the boundaries of pared deformation spaces are tame, the geometry of the boundary $\partial \mathcal{H}_\Gamma$ can be quite complicated in general.
We do not know if the closures of two hyperbolic components $\overline{\mathcal{H}_\Gamma}$ and $\overline{\mathcal{H}_{\Gamma'}}$ can intersect in exotic ways.

\section{Markov partitions and monodromy representations}\label{sec:mr}
In this section, we will construct a monodromy representation and prove Theorem \ref{thm:mono}.

Let $\mathcal{R}_\Gamma$ be a critically fixed anti-rational map with Tischler graph $\mathscr{T}$. Furthermore, let $F_1,\cdots, F_{d+1}$ be the faces of $\mathscr{T}$. Then by \cite[Corollary~4.7]{LLM22},  we have
$$
\mathcal{R}_\Gamma(\overline{F_i}) = \bigcup_{j\neq i} \overline{F_j}.
$$
Hence, one obtains a Markov partition for the action of $\mathcal{R}_\Gamma$ on $\widehat\C$, with transition matrix
$$
M = \begin{bmatrix}
    0 & 1 & 1 & \dots  & 1 \\
    1 & 0 & 1 & \dots  & 1 \\
    \vdots & \vdots & \vdots & \ddots & \vdots \\
    1 & 1 & 1 & \dots  & 0
\end{bmatrix}.
$$

Let $\mathcal{J}(\mathcal{R}_\Gamma)$ be the Julia set of $\mathcal{R}_\Gamma$.
Then the closure of the faces of $\mathscr{T}$ induce a Markov partition on 
$$
\mathcal{J}(\mathcal{R}_\Gamma) = \bigcup_{i=1}^{d+1} \mathcal{K}_i(\mathcal{R}_\Gamma)
$$ 
with the same transition matrix, where $\mathcal{K}_i(\mathcal{R}_\Gamma)= \mathcal{J}(\mathcal{R}_\Gamma)\cap\overline{F_i}$.

For $K >0$, let $\mathcal{H}_{\Gamma, Rat}(K) \subseteq \mathcal{H}_{\Gamma, Rat} \subset \Rat_d^-(\C)$ be the lifts of $\mathcal{H}_\Gamma(K) \subseteq \mathcal{H}_\Gamma \subset \mathcal{M}_d^-$. Define
$$
\mathcal{X}_d(K) := \displaystyle\bigcup_{\Gamma} \overline{\mathcal{H}_{\Gamma, Rat}(K)} \subset \Rat_d^-(\C).
$$

By Corollary \ref{cor:br}, the dynamics on the Julia set of $R \in \mathcal{X}_d(K)$ is topologically conjugate to a critically fixed anti-rational map $\mathcal{R}_\Gamma$. 
Thus, each map $R$ is equipped with a Markov partition 
	$$
	\mathcal{J}(R) = \bigcup_{i=1}^{d+1} \mathcal{K}_i(R)
	$$ 
on its Julia set.

Note that by Theorem \ref{thm:cparm} periodic points of period $\geq 3$ in $\mathcal{X}_d(K)$ never collide.
Thus they move continuously in $\mathcal{X}_d(K)$.
Let $\gamma$ be a path in $\mathcal{X}_d(K)$ and let $x$ be a periodic point of $\gamma(0)$ with period $\geq 3$.
We use $x(t) \in \widehat{\C}$ to denote the corresponding periodic point of $\gamma(t)$ by tracing along the path.

\subsection*{Monodromy representation of $\pi_1(\mathcal{X}_d(K))$}
In the following, we show that there exists an induced monodromy representation
$$ 
\rho: \pi_1(\mathcal{X}_d(K))\longrightarrow \MCG(S_{0, d+1}).
$$
We choose the base point as $\mathcal{R}_0(z) = \bar z^d \in \mathcal{X}_d(K)$.

Let ${\bf x} = (x_1, \cdots, x_{d+1})$ and ${\bf y} = (y_1,\cdots, y_{d+1})$ be two ordered subset of $\mathcal{J}(\mathcal{R}_0)=\mathbb{S}^1$ oriented counterclockwise.
Let
$$
\Psi_{{\bf x} \to {\bf y}}: (\widehat{\C}, x_1, \cdots, x_{d+1}) \longrightarrow (\widehat{\C}, y_1, \cdots, y_{d+1})
$$
be a homeomorphism with $\Psi_{{\bf x} \to {\bf y}}(\mathbb{S}^1) =\mathbb{S}^1$.
Note that any two such homeomorphisms are isotopic.
We denote this unique isotopy class by $[\Psi_{{\bf x} \to {\bf y}}]$.

Let $x_1(0), \cdots, x_{d+1}(0)$ be periodic points of period $\geq 3$ of $\mathcal{R}_0$ such that $x_{i}(0) \in \mathcal{K}_i(0):=\mathcal{K}_i(\mathcal{R}_0)$.
Note that there are many choices here, but the monodromy representation will not depend on such choices.

Let $\gamma \subset \mathcal{X}_d(K)$ be a closed curve based at $\mathcal{R}_0$. 
Then we can trace the continuous motions of $x_1(t), \cdots, x_{d+1}(t)$.
These continuations define an isotopy class of maps $[\widehat{\Psi}_{\gamma}]$ from $(\widehat{\C}, x_1(0), \cdots, x_{d+1}(0))$ to $(\widehat{\C}, x_1(1), \cdots, x_{d+1}(1))$.

We remark that the set of points $\{x_1(0), \cdots, x_{d+1}(0)\}$ may be different from the set $\{x_1(1), \cdots, x_{d+1}(1)\}$, and hence $[\widehat{\Psi}_{\gamma}]$ is in general not an element of the mapping class group of $S_{0,d+1}$.
To remedy this, we first note that $x_i(t) \neq x_j(t)$ belong to different Markov pieces of the Julia set of $\gamma(t)$ for all $t$.
We can thus continuously label the Markov pieces by $\mathcal{K}_i(t)$.

Let $s: \{1,\cdots, d+1\} \longrightarrow \{1,\cdots, d+1\}$ be the permutation so that $\mathcal{K}_{s(i)}(1) = \mathcal{K}_{i}(0)$.
Then $x_{s(i)}(1) \in \mathcal{K}_i(0)$.
Let ${\bf x(0)} = (x_1(0), \cdots, x_{d+1}(0))$ and ${\bf x(1)} = (x_{s(1)}(1), \cdots, x_{s(d+1)}(1))$.
We define 
$$
\rho([\gamma]) = [\Psi_{\gamma}] = [\Psi_{{\bf x(1)} \to {\bf x(0)}} \circ \widehat{\Psi}_{\gamma}] \in \MCG(\widehat{\C}, x_1(0), \cdots, x_{d+1}(0)).
$$
It is clear that the homotopy class $[\widehat{\Psi}_{\gamma}]$ does not depend on the choice of representative for $[\gamma]$, we thus have a well-defined monodromy map
$$
\rho: \pi_1(\mathcal{X}_d(K))\longrightarrow \MCG(\widehat{\C}, x_1(0), \cdots, x_{d+1}(0)) \cong \MCG(S_{0, d+1}).
$$
It is also easy to verify that this monodromy map gives a group homomorphism, which we call the {\em monodromy representation}.
We are now ready to prove Theorem \ref{thm:mono}:

\begin{proof}[Proof of Theorem \ref{thm:mono}]
Let $\mathcal{R}_0(z) = \bar{z}^d$ be the base point for $\mathcal{X}_d(K)$.
Let $\mathcal{J}(\mathcal{R}_0)=\mathbb{S}^1= \bigcup_{i=1}^{d+1} \mathcal{K}_i(0)$ be the Markov partition for $\mathcal{R}_0$, and $x_i \in \mathcal{K}_i(0)$ be a periodic point with period $\geq 3$.

\begin{figure}[ht]
  \centering
  \resizebox{0.8\linewidth}{!}{
    \def\svgwidth{\columnwidth}
    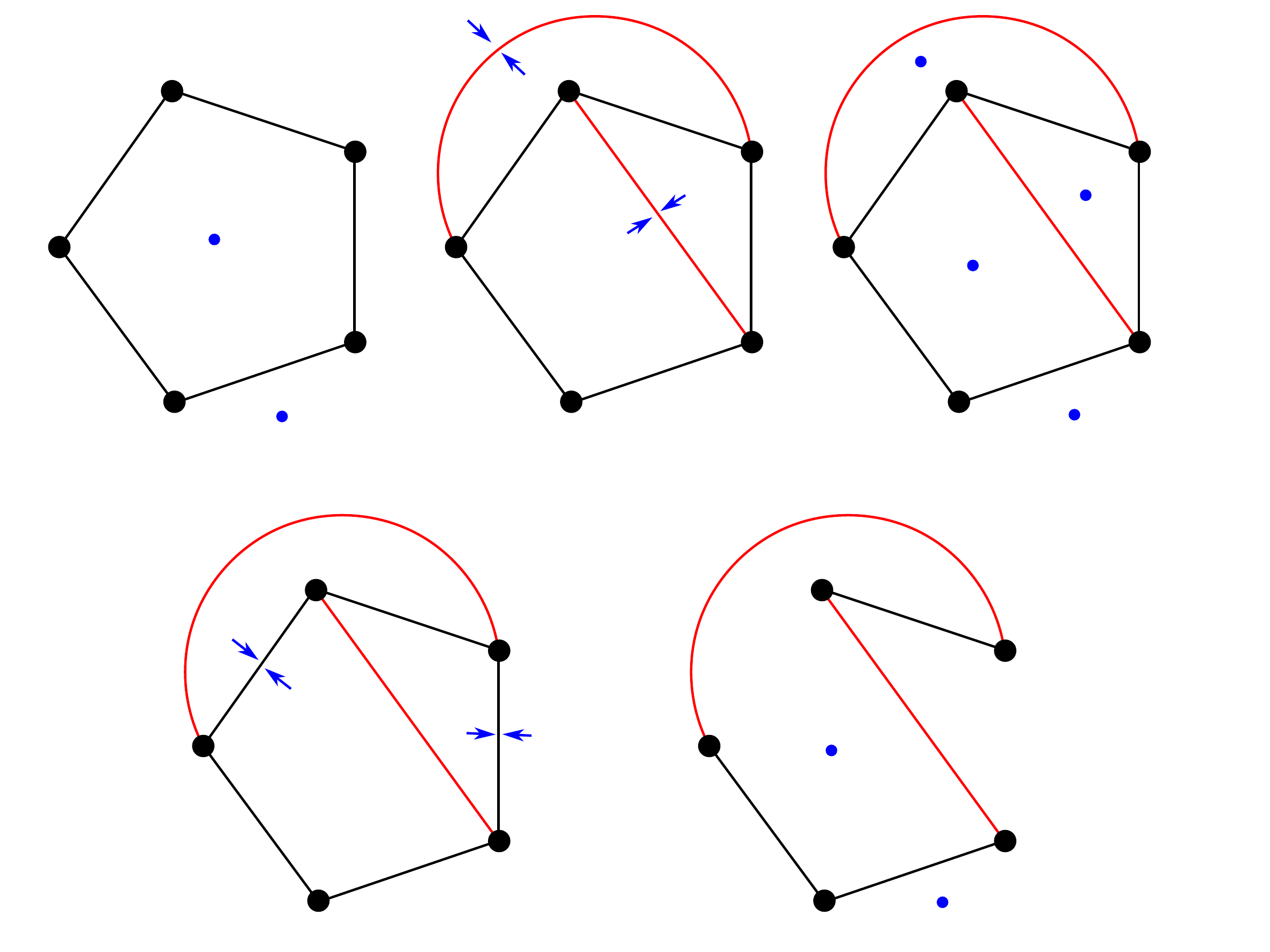

  }
  \caption{A generator of $\MCG(S_{0, d+1})$.}
  \label{fig:Gen}
\end{figure}

Let $\sigma_i \subset \widehat{\C}\setminus \{x_1,\cdots, x_{d+1}\}$ be a simple closed curve around $x_i, x_{i+1}$, and let $T_i$ be the half Dehn twist along $\sigma_i$.
It suffices to show the $T_i$ is in the image of the monodromy representation.

To construct the closed curve $\gamma_i \subset \mathcal{X}_d(K)$ giving $T_i$, we consider a path illustrated in Figure \ref{fig:Gen}.

The first graph represents the dual graph of the Tischler graph for $\gamma(0) = \mathcal{R}_0(z) = \bar z^d$.

The second graph represents a parabolic anti-rational map $\gamma(1/4) \in \mathcal{X}_d(K)$, with two parabolic fixed points with attracting direction illustrated by the arrows (see Appendix~\ref{sec:ve} for a detailed discussion of arrowed graphs).
We assume that the dynamics on the invariant Fatou components of $\gamma(1/4)$ are conjugate to parabolic unicritical real anti-Blaschke products.
Note that under such conditions, $\gamma(1/4)$ is rigid: it is unique up to conformal conjugacy.

The path $\gamma: [0,1/4) \longrightarrow \mathcal{H}_{0, Rat}(K)$, where $\mathcal{H}_{0, Rat}(K)$ is the component containing $\mathcal{R}_0$, is constructed by Proposition \ref{prop:rabl}, where the rescaling limits are parabolic unicritical real anti-Blaschke products.
The rigidity of $\gamma(1/4)$ allows us to extend $\gamma$ continuously to $1/4$.

The third graph represents a critically fixed anti-rational map $\gamma(1/2)$, and the fourth graph represents a parabolic anti-rational map $\gamma(3/4) \in \mathcal{X}_d(K)$.
The critically fixed anti-rational map for the last graph is $\gamma(1)(z) = \mathcal{R}_0(z) = \bar z^d$.
The path from $\gamma(1/4)$ to $\gamma(1/2)$ and the path from $\gamma(1/2)$ to $\gamma(3/4)$ that connect a root to a center of a hyperbolic component can be constructed by applying Theorem \ref{thm:cvre}, or by a standard {\em simple pinching} construction (cf. \cite{CT18}). 
The path from $\gamma(3/4)$ to $\gamma(1)$ is constructed similarly as the path from $\gamma(0)$ to $\gamma(1/4)$.
After performing a rotation if necessary, we might assume that Markov partition $\mathcal{K}_3(0) = \mathcal{K}_3(1)$.

From the diagram, it follows that $\rho([\gamma_1]) = T_1$ is the half Dehn twist around $x_1, x_2$.
A similar construction gives $\gamma_i$ with $\rho([\gamma_1])  = T_i$ for any $i$, so $\rho$ is surjective.
\end{proof}

\appendix

\section{Laminations, automorphisms and accesses}\label{sec:pb}
In this appendix, we show how the Julia sets of critically fixed anti-rational maps are related via laminations (Proposition \ref{lem:pbtc}).

Laminations appear naturally as dual objects of enrichments of Tischler graphs. We study the action of the automorphism groups of critically fixed anti-rational maps on the set of laminations.
This allows us to count the number of inequivalent semiconjugacies between Julia sets (Proposition \ref{prop:scn}), and the number of accesses from one pared deformation space to another (Proposition \ref{prop:an}).
 
We remark that degenerations of kissing reflection groups are achieved by pinching suitable multi-curves on the conformal boundaries.
By lifting, the multi-curves give a lamination on each component of the domain of discontinuity, and the limit set of the limiting kissing reflection group is obtained by simultaneously pinching these laminations on the domain of discontinuity.
We shall see that a similar phenomenon also holds in the anti-rational map setting.

\subsection{Lamination}
A  {\em lamination} $\mathcal{L}$ is a family of disjoint hyperbolic geodesics in $\D$ together with the two endpoints in $\mathbb{S}^1\cong \R/\Z$, whose union is closed.
The hyperbolic geodesics are called the {\em leaves} of the lamination.

A lamination $\mathcal{L}$ generates an equivalence relation $\sim_\mathcal{L}$ on $\mathbb{S}^1$ given by $a\sim_{\mathcal{L}} b$ if $a,b \in \mathbb{S}^1$ are connected by a finite chain of leaves in $\mathcal{L}$.

Let $d\geq 2$, and $m_{- d}: \mathbb{S}^1\longrightarrow \mathbb{S}^1$ be the map $m_{- d}(t) = - d\cdot t$.
A lamination $\mathcal{L}$ is said to be $m_{- d}-$invariant if it satisfies the following two conditions.
\begin{itemize}
\item If there is a leaf with endpoints $x$ and $y$, then either $m_{- d}(x) = m_{- d}(y)$ or there is a leaf with end-points $m_{- d}(x)$ and $m_{- d}(y)$;
\item If there is a leaf with endpoints $x$ and $y$, then there is a set of $d$ disjoint leaves with one endpoint in $m_{- d}^{-1}(x)$ and the other endpoint in $m_{- d}^{-1}(y)$.
\end{itemize}
Note that there are exactly $d+1$ fixed points $0, \frac{1}{d+1}, \cdots, \frac{d}{d+1} \in \mathbb{S}^1$ of $m_{-d}$.
These $d+1$ fixed points partition the circle into $d+1$ intervals, which form a Markov partition for $m_{-d}$ with associated transition matrix
$$
M = \begin{bmatrix}
    0 & 1 & 1 & \dots  & 1 \\
    1 & 0 & 1 & \dots  & 1 \\
    \vdots & \vdots & \vdots & \ddots & \vdots \\
    1 & 1 & 1 & \dots  & 0
\end{bmatrix}.
$$
Thus the $2$-cycles of $m_{-d}$ are in bijective correspondence with pairs of non-adjacent Markov pieces.

Given a $2$-cycle $\mathcal{C} = \{x,y\}$, we denote the closure (in $\D\cup\mathbb{S}^1$) of the hyperbolic geodesic of $\D$ connecting $x,y$ by  $l_{\mathcal{C}}$. We call a collection of $2$-cycles $\mathcal{C}_1,\cdots, \mathcal{C}_k$ {\em simple} if the corresponding (compactified) geodesics $l_{\mathcal{C}_1},\cdots, l_{\mathcal{C}_k}$ are pairwise disjoint.

The following easy lemma can be proved by induction on the backward orbits of the endpoints of the geodesics.
\begin{lem}\label{lem:g2c}
Let $\mathcal{C}_1,\cdots, \mathcal{C}_k$ be a simple collection of $2$-cycles of $m_{-d}$.
Then there exists a unique $m_{-d}-$invariant lamination $\mathcal{L}$ so that
\begin{itemize}
\item every leaf $l$ of $\mathcal{L}$ is eventually mapped to $l_{\mathcal{C}_i}$ for some $i$; and
\item no leaves other than $l_{\mathcal{C}_i}$ separate the fixed points of $m_{-d}$.
\end{itemize}
\end{lem}

We call the lamination $\mathcal{L}$ in Lemma \ref{lem:g2c} the lamination generated by the $2$-cycles $\mathcal{C}_1,\cdots, \mathcal{C}_k$.

\subsection*{Lamination as dual of $d+1-$ended ribbon trees}
Let $\mathcal{L}$ be a lamination generated by the $2$-cycles $\mathcal{C}_1,\cdots, \mathcal{C}_k$.
The leaves $l_{\mathcal{C}_1}, \cdots, l_{\mathcal{C}_k}$ cut the closed disk $\overline{\D}$ into finitely many regions.
We can construct the dual tree $\mathcal{T}$ by introducing a vertex for each complementary component of 
$$
\overline{\D} \setminus \bigcup_{i=1}^k l_{\mathcal{C}_i},
$$
connecting two vertices if the corresponding regions share a common leaf as boundary, and attaching a ray to a vertex for each fixed point of $m_{-d}$ on the boundary of the corresponding region (see Figure~\ref{dual_tree_fig}).

\begin{figure}[ht]
 \begin{tikzpicture}[thick]
\node[anchor=south west,inner sep=0] at (0,0) {\includegraphics[width=.46\linewidth]{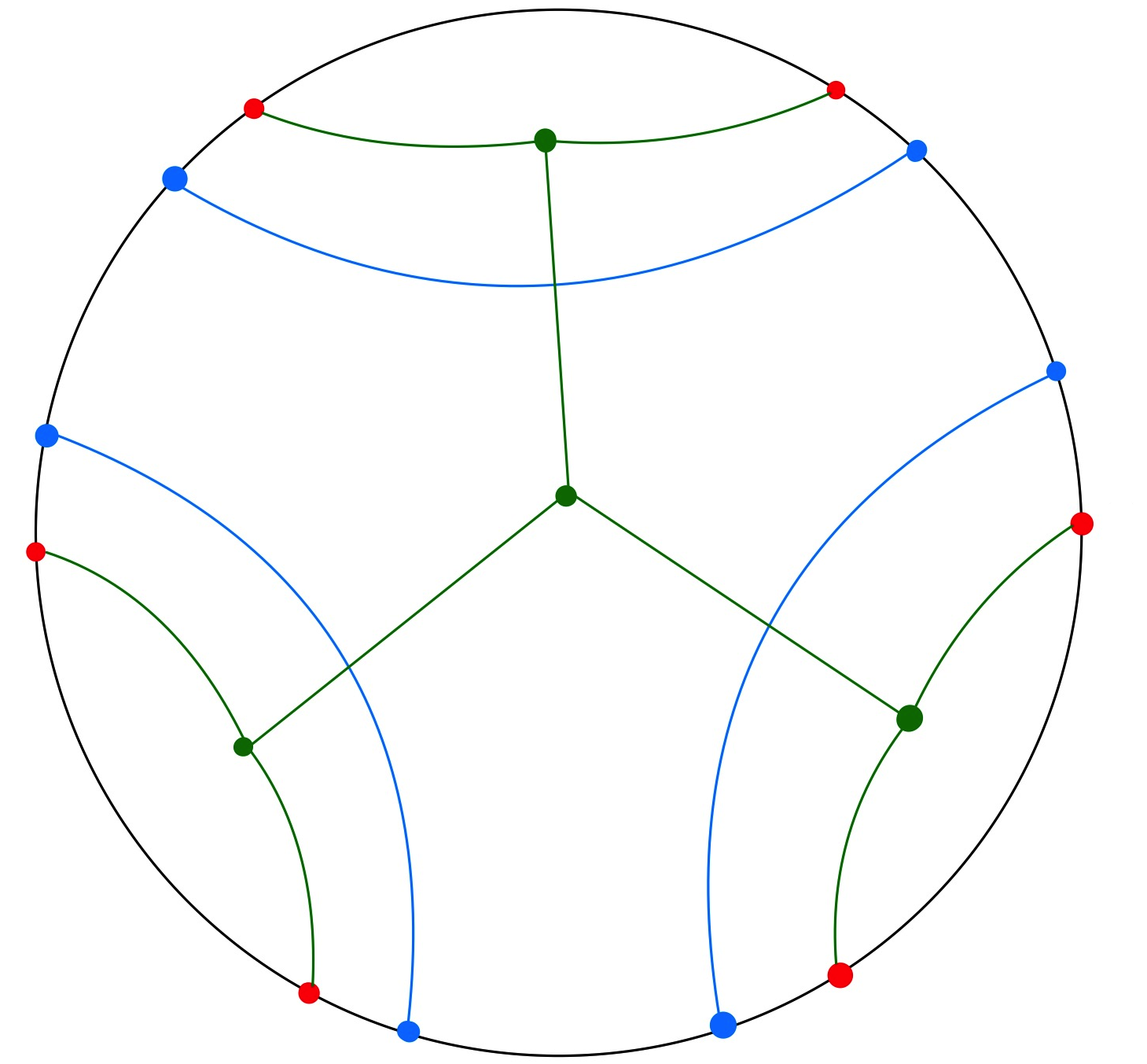}};
\node[anchor=south west,inner sep=0] at (6.9,0) {\includegraphics[width=.46\linewidth]{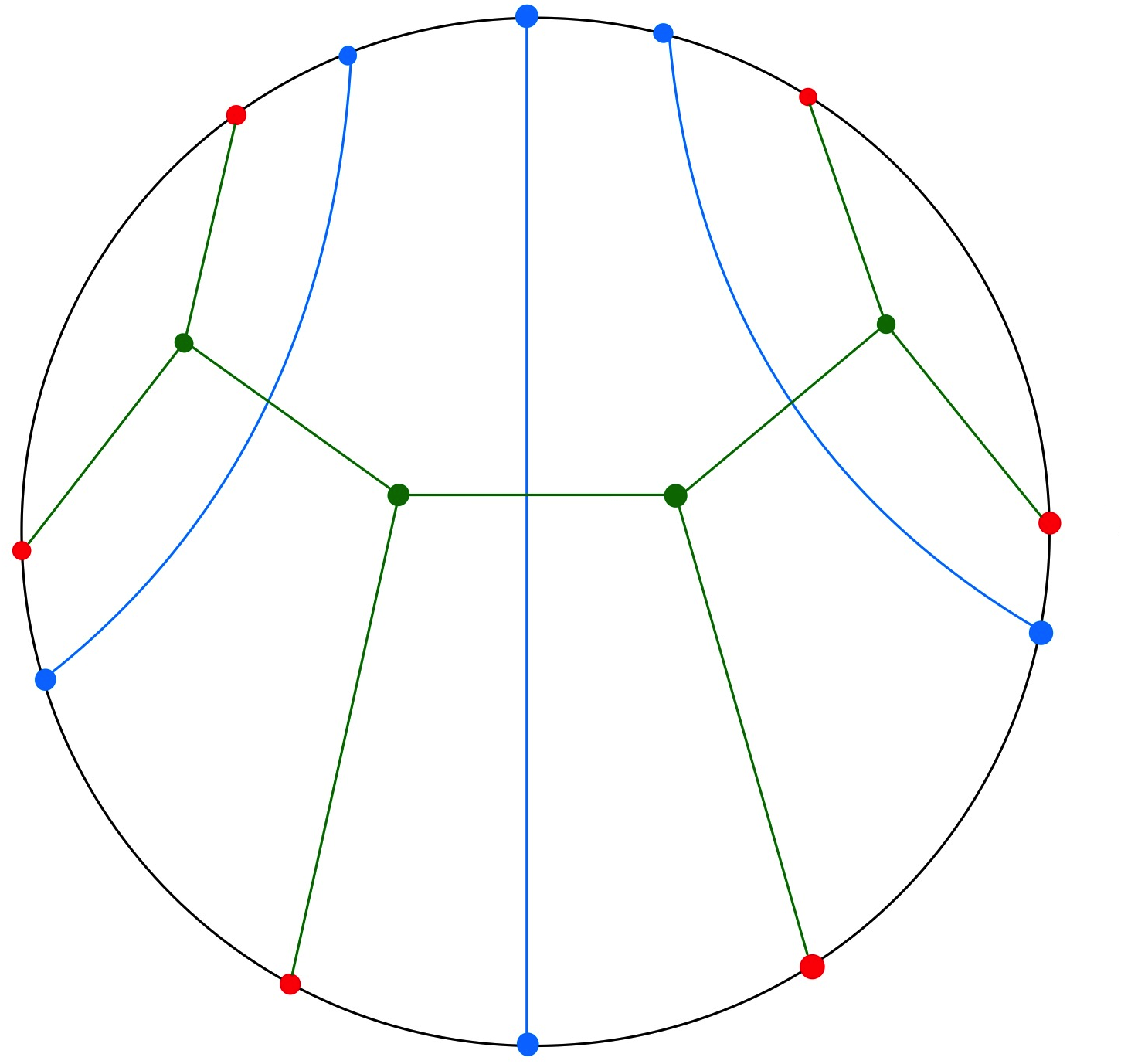}};
\node at (5.8,2.8) {$0$};
\node at (12.7,2.8) {$0$};
\end{tikzpicture}
\caption{The leaves of the laminations (generated by $2$-cycles of $m_{-5}$) are drawn in blue, the fixed points (of $m_{-5}$) are marked in red, and the corresponding dual trees are drawn in green.}
\label{dual_tree_fig}
\end{figure}

The following lemma can be verified directly from the definitions.
\begin{lem}\label{lem:ert}
Let $\mathcal{L}$ be a lamination generated by $2$-cycles.
The dual tree $\mathcal{T}$ of $\mathcal{L}$ is a $d+1-$ended ribbon tree.

Conversely, every $d+1-$ended ribbon tree arises as the dual tree for some lamination generated by $2$-cycles.
\end{lem}

\begin{rmk}
We regard the dual tree $\mathcal{T}$ of $\mathcal{L}$ as a marked $d+1-$ended ribbon tree by declaring the fixed point $0\in\mathbb{S}^1$ of $m_{-d}$ as the marked endpoint of $\mathcal{T}$.
\end{rmk}

\subsection*{Quotient of Julia set by lamination}
Invariant laminations naturally give quotient maps of Julia sets.
Such a construction works for general (anti-)rational maps.
For simplicity, we will restrict ourselves to the current setting.

Let $\mathcal{R}_\Gamma$ be a critically fixed anti-rational map with a boundary marking $q$ (see Subsection~\ref{subsec:pds}).
Let $U_1,\cdots, U_m$ be an enumeration of the critical Fatou components.
Then the induced dynamics on the ideal boundary $\mathbb{S}^1 = I(U_i)$ is conjugate to $m_{-d_i}$, where $d_i$ is the degree of the map $f: U_i \longrightarrow U_i$.

Let $\mathcal{L}_i$ be an $m_{-d_i}-$invariant lamination generated by a simple collection of $2$-cycles $\mathcal{C}_{i,1},\cdots, \mathcal{C}_{i, k_i}$ of $m_{-d_i}$.
Such a collection of laminations $\mathcal{L}_1,\cdots, \mathcal{L}_m$ defines an equivalence relation $\sim$ on the Julia set $\mathcal{J}(\mathcal{R}_\Gamma)$ as follows.

Let $x, y\in \partial U_i$.
We define $x\sim y$ if there exists $\tilde x \in \phi_i^{-1}(x)$ and $\tilde y \in \phi_i^{-1}(y)$ with $\tilde x \sim_{\mathcal{L}_i} \tilde y$, where $\phi_i: \mathbb{S}^1 \longrightarrow \partial U_i$ is the unique semiconjugacy compatible with the boundary marking $q$ (more precisely, $\phi_i$ maps $0\in\R/\Z$ to the marked fixed point on the ideal boundary of $U_i$).

Next, if $U$ is a strictly pre-fixed Fatou component with pre-period $l$, then $f^{\circ l}: \partial U \longrightarrow \partial U_i$ is a homeomorphism for some $i$.
We use this identification to pull back the equivalence relation to $\partial U$.

Since the equivalence relation induced by $\mathcal{L}_i$ in each Fatou component is closed, and the diameters of the Fatou components of $\mathcal{R}_\Gamma$ shrink to $0$ as the pre-period $l \to\infty$, it is easy to check that the equivalence relation $\sim$ is closed.
We say that such an equivalence relation $\sim$ is {\em generated by pinching $2$-cycles}.

Let $\Gamma' \geq \Gamma$.
Let $i: \Gamma \longrightarrow \Gamma'$ be an embedding.
By Lemma \ref{lem:rdetg}, the embedding $i$ makes the dual graph $\mathscr{T}'$ of $\Gamma'$ an enrichment of $\mathscr{T}$.
Thus, on each invariant Fatou component of $\mathcal{R}_\Gamma$, $\mathscr{T}'$ gives a lamination generated by $2$-cycles.
Let $\sim_{i}$ denote the equivalence relation generated by these laminations on invariant Fatou components of $\mathcal{R}_\Gamma$.

We remark that there may be many different embeddings $i: \Gamma \longrightarrow \Gamma'$, giving rise to different equivalence relations. 
This phenomenon is related to the automorphism group of $\Gamma$ and $\Gamma'$ and is discussed in Appendices \ref{sec:aa} and~\ref{app:eqm}.

The following proposition, which can be verified directly, relates the topology of the Julia sets.
\begin{prop}\label{lem:pbtc}
Let $\Gamma' \geq \Gamma$ and let $i: \Gamma \longrightarrow \Gamma'$ be an embedding. Then the Julia sets of the critically fixed anti-rational maps $\mathcal{R}_{\Gamma}$ and $\mathcal{R}_{\Gamma'}$ satisfy
$$
\mathcal{J}(\mathcal{R}_{\Gamma'}) = \mathcal{J}(\mathcal{R}_\Gamma)/\sim_{i}.
$$
Moreover, the quotient map $\Psi_i: \mathcal{J}(\mathcal{R}_{\Gamma'}) \longrightarrow  \mathcal{J}(\mathcal{R}_\Gamma)$ is a semi-conjugacy.
\end{prop}

\begin{rmk}
For the principal hyperbolic component of $\mathcal{M}_d^-$ (i.e., the one containing $[\overline{z}^d]$), degenerations in the corresponding pared deformation space are parametrized by a pair of $m_{-d}$-invariant laminations generated by $2$-cycles. The corresponding enriched Tischler graphs are admissible if and only if the associated pair of laminations are \emph{non-parallel}; i.e., they share no common leaf under the natural identification of the two copies of the circle. In this case, a simple extension of Proposition~\ref{lem:pbtc} shows that an accumulation point of such a degeneration is a geometric mating of a pair of parabolic anti-polynomials (cf. \cite[Theorem~1.3]{LLM22}). 
\end{rmk}

\subsection{Automorphisms}\label{sec:aa}
Let $f$ be an anti-rational map.
The automorphism group $\Aut(f)$ is the subgroup of $\PSL_2(\C)$ consisting of M{\"o}bius maps that commute with $f$.
Let $\QC(f)$ denote the group of all quasiconformal maps $\phi: \widehat{\C} \longrightarrow \widehat{\C}$ that commute with $f$. 
The {\em modular group} of $f$, denoted by $\Mod(f)$, is the group of isotopy classes of such $\phi$: in other words, it is the quotient of $\QC(f)$ by the path component of the identity.
Note that there is a natural inclusion $\Aut(f) \subseteq \Mod(f)$.

For an $f-$invariant closed subset $K_f\subset\widehat{\C}$, we define $\QC(K_f,f)$ as the group of all quasiconformal homeomorphisms $\phi: \widehat{\C} \to \widehat{\C}$ that commute with $f$ on $K_f$. The {\em modular group of $f$ on $K_f$}, denoted by $\Mod(K_f, f)$ is the group of isotopy classes of such $\phi$ (see \cite{McM88} for details).

Let $\Gamma$ be a plane graph. We define $\Aut(\Gamma)$ as the group of planar automorphisms of the graph $\Gamma$.
The first isomorphism in the following proposition is an anti-holomorphic counterpart of \cite[Proposition~6.10]{McM88}, while the second isomorphism is a consequence of Thurston rigidity of postcritically finite maps.
\begin{prop}\label{prop:maa}
Let $\mathcal{R}_{\Gamma}$ be a critical fixed anti-rational map. Then 
$$
\Mod(\mathcal{J}(\mathcal{R}_\Gamma), \mathcal{R}_\Gamma) \cong \Aut(\mathcal{R}_\Gamma) \cong \Aut(\Gamma).
$$
\end{prop}

\subsection{Enumeration of the quotient maps}\label{app:eqm}
Let $\mathcal{H}_\Gamma$ be a hyperbolic component, and $c_i$ be the distinct critical points of the critically fixed map $\mathcal{R}_\Gamma\in\mathcal{H}_\Gamma$ with local degree $d_i$ at $c_i$, $i\in\{1,\cdots, m\}$.

Recall that 
$$
\widetilde{\mathcal{H}}_\Gamma \cong \BP^-_{d_1} \times \BP^-_{d_2} \times \cdots \times \BP^-_{d_m}
$$ 
is the lift of $\mathcal{H}_\Gamma$ consisting of elements of $\mathcal{H}_\Gamma$ equipped with boundary markings.
As an extension of terminology, we say that an $m$-tuple of laminations $\mathcal{L}:= (\mathcal{L}_1,\cdots, \mathcal{L}_m)$ is a {\em lamination} for $\widetilde{\mathcal{H}}_\Gamma$ if each $\mathcal{L}_i$ is an $m_{-d_i}$-invariant lamination generated by $2$-cycles.
By Lemma~\ref{lem:ert}, each lamination $\mathcal{L}_i$ is dual to some $(d_i+1)$-ended ribbon tree.

We call $\mathcal{L}$ an {\em admissible} lamination for $\widetilde{\mathcal{H}}_\Gamma$ if the corresponding enrichment of $\mathscr{T}$ (obtained by blowing up the vertex $c_i$ of the Tischler graph $\mathscr{T}$ of $\mathcal{R}_\Gamma$ to the marked $(d_i+1)$-ended ribbon tree respecting markings) is admissible.
We will denote the set of admissible laminations (for $\widetilde{\mathcal{H}}_\Gamma$) by
$\widetilde{\mathcal{AL}}(\Gamma)$.


The group $\Aut(\mathcal{R}_\Gamma)$ acts on $\widetilde{\mathcal{H}}_\Gamma$ in the following way. Fix a representative $(\mathcal{R}_\Gamma, (x_1^0,\cdots, x_m^0))$ of the `center' of $\widetilde{\mathcal{H}}_\Gamma$. Then, for a representative $(R, (x_1,\cdots, x_m))$ of any element in $\widetilde{\mathcal{H}}_\Gamma$, there exists a unique orientation preserving homeomorphism $\kappa$ of the sphere that conjugates $\mathcal{R}_\Gamma\vert_{\mathcal{J}(\mathcal{R}_\Gamma)}$ to $R\vert_{\mathcal{J}(R)}$ preserving the marking. For $M\in\Aut(\mathcal{R}_\Gamma)$, we define its action on $\widetilde{\mathcal{H}}_\Gamma$ as
$$
M\cdot [R, (x_1,\cdots, x_m)]= [R, (\kappa(\widetilde{M}(x_1^0)),\cdots, \kappa(\widetilde{M}(x_m^0)))],
$$
where $\widetilde{M}$ is the action on the ideal boundaries of the fixed Fatou components of $R$ induced by $M$.
(It is trivial to check that the action does not depend on the choice of the representative $(R, (x_1,\cdots, x_m))$.) We note that this action is not free as each element of $\Aut(\mathcal{R}_\Gamma)$ fixes the `center' of $\widetilde{\mathcal{H}}_\Gamma$. Identifying $\Aut(\Gamma)$ with $\Aut(\mathcal{R}_\Gamma)$ (using Proposition \ref{prop:maa}), we have that

\begin{lem}\label{lem:muh}
The hyperbolic component $\mathcal{H}_\Gamma = \widetilde{\mathcal{H}}_\Gamma / \Aut(\Gamma)$.
\end{lem}

The automorphism group $\Aut(\Gamma)$ also acts naturally on $\widetilde{\mathcal{AL}}(\Gamma)$. Indeed, given $\mathcal{L}= (\mathcal{L}_1,\cdots, \mathcal{L}_m)$, think of $\mathcal{L}_i$ as a lamination sitting inside the Fatou component $U_i\ni c_i$ of $\mathcal{R}_\Gamma$ such that the marked endpoint of the dual tree corresponds to the marked fixed point on the (ideal) boundary of $U_i$. Any $M\in\Aut(\mathcal{R}_\Gamma)$ permutes the invariant Fatou components of $\mathcal{R}_\Gamma$. Thus, if $c_s=M(c_r)$, the image of $\mathcal{L}_r$ under $M$ is a lamination in $U_s$, but the $M$-image of the marked endpoint of the corresponding dual tree is not necessarily the marked fixed point on the (ideal) boundary of $U_s$. Hence, the images of various $\mathcal{L}_i$ under $M$ yield a new $m$-tuple of laminations, denoted by $M\cdot\mathcal{L}$, whose $s-$th coordinate is a rotate of the lamination $\mathcal{L}_r$. It is easy to see that this $m$-tuple of laminations is an admissible lamination for $\widetilde{\mathcal{H}}_\Gamma$.
We denote the quotient by 
$$
\mathcal{AL}(\Gamma) := \widetilde{\mathcal{AL}}(\Gamma)/\Aut(\Gamma).
$$

The automorphism groups $\Aut(\Gamma)$ and $\Aut(\Gamma')$ act naturally on the set of embeddings $i: \Gamma \longrightarrow \Gamma'$ by pre-composition and post-composition.
Let 
$$
N(\Gamma \xhookrightarrow{} \Gamma') := |\Aut(\Gamma) \backslash (\Gamma \xhookrightarrow{} \Gamma') /\Aut(\Gamma')|
$$
be the number of embeddings of $\Gamma$ into $\Gamma'$ up to the action of $\Aut(\Gamma) \times \Aut(\Gamma')$. 
Since each embedding realizes the Julia set $\mathcal{J}(\mathcal{R}_{\Gamma'})$ as a quotient of $\mathcal{J}(\mathcal{R}_{\Gamma})$ (see Proposition \ref{lem:pbtc}), using Proposition \ref{prop:maa}, we have

\begin{prop}\label{prop:scn}
Let $\Gamma' \geq \Gamma$. Then there are $N(\Gamma \xhookrightarrow{} \Gamma')$ inequivalent semiconjugacies between $\mathcal{J}(\mathcal{R}_{\Gamma})$ and $\mathcal{J}(\mathcal{R}_{\Gamma'})$.
\end{prop}

Here two semiconjugacies $\Psi_1$ and $\Psi_2$ are said to be equivalent if they differ by pre and post-composing with automorphisms of $\mathcal{J}(\mathcal{R}_{\Gamma})$ and $\mathcal{J}(\mathcal{R}_{\Gamma'})$.

\subsection{Accesses to $\partial \mathcal{H}_\Gamma(K)$}
The cardinality $N(\Gamma \xhookrightarrow{} \Gamma')$ of the double quotient space is also related to the number of accesses from $\mathcal{H}_\Gamma(K)$ to $\mathcal{H}_{\Gamma'}(K)$.

Let $[R] \in \partial {\mathcal{H}}_\Gamma(K)$. Abusing notation, we choose a lift of $[R]$ in $\partial \widetilde{\mathcal{H}}_\Gamma(K)$ and also denote it by $[R] \in \partial \widetilde{\mathcal{H}}_\Gamma(K)$.
Let $[R_n] \in \widetilde{\mathcal{H}}_\Gamma(K)$ converge to $[R]$.
After passing to a subsequence, let $\mathscr{T}^{En}$ be the enriched Tischler graph for $[R_n]$.
By Theorem \ref{thm:cparm}, $\mathscr{T}^{En}$ is admissible.
Let $\mathcal{L} \in \widetilde{\mathcal{AL}}(\Gamma)$ be the corresponding admissible lamination.
We call it a lamination for $[R]$.

Therefore, we can associate an equivalence class of admissible laminations $[\mathcal{L}] \in \mathcal{AL}(\Gamma)$ for any map $[R] \in \partial \mathcal{H}_\Gamma(K)$.
We remark that by choosing different lifts of $[R]$ in $\partial \widetilde{\mathcal{H}}_\Gamma(K)$, one may obtain multiple laminations associated to a map $[R]  \in \partial \mathcal{H}_\Gamma(K)$, in which case, a {\em self-bump} appears on $\partial \mathcal{H}_\Gamma(K)$.

This motivates the following definition.

\begin{defn}\label{defn:mp}
Let $\mathcal{U}$ be a connected component of $\overline{\mathcal{H}_\Gamma(K)} \cap \overline{\mathcal{H}_{\Gamma'}(K)}$.
We say that the equivalence class of admissible laminations $[\mathcal{L}] \in \mathcal{AL}(\Gamma)$ gives an {\em access} to $\mathcal{U}$ if $[\mathcal{L}]$ is associated to some $[R] \in \mathcal{U}$.
We define the {\em multiplicity} of $\mathcal{U}$ as the number of accesses to $\mathcal{U}$.
\end{defn}

By Proposition~\ref{prop:enriched_tischler_realized} and Theorem \ref{thm:cparm}, we have that
\begin{prop}\label{prop:an}
For all large $K$, there are at least $N(\Gamma \xhookrightarrow{} \Gamma')$ components in $\overline{\mathcal{H}_\Gamma(K)} \cap \overline{\mathcal{H}_{\Gamma'}(K)}$ counting multiplicity.
\end{prop}

We conjecture that the number of components in $\overline{\mathcal{H}_\Gamma(K)} \cap \overline{\mathcal{H}_{\Gamma'}(K)}$ (counted with multiplicities) is exactly $N(\Gamma \xhookrightarrow{} \Gamma')$.
In Appendix \ref{sec:ve}, we shall see how a component of $\overline{\mathcal{H}_\Gamma(K)} \cap \overline{\mathcal{H}_{\Gamma'}(K)}$ with multiplicity $\geq 2$ is related to the self-bump and shared mating phenomena.
We shall also discuss examples where $\overline{\mathcal{H}_\Gamma(K)} \cap \overline{\mathcal{H}_{\Gamma'}(K)}$ is disconnected.

\section{Shared matings, self-bumps and disconnected roots}\label{sec:ve}

In this appendix, we will briefly discuss some examples of how different pared deformation spaces interact.

\subsection{Self-bumps and shared matings}\label{subsec:sbsm}
In this subsection, we will discuss some techniques to construct components of $\overline{\mathcal{H}_\Gamma(K)} \cap \overline{\mathcal{H}_{\Gamma'}(K)}$ with multiplicity at least $2$.
We shall also explain how this is related to the shared mating phenomenon using examples.

\subsection*{An example of self-bump}
Let $\Gamma$ and $\Gamma'$ be the left and right graph in Figure \ref{fig:SB1a}.
The graph $\Gamma$ has two different embeddings into $\Gamma'$.
The images of the embeddings are displayed in black on the right.
Note that both $\Aut(\Gamma)$ and $\Aut(\Gamma')$ are trivial, so the two embeddings are in different orbits of $\Aut(\Gamma) \times \Aut(\Gamma')$.

\begin{figure}[ht]
  \centering
  \resizebox{1\linewidth}{!}{
    \def\svgwidth{\columnwidth}
    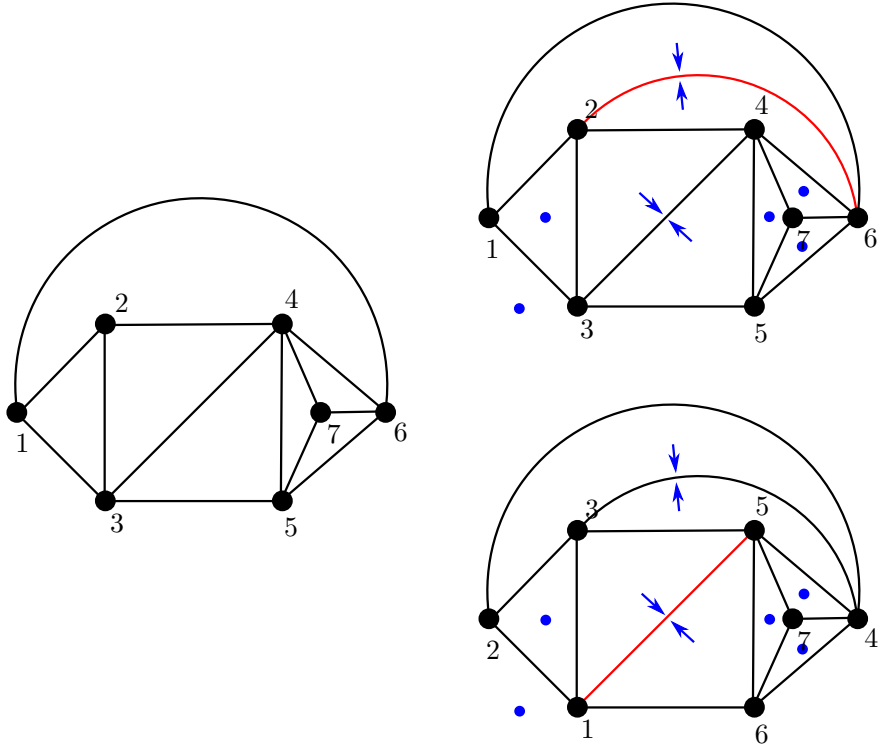

  }
  \caption{The graph $\Gamma$ on the left has two different embeddings into the graph $\Gamma'$ on the right. The images of the embeddings are depicted in black on the right, where the indices of vertices are preserved under the embeddings.}
  \label{fig:SB1a}
\end{figure}

\begin{figure}[ht]
  \centering
  \resizebox{1\linewidth}{!}{
    \def\svgwidth{\columnwidth}
    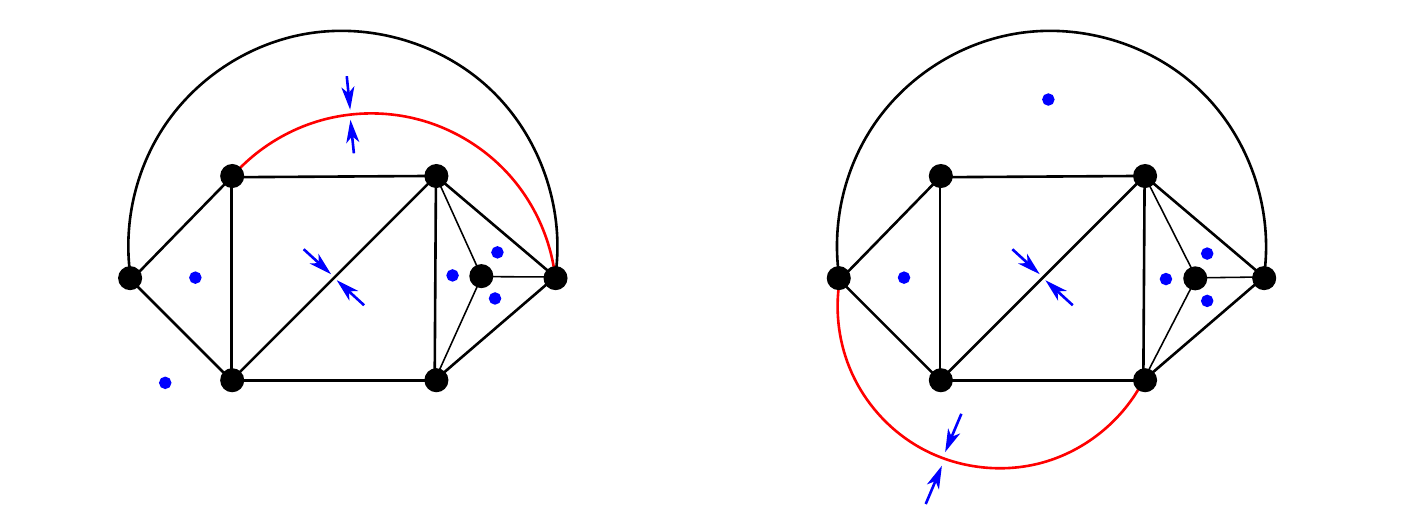

  }
  \caption{The corresponding degeneration diagrams associated to the two embeddings.}
  \label{fig:SB1b}
\end{figure}

We define an {\em arrow structure} on $\Gamma'$ by associating to each face either a dot, or an arrow towards a boundary edge. We call the graph $\Gamma'$ together with an arrow structure an {\em arrowed graph}.

The arrow structure is used to indicate the dynamics of a root $[R] \in \partial \mathcal{H}_{\Gamma'}$.
Note that a face $F'$ of the graph $\Gamma'$ corresponds to a fixed Fatou component $\Omega_{F'}$, and an edge on the boundary of $F'$ corresponds to a fixed point on $\partial \Omega_{F'}$.
For each face $F'$ with $k$ edges, a dot indicates the dynamics of $R$ on $\Omega_{F'}$ is conjugate to $\bar z^k$.
An arrow towards a boundary edge indicates that $\Omega_{F'}$ is a parabolic Fatou component for $R$ with the corresponding fixed point on $\partial \Omega_{F'}$ being the parabolic fixed point.

The two embeddings of $\Gamma$ into the arrowed graph $\Gamma'$ provide two different accesses to $[R]$, as indicated in Figure \ref{fig:SB1b}.
We will call the graphs as in Figure \ref{fig:SB1b} {\em degeneration diagrams}.

We remark that the different perspectives between the arrowed graph and the degeneration diagram (see Figure \ref{fig:SB1a} 
vs Figure \ref{fig:SB1b}, or Figure \ref{fig:SB2}) is that in the former, $\Gamma'$ is fixed, and the graph $\Gamma$ is represented by different subgraphs of $\Gamma'$; while in the degeneration diagram, the original graph $\Gamma$ is fixed, and $\Gamma'$ is constructed from $\Gamma$ by adding new edges.

Note that the faces of $\Gamma$ are divided into {\em chambers}, as indicated by the {\em red boundaries}, in the degeneration diagram.

Let $C, C_0$ be two distinct chambers of a face $F$ of $\Gamma$.
Then there exists a unique boundary edge of $C$ corresponding to the direction of $C_0$.
We say an arrow in $C$ points towards $C_0$ if the arrow points towards this boundary edge.
Similarly, let $E$ be an edge of a chamber in $F$.
Then there exists a unique boundary edge of $C$ corresponding to the direction of $E$.
We say an arrow in $C$ points towards $E$ if the arrow points towards this boundary edge.
See Figure \ref{fig:SB1b} or Figure \ref{fig:SB2} for illustrations.

\begin{defn}\label{def_asc}
An arrow structure on a degeneration diagram is said to be {\em compatible} if for each face $F$ of $\Gamma$, either all the arrows point towards a unique chamber in $F$ containing the dot, or all the arrows point towards a unique edge on the boundary of a chamber in $F$.
\end{defn}

Let $F$ be a face of $\Gamma$ with $e$ sides.
A degeneration diagram partitions the face into finitely many chambers.
The associated dual graph is an $(e+1)$-ended tree and the arrow structure gives either a special core vertex or an edge.
Thus we can associate either a pointed $(e+1)$-ended tree, or an extended pointed $(e+1)$-ended tree for each face $F$ of $\Gamma$. Theorems~\ref{thm:cvr} and~\ref{thm:cvre} and the Schwarz lemma now imply that a prescribed dynamical behavior on the fixed Fatou components (given by an arrow structure on $\Gamma'$) is realized in the dynamical plane of an accumulation point of an access determined by a degeneration diagram if and only if the arrow structure is compatible with the degeneration diagram in the sense of Definition~\ref{def_asc}.

Let $[R_0] \in \partial \mathcal{H}_{\Gamma'}$ be the parabolic map corresponding to the right graph of Figure \ref{fig:SB1a}, so that the dynamics on (each of) the invariant parabolic Fatou components is conjugate to a {\em real} Blaschke product.
This real assumption is required to uniquely specify the map, as unicritical parabolic Blaschke products are {\em not} rigid in the anti-rational map setting (cf. \cite[Theorem~3.2]{MNS17}).

In our example, each face is divided into at most 2 chambers.
Thus, using Proposition \ref{prop:rabl} and Remark \ref{rmk:ext} we can construct two families $[R_{i,t}] \in \mathcal{H}_\Gamma(K)$, $i=1,2$, realizing the accesses determined by the two degeneration diagrams with $[R_{i,t}] \to [R_0]$.
Thus, the component $U$ in $\overline{\mathcal{H}_\Gamma(K)} \cap \overline{\mathcal{H}_{\Gamma'}(K)}$ containing $[R_0]$ has multiplicity $\geq 2$.
One can in fact verify that (cf. \cite[Theorem 1.4]{L21a})
\begin{prop}\label{prop:sb4}
For all large $K$, the pared deformation space $\mathcal{H}_\Gamma(K)$ has a self-bump at $[R_0]$, i.e. for all sufficiently small neighborhood $U \subset \mathcal{M}_d^-$ containing $[R_0]$, the intersection $U \cap \mathcal{H}_\Gamma(K)$ is disconnected.
\end{prop}

\subsection*{A shared parabolic mating example.}
We will now discuss connections between multiple accesses and shared mating phenomenon.

To construct a parabolic map that can be unmated in two different ways, we consider the arrowed graph $\Gamma'$ in Figure \ref{fig:SB2}.
The graph $\Gamma'$ contains two different Hamiltonian cycles that are marked black in Figure \ref{fig:SB2}. It is easy to verify that $\Aut(\Gamma')$ is trivial. Thus, in particular, the two Hamiltonian cycles in Figure \ref{fig:SB2} are in different orbits of $\Aut(\Gamma')$.
Note $\Gamma$ has rotation symmetries, but such a symmetry always preserves the image of $\Gamma$ in $\Gamma'$.
Therefore the corresponding two embeddings of the hexagonal graph $\Gamma$ in $\Gamma'$ are in different orbits of $\Aut(\Gamma) \times \Aut(\Gamma')$.

\begin{figure}[ht]
  \centering
  \resizebox{1\linewidth}{!}{
    \def\svgwidth{\columnwidth}
    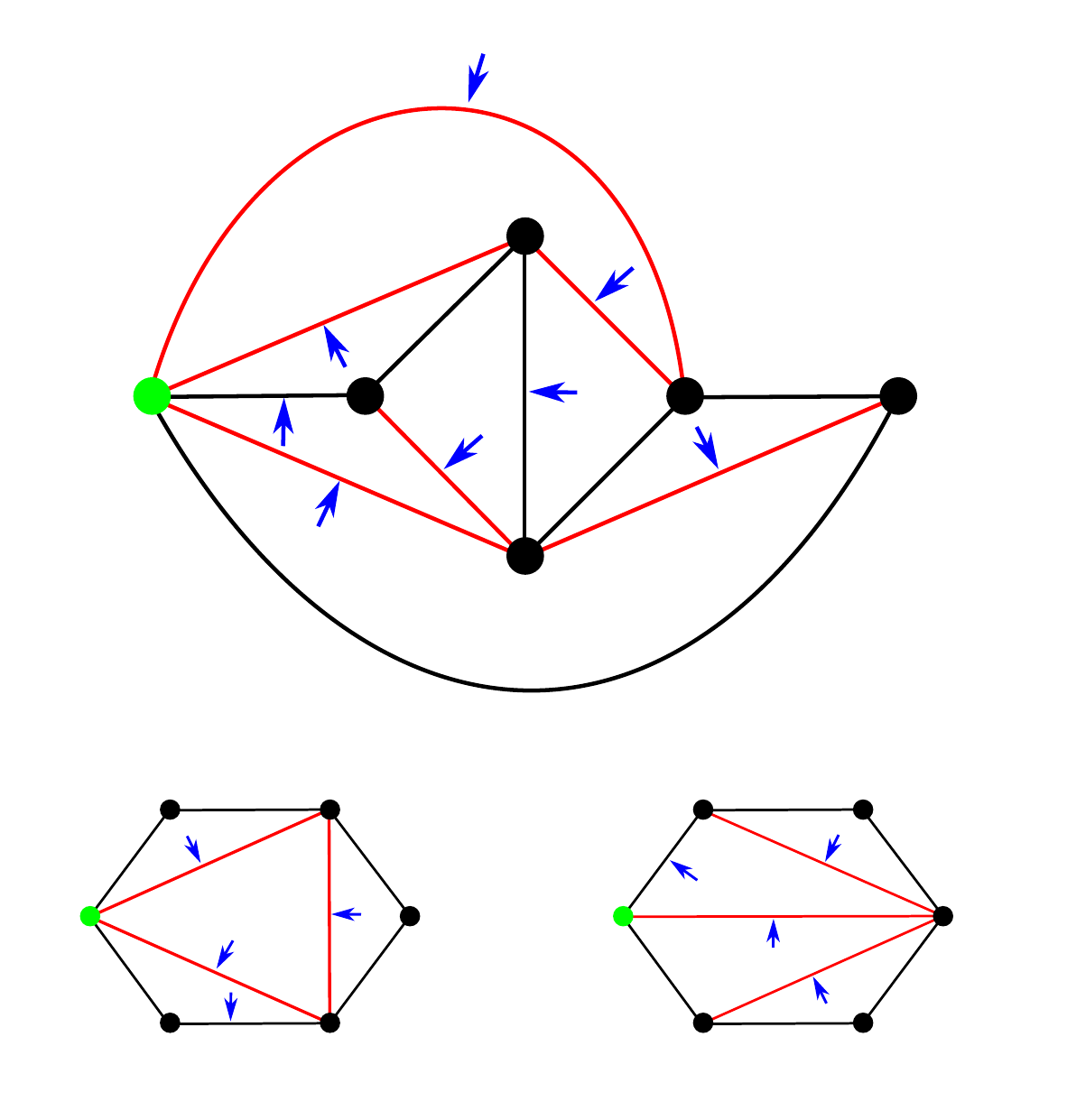

    \def\svgwidth{\columnwidth}
    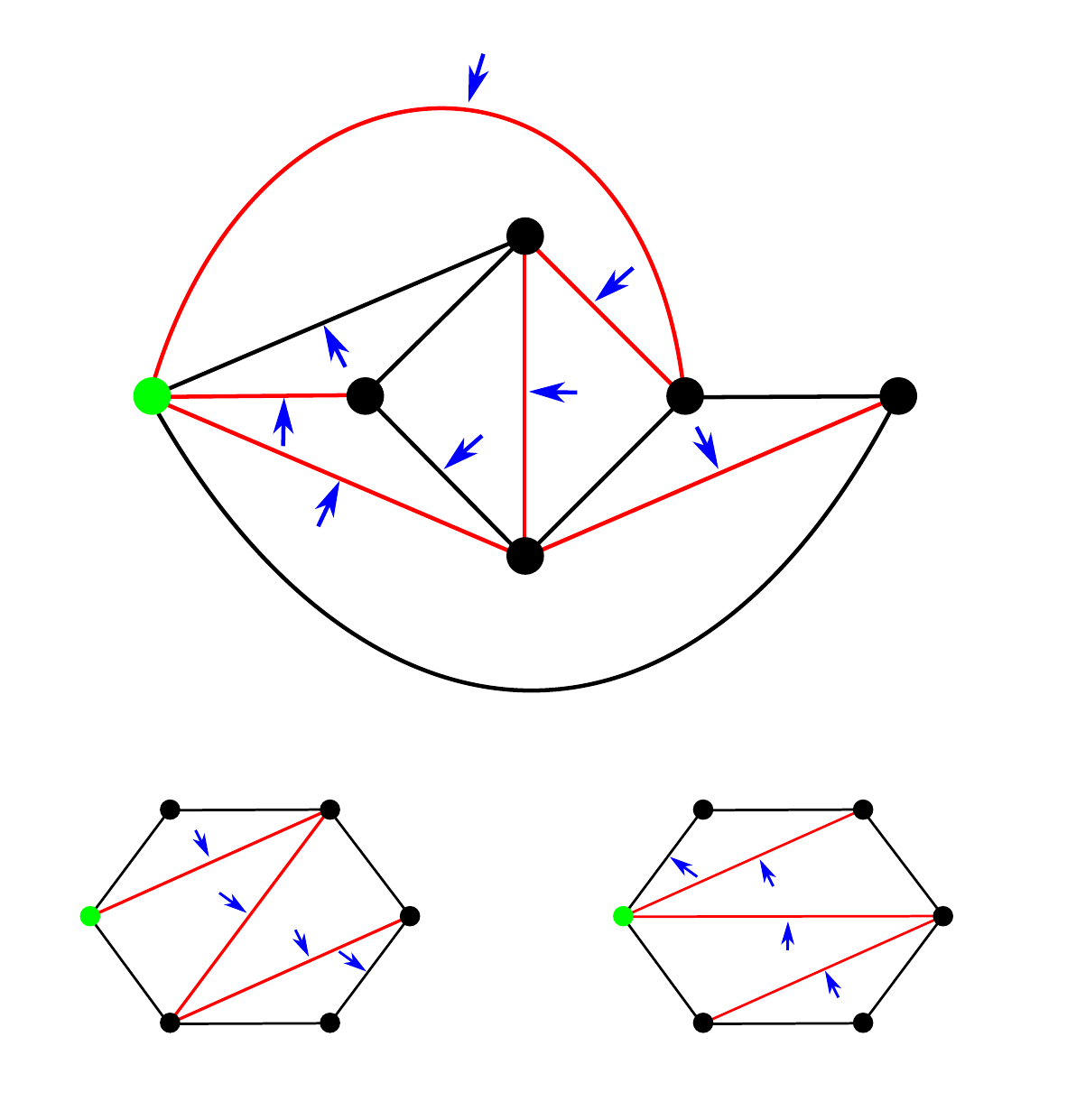

  }
  \caption{The top two graphs exhibit a graph $\Gamma'$ with distinct Hamiltonian cycles drawn using black edges. Each Hamiltonian cycle gives a corresponding decomposition of the parabolic rational map as a mating of parabolic polynomials, exhibited below.}
  \label{fig:SB2}
\end{figure}

The Hamiltonian cycle divides the graph into two sub-graphs, with degeneration diagrams as in the lower part of Figure \ref{fig:SB2}.
Note that the corresponding laminations are different, and the arrow structure on $\Gamma'$ is compatible with both degeneration diagrams.

Theorem~\ref{thm:cvre} and Theorem \ref{thm:cparm} allow us to construct two sequences $[R_{i,n}] \in \mathcal{H}_{\Gamma}(K),\, i=1,2$, realizing the accesses determined by the two degeneration diagrams such that the sequences converge to parabolic maps $[R_i],\, i=1,2$. Note that the parabolic maps $[R_1], [R_2]$ correspond to the same arrowed graph, but we cannot guarantee that $[R_1]=[R_2]$ as unicritical parabolic Blaschke products in the anti-rational map setting are {\em not} rigid. In fact, we do not know how to prove that the two accesses have a common accumulation point.
However, we conjecture that the degeneration diagrams in Figure \ref{fig:SB2} should also produce a self-bump.

On the other hand, the parabolic map $R_1$ is the shared mating of two different pairs of parabolic anti-polynomials where the unmatings are given by the two distinct Hamiltonian cycles in $\Gamma'$.

\subsection{Multiple components in $\overline{\mathcal{H}_\Gamma(K)} \cap \overline{\mathcal{H}_{\Gamma'}(K)}$}
In this subsection, we explain a technique to construct examples where $\overline{\mathcal{H}_\Gamma(K)} \cap \overline{\mathcal{H}_{\Gamma'}(K)}$ is disconnected.

Let $\Gamma'$ be the graph with two different Hamiltonian cycles as in Figure \ref{fig:NSB}.
One can easily verify that
\begin{itemize}
\item $\Aut(\Gamma')$ is trivial, so the corresponding two embeddings are in different orbits; and
\item there are exactly two Hamiltonian cycles in $\Gamma'$ indicated as black in Figure \ref{fig:NSB}.
\end{itemize}

But the difference between this example and the one considered in Section \ref{subsec:sbsm} is that in this example, there is no arrow structure on $\Gamma'$ compatible with both degeneration diagrams.
More precisely, we can not assign an arrow or a dot on each face of $\Gamma'$ so that the arrow structure satisfies the requirements of Definition~\ref{def_asc} for {\em both} degeneration diagrams.
Therefore, the set of accumulation points of the accesses determined by the two degeneration diagrams in Figure \ref{fig:NSB} lie in two distinct components of the root locus $\overline{\mathcal{H}_\Gamma(K)} \cap \overline{\mathcal{H}_{\Gamma'}(K)}$.
So the $\overline{\mathcal{H}_\Gamma(K)} \cap \overline{\mathcal{H}_{\Gamma'}(K)}$ contains at least $2$ components.

This example is related to the shared mating phenomenon.
Although the existence of two different Hamiltonian cycles shows that the critically fixed anti-rational map $[\mathcal{R}_{\Gamma'}]$ is a shared mating of two different pairs of critically fixed anti-polynomials (cf. \cite{LLM22}), no parabolic map $[R] \in \partial \mathcal{H}_{\Gamma'}(K)$ can be unmated in two different ways.

\begin{figure}[ht]
  \centering
  \resizebox{0.8\linewidth}{!}{
    \def\svgwidth{\columnwidth}
    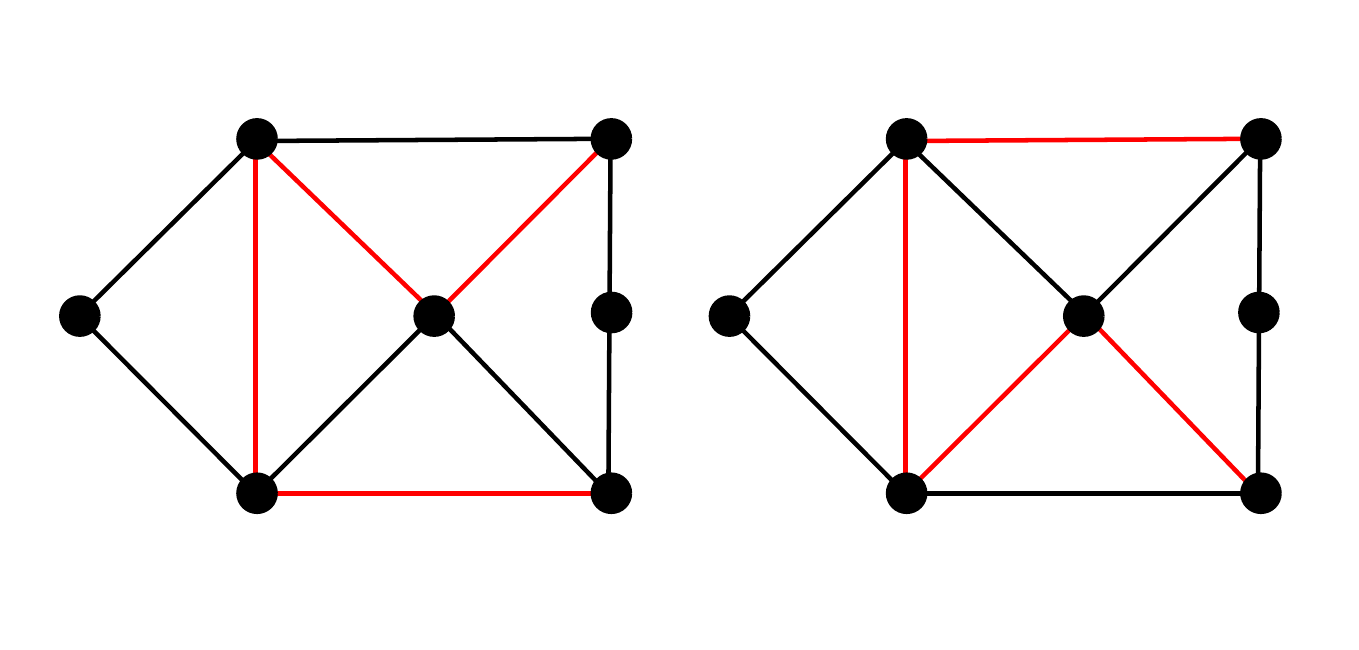

  }
  \caption{An example where no arrow structure on $\Gamma'$ is compatible with two different degeneration diagrams.}
  \label{fig:NSB}
\end{figure}

\end{document}